\documentclass[12pt]{article}

\usepackage{graphicx} 
\usepackage[english]{babel}
\usepackage[letterpaper,margin=1in]{geometry}

\usepackage[algo2e,ruled,vlined]{algorithm2e}
\usepackage[noadjust]{cite}
\usepackage{amsmath, amsfonts, amsthm}
\usepackage{graphicx,soul}
\usepackage[colorlinks=true, allcolors=blue]{hyperref}
\usepackage{amsfonts}
\usepackage{color,xcolor}
\usepackage{multirow}
\usepackage{soul}

\usepackage{amssymb,latexsym}

\usepackage{algorithm}
\usepackage{algorithmicx}
\usepackage{algpseudocode}
\usepackage{mathtools}
\newif\ifPDF
\ifx\pdfoutput\undefined
\PDFfalse
\else
\ifnum\pdfoutput > 0
\PDFtrue
\else
\PDFfalse
\fi
\fi

\ifPDF
\usepackage{pdftricks}
\begin{psinputs}
	\usepackage{pstricks}
	\usepackage{pstcol}
	\usepackage{pst-plot}
	\usepackage{pst-tree}
	\usepackage{pst-eps}
	\usepackage{multido}
	\usepackage{pst-node}
	\usepackage{pst-eps}
\end{psinputs}
\else
\usepackage{pstricks}
\fi

\usepackage[openbib]{currvita}

\usepackage{fancyhdr}

\newtheorem{theorem}{Theorem}[section]
\newtheorem{lemma}[theorem]{Lemma}

\newtheorem{proposition}[theorem]{Proposition} 
\newtheorem{remark}[theorem]{Remark} 
\newtheorem{corollary}[theorem]{Corollary}

\usepackage{nicematrix}

\newcommand{\bbR}{\mathbb R} \newcommand{\bbS}{\mathbb S}

\newcommand{\bzeta}{\boldsymbol \zeta}

\newcommand{\bLambda}{\boldsymbol \Lambda}

\newcommand{\bSigma}{{\boldsymbol\Sigma}}

\newcommand{\ba}{\mathbf a} \newcommand{\bb}{\mathbf b}
\newcommand{\bc}{\mathbf c} 
\newcommand{\be}{\mathbf e}

\newcommand{\bu}{\mathbf u} \newcommand{\bv}{\mathbf v} 
\newcommand{\bw}{\mathbf w} \newcommand{\bx}{\mathbf x} 
\newcommand{\by}{\mathbf y} \newcommand{\bz}{\mathbf z} 
\newcommand{\bA}{\mathbf A} \newcommand{\bB}{\mathbf B}

\newcommand{\bG}{\mathbf G} 
\newcommand{\bI}{\mathbf I} 
\newcommand{\bK}{\mathbf K} 
 
 \newcommand{\bP}{\mathbf P} 
 \newcommand{\bR}{\mathbf R}
 
\newcommand{\bU}{\mathbf U} \newcommand{\bV}{\mathbf V}
\newcommand{\bW}{\mathbf W}

\newcommand{\cG}{\mathcal G} 
\newcommand{\cI}{\mathcal I} 
\newcommand{\cK}{\mathcal K} \newcommand{\cL}{\mathcal L}
 \newcommand{\cN}{\mathcal N}
\newcommand{\cO}{\mathcal O}  
 
\newcommand{\cS}{\mathcal S} 
\newcommand{\cU}{\mathcal U}

\makeatletter
\newcommand{\chapterauthor}[1]{%
	{\parindent0pt\vspace*{-25pt}%
		\linespread{1.1}\large\scshape#1%
		\par\nobreak\vspace*{35pt}}
	\@afterheading%
}
\makeatother

\usepackage{algorithm2e}
\usepackage{booktabs}
\RestyleAlgo{ruled}
\SetKwComment{Comment}{/* }{ */}

\newcommand{\ReLU}{\text{ReLU}}
\newcommand{\PINN}{\text{PINN}}
\newcommand{\FEMsp}{\text{FEMsp}}
\newcommand{\DRM}{\text{DRM}}
\usepackage{authblk}

\usepackage[table]{xcolor}
\title{What Can One Expect When Solving PDEs Using Shallow Neural Networks?}

\author[1]{Roy Y. He}
\author[2]{Ying Liang}
\author[2]{Hongkai Zhao}
\author[3]{Yimin Zhong}

\affil[1]{Department of Mathematics, City University of Hong Kong, Kowloon Tong, Hong Kong}
\affil[2]{Department of Mathematics, Duke University, Durham, NC, USA}
\affil[3]{Department of Mathematics and Statistics, Auburn University, Auburn, AL, US}
	
\date{}

\begin{document}
\maketitle
\begin{abstract}
We use elliptic partial differential equations (PDEs) as examples to show various properties and behaviors when shallow neural networks (SNNs) are used to represent the solutions. In particular, we study the numerical ill-conditioning, frequency bias,  and the balance between the differential operator and the shallow network representation for different formulations of the PDEs and with various activation functions. Our study shows that the performance of Physics-Informed Neural Networks (PINNs) or Deep Ritz Method (DRM) using linear SNNs with power ReLU activation is dominated by their inherent ill-conditioning and spectral bias against high frequencies. Although this can be alleviated by using non-homogeneous activation functions with proper scaling, achieving such adaptivity for nonlinear SNNs remains costly due to ill-conditioning.  
\end{abstract}

\section{Introduction}
Neural networks (NNs) provide a special form of nonlinear parametrization for approximating functions. There are extensive results in terms of mathematical approximation theory for NNs. One of the main themes is showing universal approximation property (UAP) of NNs~\cite{cybenko1989approximations,barron2002universal,devore2021neural,kidger2020universal}. The general principle is that, under certain conditions, NNs can approximate functions in specific classes to arbitrary accuracy. Recently, more concrete approximation error estimates in terms of NNs' architectural parameters, i.e., depth, width, and height, as well as various activation functions, have been obtained~\cite{lu2017expressive, eldan2016power,shen2019deep, zhang2022deep}. However, these existing mathematical results do not address the key practical challenges: how accurately and at what computational cost a solution can be found by solving a high-dimensional, non-convex optimization problem, which is typically the computational bottleneck. For example, it was shown in \cite{zhang2025shallow} that, for two-layer NNs, the strong correlation among randomly parametrized global activation functions leads to ill-conditioning and strong frequency bias against high frequencies in both representation and learning. Consequently, it is difficult to train a two-layer neural network to accurately approximate a function with very high-frequency components. 

Naturally, NNs can be used to represent or approximate the solutions of PDEs. This approach has been explored in numerous works, such as Physics-Informed Neural Networks (PINNs)~\cite{raissi2019physics,raissi2020hidden}, the Deep Ritz Method (DRM)~\cite{weinan2018deep},  Weak Adversarial Networks (WAN)~\cite{zang2020weak}, and  others~\cite{lu2021learning,li2020fourier,de2024wpinns}. Significant effort has been devoted to the theoretical justification of these methods. For instance, the consistency of PINNs for linear second-order elliptic and parabolic PDEs was established in~\cite{shin2020convergence}, and their generalization error was analyzed in~\cite{mishra2023estimates}. Analogous results exist for the DRM~\cite{muller2019deep, duan2022convergence, lu2021priori, muller2022error}. The general conclusion from these studies is that, given sufficient data and appropriate handling of the boundary conditions, the minimizer of the associated optimization problem converges to the PDE solution. The NN-based methods for solving PDEs can potentially offer significant advantages, such as being mesh-free, simple to implement, scalable to high dimensions~\cite{hinton2006reducing}, and inherently adaptive~\cite{suzuki2018adaptivity}.

However, this nonlinear parameterization complicates the formulation of the optimization problem itself, specifically in how to enforce the differential equation along with its boundary and initial conditions.  This introduces a significant computational hurdle: solving a large-scale, non-convex optimization problem. In practice, this is most often done using gradient-based methods, which are highly sensitive to ill-conditioning~\cite{saarinen1993ill}. 
Consequently, practitioners frequently rely on a suite of heuristic techniques~\cite{wang2023expert}. This performance gap has attracted significant attention in recent studies~\cite{cuomo2022scientific,krishnapriyan2021characterizing,basir2022investigating,song2024does,cao2025analysis}. For instance, Krishnapriyan et al.~\cite{krishnapriyan2021characterizing} observed that PDE-based differential operators can introduce ill-conditioning, while Basir~\cite{basir2022investigating} investigated how gradients contaminated by high-order differential operators impede training. Collectively, this body of numerical evidence highlights critical computational challenges concerning the accuracy and efficiency of NN-based PDE solvers. A crucial computational question arises: what can we expect regarding the accuracy, stability, convergence, and computational cost when using NNs to solve PDEs?

In this work, we present an in-depth analysis and computational study on solving elliptic PDEs using shallow neural networks (SNNs), i.e., two-layer NNs. Using spectral analysis, we demonstrate the competition between the inherent ill-conditioning and frequency bias for high frequencies of a differential operator and the ill-conditioning and frequency bias against high frequencies of an SNN representation for PINN and DRM. We show that the frequency bias towards low frequency, which is a phenomenon also known as the frequency principle~\cite{rahaman2019spectral,xu2024overview}, leads to different levels of ill-conditioning for various activation functions.
As a consequence, when gradient-based training is used, the smooth component of the solution is captured relatively quickly. In contrast, the high-frequency component will be recovered slowly if possible (depending on the computation cost and machine precision or noise level). Hence, SNNs struggle to capture high frequency and achieve high accuracy when the PDE solution contains significantly high-frequency components.

Note that this is in contrast to the behavior of traditional FEMs when iterative methods are used to solve the discretized linear system without preconditioning, in which case the corresponding differential operator induces the frequency bias. At the same time, the FEM representation itself has no frequency bias. Effective preconditioning techniques, such as multi-grids~\cite{zhu2006multigrid} and domain decomposition~\cite{cowsar1995balancing}, have been well developed for FEMs.

We focus on the ill-conditioning in PINNs and the DRM, which is directly caused by the spectral bias of SNNs. First, we conduct the aforementioned studies on linear models, which are closely related to the random basis method. In this well-controlled scenario, we analyze the spectrum of Gram matrices associated with power ReLU activation functions and study both direct solvers and iterative gradient-based methods. We also examine the performance when boundary conditions are enforced as either constraints or regularizations. Through this comprehensive study, we demonstrate that the ill-conditioning and frequency bias of SNN representations lead to unstable reconstruction and an ineffective training process. We then extend our study to other activation functions. We find that scaling the random initialization can alleviate frequency bias and significantly improve approximation accuracy. The benefits of such scaling have been noted in prior works on function approximation~\cite{zhang2025shallow,sitzmann2020implicit}. However, the scaling strategy can only achieve the potential of linear random basis methods because the SNN representation still cannot capture frequencies or features beyond what the current random basis can resolve, and hence cannot lead to effective training of adaptivity, which is crucial for a nonlinear representation. 

On the one hand, multi-layer NNs can utilize combinations and compositions of smooth functions to approximate complex functions effectively, as demonstrated in \cite{zhang2025shallow,zhang2024structured}, and can potentially significantly reduce, if not eliminate, the frequency bias of SNNs. On the other hand, one still has to deal with the ill-conditioning and frequency bias induced by the underlying differential operators and overcome the challenges of solving a high-dimensional non-convex optimization problem when using multi-layer NNs to solve PDEs. We will report our investigation of this interesting and challenging problem in future work.

To summarize, our contributions in this work are as follows:

\begin{itemize}
\item An explicit analysis of the ill-conditioning and frequency bias of both PINN and DRM using SNNs with power ReLUs as activation functions. 
\item Characterization of the numerical stability of the direct solvers and training dynamics of iterative gradient-based methods through the frequency bias of the SNNs.

\item A comprehensive set of numerical experiments conducted under carefully controlled conditions. We meticulously examine PINN and DRM formulations using both boundary constraints and regularization techniques, solved directly as linear systems or iteratively with first-order optimization methods.
 \end{itemize}

This paper is organized as follows. In Section~\ref{sec_related}, we review the most relevant literature. Section~\ref{sec_NN_poisson} describes the experimental setups we focus on, specifically investigating the solution's accuracy of PINNs and the DRM using linear SNNs with power ReLU activations for the Poisson equation with Dirichlet boundary conditions. We analyze various settings and optimization strategies. Section~\ref{sec_other_activate} extends the study numerically to explore the spectral properties of neural networks with other activation functions, and we discuss additional challenges that arise when all parameters are learnable. Finally, we present our conclusions in Section~\ref{sec_conclude}.
 
\section{Related Works}\label{sec_related}
We begin with well-developed classical methods that utilize representations linear in their respective parameters or unknowns. In particular, by reviewing the most popular and successful approach, Galerkin methods, such as finite element methods (FEMs), for second-order elliptic PDEs, we discuss some of the computational challenges for solving PDEs, especially in terms of the mathematical formulation and efficient solvers.  More importantly, we highlight the effects of ill-conditioning and frequency bias on iterative methods, including gradient descent optimization, which sheds light on our subsequent investigations into PINN and DRM. We also review recent theoretical and computational works on using NNs to solve PDEs.

\subsection{Classical methods}
Galerkin-type methods, such as FEMs, are the most popular and successful approaches in solving elliptic PDEs in general domains~\cite{ciarlet2002finite} in three and lower dimensions. They rely on a weak formulation equivalent to the original PDE in an appropriate function space. Such a flexible framework is not only the mathematical foundation for error estimates and convergence analysis but also crucial for numerical computations. By discretizing the domain with a mesh and introducing a finite dimensional linear space with compactly supported basis functions to approximate the underlying solution space, one obtains a linear system for the unknown parameters (coefficients) of the linear representation. Due to the localized basis functions, the linear system consists of a sparse matrix, called the \textit{stiffness matrix}. 

The stiffness matrix is ill-conditioned, which is inherited from and consistent with the differential operator~\cite{bertaccini2018iterative, bramble1990parallel, mandel1993balancing}. For example, using the standard piecewise linear finite element basis, while there is no frequency bias in the representation, i.e., the condition number of the Gram (mass) matrix of the basis is of $\cO(1)$, a second-order differential operator, such as the Laplacian, amplifies each of its eigenmodes by its frequency squared.  This leads to the ill-conditioning of the stiffness matrix of the discretized linear system with a condition number of order $\cO(h^{-2})$, where $h$ is the mesh size and $h^{-1}$ is the order of the highest frequency that the mesh can resolve. To solve the resulting linear sparse system, iterative methods are usually the natural choices; however, ill-conditioning often leads to their slower convergence (\cite[Section 4.2]{saad2003iterative}) if without preconditioning.

Moreover, the ill-conditioning inherited from the differential operator introduces a frequency bias against low frequencies. Specifically, when using an iterative solver, the errors in high-frequency modes decay rapidly, while the errors in low-frequency modes decay slowly. Consequently, naive iterative methods are inefficient even for recovering smooth solutions, rendering them impractical to use in real applications. To address this, effective preconditioning techniques~\cite{bertaccini2018iterative},  multigrid~\cite{bramble1990parallel},  and domain decomposition~\cite{mandel1993balancing} have been well developed to solve elliptic PDEs efficiently. 

Noticeably, ill-conditioning becomes even more severe if a strong formulation is adopted when using finite element basis functions to solve a second-order elliptic PDE. This approach not only requires a higher-order finite element basis to ensure more regularity of the basis functions, which is more complex to construct, but also leads to a discrete linear system with a condition number of $\mathcal{O}(h^{-4})$, yielding an even stronger bias against low frequencies.

The main challenges of traditional Galerkin methods include the requirement for a good-quality mesh that conforms to the domain geometry~\cite{cottrell2009isogeometric}, the construction of basis functions~\cite{vzenivsek1970interpolation,zhang2009family}, and the computational complexity that grows exponentially with the problem's dimension~\cite{bellman1959mathematical}. Although adaptive mesh techniques can improve efficiency by locally refining regions with rapid solution changes guided by \textit{a posteriori} error estimates, they still require the initial mesh to resolve these local details to some degree. Furthermore, the complexity of this adaptive approach grows quickly as the dimensionality increases. 

We remark that Galerkin methods, along with their associated mathematical theory, algorithms, and commercial software implementations, have undergone significant development over the past half-century. Consequently, they remain the benchmark for comparison in problems of three dimensions or fewer.

\subsection{Using neural network representation to solve PDEs}

Neural networks (NNs), as a class of parametrized, nonlinear representations of functions, require no mesh, fit to high dimensions, and, most importantly, have the potential for adaptivity. Furthermore, Universal Approximation Theorems (UATs) guarantee their ability to approximate functions in various spaces, including smooth and analytic functions~\cite{mhaskar1996neural}, Sobolev spaces~\cite{guhring2020error}, and Besov spaces~\cite{siegel2023optimal}. This strong theoretical foundation makes NNs a natural ansatz for solving PDEs~\cite{raissi2019physics,raissi2020hidden,weinan2018deep,zang2020weak}. These representations are later confirmed to be consistent and convergent for broad classes of PDEs under various settings~\cite{shin2020convergence,mishra2023estimates,luo2024two,shin2023error,he2018relu,duan2021convergence, xu2020finite,kharazmi2019variational}. 
However, realizing this potential is challenging due to the associated nonlinear and nonconvex optimization problems. Achieving both accuracy and efficiency in this context thus constitutes a major mathematical and computational challenge~\cite{xu2025understanding}.

Recently, a significant body of computational and theoretical efforts has emerged on NN-based PDE solutions~\cite{kuvakin2025weak,chen2022bridging,raissi2019physics,raissi2020hidden,weinan2018deep,lu2021learning,li2020fourier,de2024wpinns,zang2020weak}. Nevertheless, in comparison to well-developed traditional methods for practical 2D and 3D problems, the NN representations used in many computational studies are typically inferior in terms of both accuracy and efficiency~\cite{chen2024optimization}. This is often compounded by a reliance on ad-hoc parameter tuning. Consequently, the theoretical results on convergence or error estimates frequently fail to address the critical computational challenges. These analyses often rely on existence arguments, which assume the successful discovery of a global minimizer for a large-scale, non-convex optimization problem, or they operate in the limit of infinite width or the Neural Tangent Kernel (NTK) regime~\cite{jacot2018neural, luo2024two}. These theoretical regimes are far from practical implementations and overlook the severe ill-conditioning inherent in the representations within these regimes, which renders effective optimization impossible.

When NNs are used as a nonlinear parametric representation to solve PDEs, two fundamental questions must be addressed: (1) the \textbf{formulation} of an appropriate optimization problem, i.e., how to enforce the differential equation along with its boundary and initial conditions; and (2) the \textbf{optimization} itself, which involves solving a large-scale, non-convex problem for which mainly gradient-based methods are computationally affordable. The interplay between the above two factors has manifested through different lenses and been reported in many works~\cite{cuomo2022scientific,krishnapriyan2021characterizing,basir2022investigating,song2024does,cao2025analysis}. For instance, Krishnapriyan et al.~\cite{krishnapriyan2021characterizing} observed that the PINN's loss function's landscape becomes increasingly complex when strong boundary regularization is forced. They demonstrated that enhancing the expressivity of NNs does not amend the optimization difficulty. Song et al.~\cite{song2024does} showed that higher PDE order causes slower convergence of gradient flow. 

Indeed, ill-conditioning can pose a severe computational challenge in practice, 
even for quadratic convex optimization problems solved either directly or iteratively. We demonstrate that this issue does happen for both PINN and DRM formulations using shallow neural networks (SNNs) during the training process. Along this direction, two of the most relevant works are~\cite{hong2022activation} and~\cite{zhang2025shallow}. Hong et al.~\cite{hong2022activation} analyzed the eigenvalue structures of Gram matrices associated with SNNs using ReLU activation and revealed their connection to the spectral bias of SNNs on evenly distributed bias in one dimension. Zhang et al.~\cite{zhang2025shallow} analyzed the ill-conditioning in general situations and for powers of ReLU activation in any dimensions and studied the implications of the decaying spectrum on numerical accuracy and learning dynamics.  Unlike these works, we present a spectral analysis of the Gram matrices (also known as stiffness matrices in FEM terminology) resulting from PINN and DRM using SNNs, and study their numerical performance in solving second-order elliptic equations.

We note that this paper is not a trivial generalization of the aforementioned works.  Ill-conditioning can arise from both the NN representation and the differential operator, which is more complicated than function approximation. Interestingly, the frequency bias from a neural network's inherent ill-conditioning can counteract the frequency bias introduced by the differential operator. Yu et al.~\cite{yu2023tuning} showed that replacing the $L^2$ norm in the loss function for function approximation with an appropriate Sobolev norm can amplify or reverse the neural network’s inherent frequency bias during training. The differential operator in the loss function for solving PDEs plays a similar role. This competition between the high-frequency bias of differential operators and low-frequency bias from network representations  leads to a few \textit{central questions}: what is the overall frequency bias, and how does it impact the accuracy, stability, convergence, and cost of solving PDEs with NNs? We investigate these properties in the context of second-order elliptic PDEs solved with PINNs and DRM using SNNs. 

\section{$\ReLU^p$ Neural Networks for Elliptic PDEs}\label{sec_NN_poisson}
In this paper, we focus on the study of solving the following type of second-order elliptic equations with Dirichlet boundary conditions using two-layer NNs: 
\begin{equation}\label{eq_PDE_problem}
	\begin{cases}
		Lu=f&\text{in}~\Omega\\
		u = g&\text{on}~\partial\Omega
	\end{cases}\;,
\end{equation}
where $\Omega\subset\mathbb{R}^d$ is a bounded domain with sufficiently smooth boundary. In~\eqref{eq_PDE_problem}, the function $u$ is the unknown, $f\in L^{\infty}(\Omega),g \in L^{\infty}(\partial\Omega)$ are  given, and  the differential operator can be written in a divergence form
\begin{equation}\label{eq_div_operator}
	Lu := -\sum_{i,j=1}^d\partial_j\left(a^{ij}(x)\partial_iu\right)  + c(x)u\;.
\end{equation}
We assume that the  coefficients are bounded, $a^{ij}=a^{ji}$ are symmetric, $c\geq 0$, and $L$ is uniformly elliptic, i.e., there exists a constant $\theta>0$ such that
\begin{equation}
	\sum_{i,j=1}^da^{ij}(x)\zeta_i\zeta_j\geq \theta|\bzeta|^2
\end{equation}
for almost every $x\in \Omega$ and all $\bzeta = (\zeta_1,\dots,\zeta_d)\in\mathbb{R}^d$. With continuous $f$ and $g$, a function $u\in C^2(\Omega)\cap C(\overline{\Omega})$ that satisfies~\eqref{eq_PDE_problem} is called its \textit{classical solution}. 

One can also define a weak solution using Sobolev spaces. Denote $T:H^1(\Omega)\to L^2(\partial\Omega)$ as the trace operator (See e.g.~\cite{evans2022partial} Section 5.5), then $u\in H_g^1(\Omega)=\{h \in H^1(\Omega)~:~Th = g\}$  is called a \textit{weak solution} for~\eqref{eq_PDE_problem} if it satisfies 
\begin{equation}\label{eq_PDE_weak}
	B[u,v] = \langle f, v\rangle\;, ~\text{for any}~v\in H_0^1(\Omega)\;.
\end{equation}
Here the bilinear mapping 
\begin{equation}\label{eq:var}
	B[u,v] :=\sum_{i,j=1}^d\int_{\Omega} a^{ij}(x) \partial_i u\partial_j v\,dx + \int_{\Omega} c(x) uv\,dx\;,
\end{equation}
and $\langle \cdot,\cdot\rangle$ is the pairing of $H^{-1}(\Omega)$ and $H_0^1(\Omega)$. With calculus of variations,  it can be shown that~\eqref{eq_PDE_weak} is equivalent to minimizing
\begin{equation}
	J(u) = \frac{1}{2}B[u,u] -\int_{\Omega} fu\,dx
\end{equation}
with the constraint that $u\in H_g^1(\Omega)$.
We refer the interested readers to~\cite{nirenberg1974topics,gilbarg1977elliptic,evans2022partial} for more details. 

\subsection{PINN and DRM with two-layer networks}
For an integer $N\geq 1$, we consider the $N$-width shallow neural network (SNN) of the following form: 
\begin{equation}\label{eq_two_layer_NN}
u(\bx,\ba,\bw,\bb) = \sum_{i=1}^N a_i\sigma(\bw_i^\top \bx -b_i)
\end{equation}
for approximating the PDE solution. Here $\ba = (a_1,\dots,a_N)\in \mathbb{R}^N$, $\bw_i\in\mathbb{R}^d$, $i=1,\dots,N$, and $\bb = (b_1,\dots,b_N)\in \mathbb{R}^N$  are the model parameters,  and $\sigma:\mathbb{R}\to\mathbb{R}$ denotes the activation function acting component-wise.   
For PINNs, the primary objective is to minimize the inconsistency of the NN approximation with the strong form~\eqref{eq_PDE_problem}. This inconsistency is measured by the loss term:
\begin{equation}\label{eq_PINN_in_loss}
\cL^{\text{in}}_{\PINN}(\ba,\bw,\bb) := \int_{\Omega}(Lu(\bx, \ba)-f(\bx))^2\,d\bx\;.
\end{equation}
In contrast, the DRM aims to address the variational form~\eqref{eq_PDE_weak}. The corresponding loss functional is derived directly from this form and is given by:
\begin{equation}\label{eq_DRM_in_loss}
\cL^{\text{in}}_{\DRM}(\ba,\bw,\bb) := \frac{1}{2}B[u,u] - \int_{\Omega}f(\bx)u(\bx)\,d\bx\;.
\end{equation}

While minimizing the error~\eqref{eq_PINN_in_loss} or the energy~\eqref{eq_DRM_in_loss} for the domain inconsistency, we consider two strategies of incorporating the Dirichlet boundary conditions. Denoting $F\in \{\PINN, \DRM\}$, the formulation with boundary constraints is written as:
\begin{equation}\label{eq_constrained_optimization}
\begin{aligned}
\min_{\ba,\bw,\bb}~&\cL_{F}^{\text{in}}(\ba,\bw,\bb)\\
&
u(\bx,\ba,\bw,\bb) = g(\bb)\;,~\text{for all}~\bx\in\partial\Omega
\end{aligned}\;,
\end{equation} 
and the formulation with boundary regularization is:
\begin{equation}\label{eq_regularized_optimization}
\min_{\ba,\bw,\bb}~\cL_{F}^{\text{in}}(\ba,\bw,\bb) + \lambda \int_{\partial\Omega} (u(\bx,\ba,\bw,\bb)-g(\bx))^2\,d\bx\;,
\end{equation} 
with some nonnegative regularization parameter $\lambda>0$. The unconstrained optimization form~\eqref{eq_regularized_optimization} is more common, but selecting an appropriate $\lambda$ can be challenging~\cite{wang2023expert}. 

Note that the network representation~\eqref{eq_two_layer_NN} is linear in the parameters $\ba$ but nonlinear in $\bw$ and $\bb$. The parameters $\bw$ and $\bb$ contribute to the network's adaptivity. To understand the training dynamics of a gradient-based approach, it is important to understand the representation with fixed $\bw$ and $\bb$.  This analysis reveals the best-case capacity of the current representation for capturing the underlying function at a given optimization step.  Furthermore, it clarifies whether the representation can provide the necessary information to effectively optimize $\bw$ and $\bb$ to achieve the desired adaptivity. For this purpose, we  study the solution of the one-dimensional Poisson equation (for simplicity) using the representation from~\eqref{eq_two_layer_NN} with \textit{fixed} $\bw$ and $\bb$, that is, the first layer is frozen. This allows us to omit the notations $\bw$ and $\bb$ from~\eqref{eq_two_layer_NN}-\eqref{eq_DRM_in_loss}, and we shall simply set $\Omega=(-1,1)$ without loss of generality. As results, for fixed $\bw,\bb$ and $\lambda$,~\eqref{eq_constrained_optimization} and~\eqref{eq_regularized_optimization} reduce to
\begin{equation}\label{eq_constrained_optimization_simple}
\begin{aligned}
\min_{\ba\in\mathbb{R}^N}~&\cL_{F}^{\text{in}}(\ba)\\
&
\begin{cases}
u(-1,\ba) = g(-1)\\
u(1,\ba) = g(1)
\end{cases}
\end{aligned}\;,
\end{equation} 
and 
\begin{equation}\label{eq_regularized_optimization_simple}
\min_{\ba\in\mathbb{R}^N}~\cL_{F}^{\text{in}}(\ba) + 
\lambda\left((u(-1,\ba)-g(-1))^2+(u(1,\ba)-g(1))^2\right)\;,
\end{equation} 
respectively; here $\cL_{\PINN}^{\text{in}}(\ba):=\frac{1}{2}\int_{-1}^1(\partial_x^2u(x, \ba)+f(x))^2\,dx$ and 
$\cL_{\DRM}^{\text{in}}(\ba):=\frac{1}{2}\int_{-1}^1|\partial_x u(x,\ba)|^2\,dx-\int_{-1}^1f(x)u(x,\ba)\,dx$ are also simplified. Thereby, we have a linear representation: the network function $u(\cdot,\ba):\mathbb{R}\to\mathbb{R}$ is a linear combination of a family of basis functions $\sigma(w_nx-b_n)$ defined by the given parameters $w_n, b_n\in\mathbb{R}$ for $n=1,\dots,N$. In particular,~\eqref{eq_constrained_optimization_simple} and~\eqref{eq_regularized_optimization_simple} become 
quadratic optimization problems, which are equivalent to solving linear systems involving the Gram matrix (or the stiffness matrix in finite element terminology). 
 
For any nonegative integer $k$, define $\bG^{(k)}_{\sigma}\in\mathbb{R}^{N\times N}$ to be the Gram matrix associated with the $k$-th derivative of $\sigma$ whose entries are given by
\begin{equation}\label{eq_Gram_matrix}
[\bG^{(k)}_{\sigma}]_{i,j} = \int_{-1}^1w_i^kw_j^k\sigma^{(k)}( w_ix - b_i)\sigma^{(k)}( w_jx - b_j)\,dx\;,\quad  i,j=1,\dots,N\;,
\end{equation}
and define a vector $\by_{\sigma}^{(k)}$ whose entries are
\begin{equation}\label{eq_Gram_vector}
[\by^{(k)}_\sigma]_i = \int_{-1}^1 w^k_i\sigma^{(k)}(w_ix-b_i) f(x)\,dx\;, i=1,\dots,N\;.
\end{equation}
With the above setup, we have the following observation.

\begin{proposition} Using two-layer NNs~\eqref{eq_two_layer_NN} with fixed $\bw$ and $\bb$, both PINN and DRM with boundary constraints~\eqref{eq_constrained_optimization_simple} are equivalent to solving a quadratic minimization problem with linear constraints of the form
\begin{equation}\label{eq_constrained_quad}
\begin{aligned}
\min_{\ba\in\mathbb{R}^N}&~\frac{1}{2}\ba^\top\bG_F\ba - \ba^\top\by_F\;,~F\in\{\PINN, \DRM\}\\
\text{s.t.}&~\bB\ba = \bc
\end{aligned}\;.
\end{equation}
And both PINN and DRM with boundary regularization~\eqref{eq_regularized_optimization_simple} are equivalent to solving
\begin{equation}\label{eq_regularization_quad}
\min_{\ba\in\mathbb{R}^N}~\frac{1}{2}\ba^\top\bG_F\ba -\ba^\top\by_{F}+ \frac{\lambda}{2}\|\bB\ba-\bc\|^2\;,~F\in\{\PINN, \DRM\}\;.
\end{equation}
 In both~\eqref{eq_constrained_quad} and~\eqref{eq_regularization_quad}, $\bc^\top=(g(-1),g(1))\in\mathbb{R}^2$,
\begin{equation}\label{eq_B_matrix}
\bB = \begin{bmatrix}
\sigma(-w_1-b_1)&\sigma(-w_2-b_2)&\cdots&\sigma(-w_N-b_N)\\
\sigma(w_1-b_1)&\sigma(w_2-b_2)&\cdots&\sigma(w_N-b_N)\\
\end{bmatrix}\in\mathbb{R}^{2\times N}\;,
\end{equation}
and we have $\bG_{\PINN} = \bG_\sigma^{(2)}$ and $\by_{\PINN} = -\by_{\sigma}^{(2)}$; $\bG_{\DRM}=\bG_{\sigma}^{(1)}$ and $\by_{\DRM}=\by^{(0)}_{\sigma}$.
\end{proposition}

We note that the quadratic minimization form~\eqref{eq_regularization_quad} for PINN was considered in~\cite{luo2024two}. The constrained optimization~\eqref{eq_constrained_optimization_simple} is rarely used in practice, but it is a useful reference where no extra hyper-parameter is needed.  Moreover, we always assume that $\bG_F$ is nonsingular, or equivalently, the functions $\{w_i^k\sigma^{(k)}(w_ix-b_i)\}_{i=1}^N$ are linearly independent for $k=1$ and $k=2$. In this case,  both problems~\eqref{eq_constrained_quad} and~\eqref{eq_regularization_quad} admit unique solutions.

\begin{figure}
\centering
\begin{tabular}{ccc}
(a)&(b)&(c)\\
\includegraphics[width=0.28\textwidth]{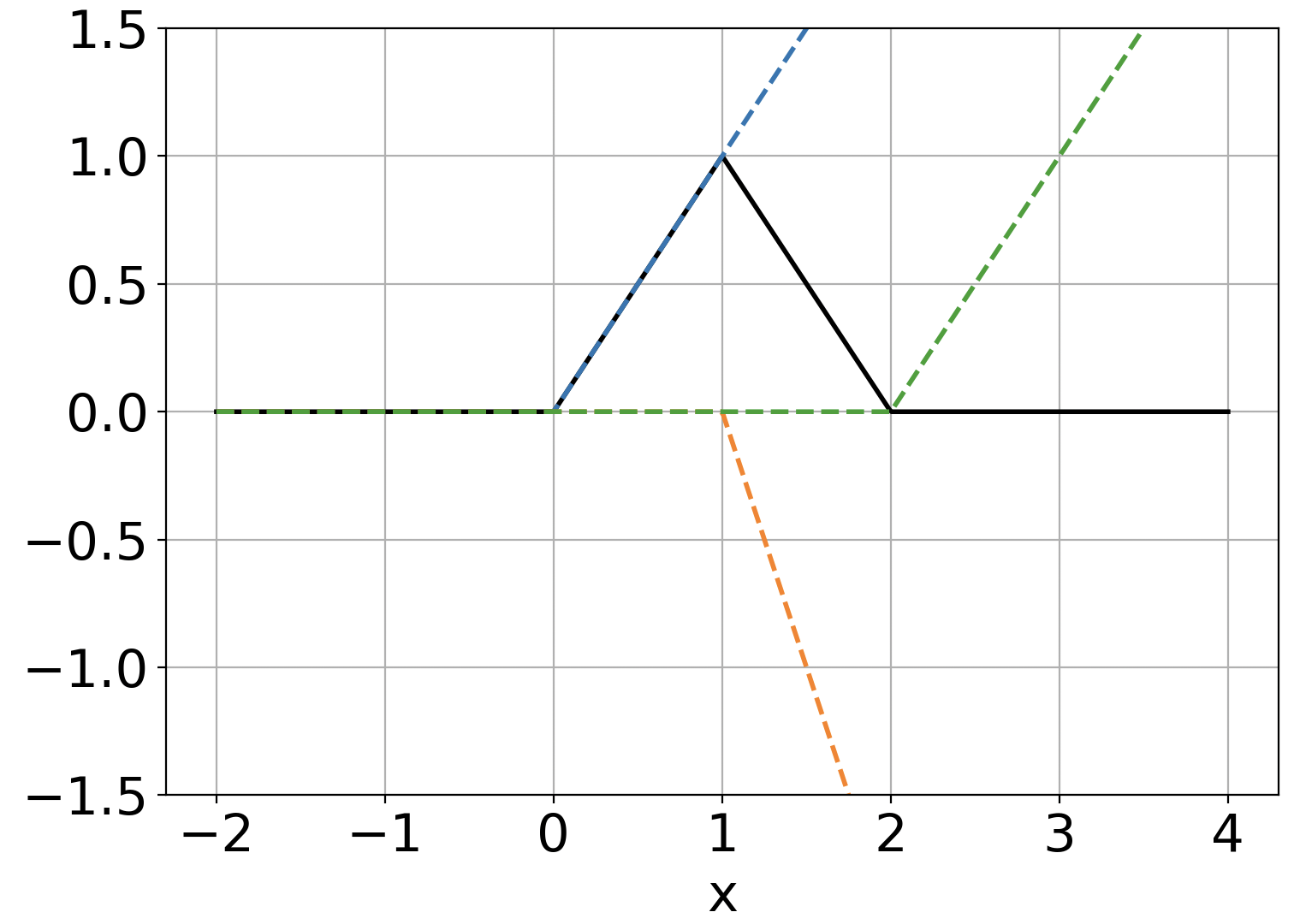}&
\includegraphics[width=0.28\textwidth]{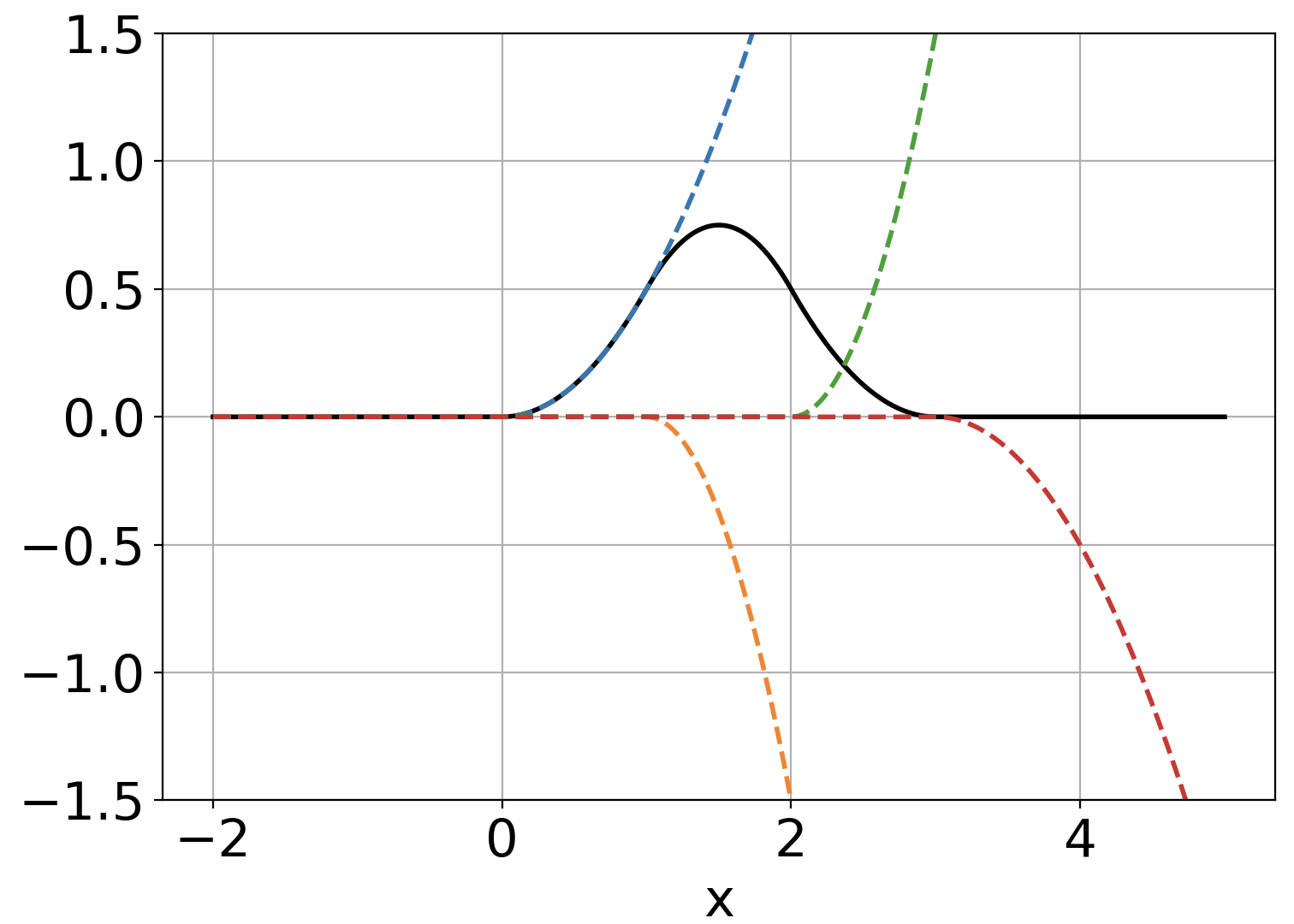}&
\includegraphics[width=0.28\textwidth]{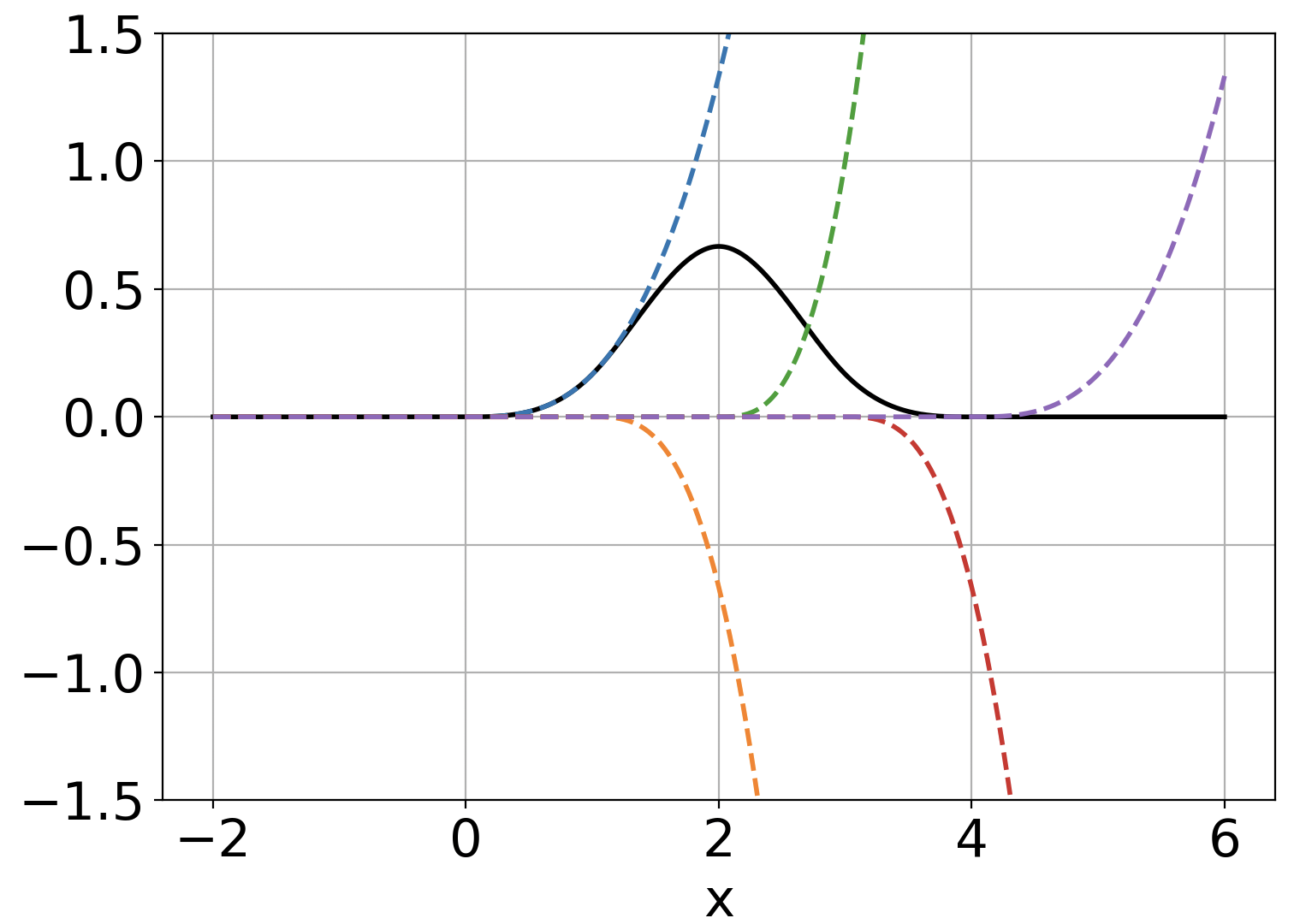}
\end{tabular}
\caption{Composition of B-spline finite element basis of order $p+1$ using linear combination of ReLU power-$p$ bases. (a) $p=1$, (b) $p=2$, (c) $p=3$. }\label{fig_spline_using_relu}
\end{figure}

\subsection{Shallow NNs with $\ReLU^p$ activation}
We first study the case where  $\sigma$ is the $\ReLU^p$ activation function defined by $z\mapsto (\max\{0,z\})^p$ for some integer $p\geq 1$. To satisfy the differentiability requirements in the respective formulations, we assume $p\geq 2$ for PINN and  $p\geq 1$ for DRM. Such architectures have been considered in~\cite{mcculloch1943logical,hornik1989multilayer,he2020relu, gao2023gradient,li2017convergence,zhang2025shallow,williams2019gradient}. When the components of $\bb$ form a regular grid on $(-1,1)$, we note that $\ReLU^p$  functions are closely related to the cardinal B-splines of degree $p$.  A degree $p$ cardinal B-spline supported on $[0,p+1]$ can be derived recursively~\cite{de1972calculating}  via
\begin{equation}
B^{p}(x) = \frac{x}{p} B^{p-1}(x) + \frac{p+1-x}{p} B^{p-1}(x-1)\;,
\end{equation}
and $B^0(x) = 1$ when $x\in [0,1]$  and $0$ otherwise. More explicitly, we have
\begin{equation}\label{eq_Bspline_ReLU}
B^{p}(x) = \frac{1}{p!}\sum_{k=0}^{p+1}(-1)^k {p+1\choose k} \ReLU^p(x-k)\;,\quad\text{for any }x\in\mathbb{R}\;.
\end{equation}
See Figure~\ref{fig_spline_using_relu} for illustrations. The linear combination of $\ReLU^p$ in~\eqref{eq_Bspline_ReLU} is clearly $0$ for $x\leq 0$. Since the coefficients correspond to $(p+1)$-th finite difference, we have  
\begin{equation}
\sum_{k=0}^{p+1}(-1)^k {p+1\choose k}(x-k)^p = 0\;,
\end{equation}
for any $x$, which ensures that~\eqref{eq_Bspline_ReLU} is $0$  whenever $x\geq p+1$. Note that with $\Delta x = 2/(N-p)$, the set
\begin{equation}
\cS_N^p=\left\{ B^p\left(\frac{x-1+\ell\Delta x}{\Delta x}\right)\;, \ell = -p, -p+1,\dots, N-p-1\right\}
\end{equation}
forms a basis for the spline space of degree $p$, i.e., piecewise polynomial functions with degree $p$ over $(-1,1)$, and it is a basic component in B-spline FEM~\cite{hollig2003finite}. Combined with~\eqref{eq_Bspline_ReLU}, we have the following observation.
\begin{proposition}\label{prop_refined}
Define $\bb = (b_1,\dots,b_N)\in\mathbb{R}^N$ with $b_i = -1+(i-p)\Delta x$ and $\Delta x = 2/(N-p)$, then
\begin{equation}
\operatorname{span}_{\mathbb{R}}\, \cS_N^p = \operatorname{span}_{\mathbb{R}}\,\{\ReLU^p(x-b_i), i=1,\dots,N\}\;.
\end{equation}
\end{proposition}
This result shows that the space of shallow $\ReLU^p$ NNs with weights $\bw=(1,\dots,1)\in\mathbb{R}^N$ and regular biases is exactly the space of splines of degree $p$.   In~\cite{xu2020finite} Lemma 3.2, Xu showed a general inclusion relation between the spline space of degree $p$ and the vector space spanned by power-ReLU with arbitrary weights and biases. Proposition~\ref{prop_refined} gives an explicit correspondence. This naturally leads us to define
\begin{equation}\label{eq_relate_G_sigma_G_FEM}
\bG_{\FEMsp}^{(k)} = \bW^\top\bG^{(k)}_{\sigma}\bW\;,\quad\by^{(k)}_{\FEMsp} = \bW^\top\by^{(k)}_{\sigma}
\end{equation}
for $\sigma = \ReLU^p$ and $k=0,1,\dots$, where $\bW\in\mathbb{R}^{N\times N}$  is a transformation matrix whose $(i,j)$-th entry is given by
\begin{equation}\label{eq_def_W}
[\bW]_{i,j} =  \begin{cases}
\frac{(-1)^{i-j}}{p!(\Delta x)^p}{p+1\choose i-j}& j\leq i \leq j+p+1 \\
0&\text{otherwise}
\end{cases}\;,
\end{equation}
and we can consider both~\eqref{eq_constrained_optimization} and~\eqref{eq_regularized_optimization} with $F=\FEMsp$.

\subsection{Eigenvalue structure of the Gram matrix}

Although Proposition~\ref{prop_refined} claims an equivalence relation between $\ReLU^p$ activation functions and cardinal B-splines of degree $p$, they induce different spectral properties of the associated Gram matrices. Note that $\bW$ defined in~\eqref{eq_def_W} is not an orthogonal matrix, thus $\bG^{(k)}_{\FEMsp}$ is merely congruent to $\bG_\sigma^{(k)}$.

To characterize the \textbf{asymptotic spectrum} of $\bG_{F}^{(k)}$ for $F\in\{\PINN, \DRM\}$,  we define the kernel function $\cG_p: [-1,1]\times[-1,1]\to\mathbb{R}$
\begin{equation}
\cG_p(x,y):=\int_{-1}^1\ReLU^p( z-x)\cdot \ReLU^p( z-y)\,dz
\end{equation}
for $\sigma=\ReLU^p$ with integer $p\geq 1$, 
where the weight parameter in the first layer is omitted~\footnote{We have $\ReLU^p(wx-y) = w^p\ReLU^p(x-y/w)$ for any $x,y$ and $w\neq 0$.}, and the associated  compact operator: 
\begin{equation}\label{eq_kernel}
\cK_p[h](x) = \int_{-1}^1\cG_p(x,y) h(y)\,dy
\end{equation}
acting on any $h\in L^2([-1,1])$.  Denote $\mu_{k,p}>0$ as the $k$-th eigenvalue of $\cG_p$ in the descending order and $\varphi_{k,p}$ the corresponding eigenfunction such that $\cK_p[\varphi_{k,p}] = \mu_{k,p}\varphi_{k,p}$. We start with the following observation. 

\begin{lemma}\label{lemma_vanishing_derivative} For $p\geq 1$, we have
\begin{equation}
\frac{\partial^{2p+2}}{\partial x^{2p+2}}\cG_p(x,y) = (-1)^{p-1}(p!)^2\delta( x- y)\;,
\end{equation}
where $\delta$ is the Dirac distribution.
\end{lemma}
\begin{proof}
See Appendix~\ref{proof_lemma_vanishing_derivative}.
\end{proof}
The above lemma implies that the eigenfunction $\varphi_k$ of $\cK_p$ satisfies a linear differential equation 
\begin{equation}\label{eq_eigen_ODE}
\mu_{k,p}\varphi_{k,p}^{(2p+2)}(x) = (-1)^{p-1}(p!)^2\varphi_{k,p}(x)\;.
\end{equation}
From this, we deduce the following characterization of the eigenvalue structure of the  compact operator~\eqref{eq_kernel}.

\begin{proposition}\label{prop_decaying} Let $\mu_{1,p}>\mu_{2,p}>\cdots$ be the eigenvalues of $\cK_p$ and $\varphi_{k,p}$ the corresponding eigenfunctions.  For any $p\geq 1$,  as $k\to+\infty$, we have
\begin{equation}
\mu_{k,p} = \cO(k^{-(2p+2)})\;.
\end{equation}

\end{proposition}
\begin{proof}
The characteristic equation for~\eqref{eq_eigen_ODE} is  $r^{2p+2} = \mu_{k,p}^{-1}(p!)^2 e^{(p-1)\pi i}$, and the roots are given by
\begin{equation}
r_n=\left(\mu_{k,p}^{-1}(p!)^2\right)^{\frac{1}{2p+2}}e^{\frac{(2n
+p-1)\pi i}{2p+2}}\;,~n=0,1,\dots,2p+1\;.
\end{equation}
Denoting $\rho_{k,p}:=\left(\mu_{k,p}^{-1}(p!)^2\right)^{\frac{1}{2p+2}}$, the general solution of~\eqref{eq_eigen_ODE} takes the form
\begin{equation}\label{eq_general_eigen}
\sum_{n=0}^{p}e^{a_{n,k} x}(A_{n,k}\cos(\omega_{n,k} x) + B_{n,k}\sin(\omega_{n,k} x))\;,
\end{equation}
where $a_{n,k} =\rho_{k,p}\cos \frac{(2n+p-1)\pi}{2p+2}$,  $\omega_{n,k} = \rho_{k,p}\sin \frac{(2n+p-1)\pi}{2p+2}$ and the undetermined constants $A_{n,k},B_{n,k}$ satisfy $A_{n,k}^2+B_{n,k}^2\neq 0$, for  $n=0,1,\dots, p$. From the proof of Lemma~\ref{lemma_vanishing_derivative}, we see that 
\begin{equation}\label{eq_wronskian}
\begin{cases}\frac{d^\ell}{d x^{\ell}}\varphi_{k,p}(1) = 0\;, \ell=0,1,\dots,p\\
\frac{d^\ell}{d x^{\ell}}\varphi_{k,p}(-1) = 0\;, \ell=p+1,p+2,\dots,2p+1\;,
\end{cases}
\end{equation}
holds for any $k$, which  completely determines $A_{n,k}$ and $B_{n,k}$ for $n=0,1,\dots,p$ up to a constant scale.  By the Birckhoff's method and Stone's estimation (See~\cite{birkhoff1908boundary} and \cite{stone1926comparison}), we have $\rho_{k,p}=\cO(k)$ as $k\to\infty$, which implies the statement. 
\end{proof}
Since $\cK_p$ is a compact Hermitian operator, the above spectral analysis of the kernel operator~\eqref{eq_kernel} allows us to infer the asymptotic spectrum of the Gram matrix $\bG_{\sigma}^{(k)}$ defined in~\eqref{eq_Gram_matrix}. Minor modifications of Theorem 1 in~\cite{zhang2025shallow} yield the following result. 
\begin{theorem}\label{theorem_bound}
 Let 
$\lambda_1^{(k)} \geq \lambda_2^{(k)} \geq \cdots \geq \lambda_N^{(k)}\geq 0$ be the eigenvalues of the Gram matrix $\bG_{\sigma}^{(k)}\in\mathbb{R}^{N\times N}$ with $\sigma=\ReLU^p$ and $p > k$. For $n=1,2,\dots,N$,  assume that $w_n=\pm 1$ and $\bb=(b_1,\dots,b_N)\in\mathbb{R}^N$  is quasi-evenly spaced on $[-1,1]$, i.e., $b_n=-1+2(n-1)/N+o(N^{-1})$. Then as $N\to\infty$, we have
\begin{equation}
|\lambda^{(k)}_n- \frac{N}{2}\mu_{n,p-k}|\leq C
\end{equation}
 for some constant $C>0$, where $\mu_{n,p-k}$ is the $n$-th eigenvalue of $\cK_{p-k}$.
\end{theorem}

Combined with Proposition~\ref{prop_refined},  Theorem~\ref{theorem_bound} implies that as $N$ becomes sufficiently large, the $n$-th eigenvalue of the scaled Gram matrix $N^{-1}\bG_{\sigma}^{(k)}$ decays as fast as $\cO(n^{-(2(p-k)+2)})$. When the power $p$ grows, the activation function $\ReLU^p\in C^{p-1}(\mathbb{R})$ becomes smoother, and the spectrum decays faster. Consequently, we expect that:
\begin{enumerate}
\item When the power ReLU is used as the activation function in a two-layer neural network, the larger the power (i.e., the smoother the bases), the faster the spectral decay. This implies worse conditioning and stronger frequency bias against high frequencies.

\item The maximum number of frequency modes that a two-layer neural network can accurately and stably recover is determined by the spectral decay rate and machine precision (or noise level). For example, with a machine of finite precision $10^{-\kappa}$, at most $\mathcal{O}(10^{\kappa/(2(p-k)+2)})$ eigenvalues can be computed. For IEEE 754 double-precision floating-point arithmetic ($\kappa \approx 16$), this yields $\mathcal{O}(10^{8/(p-k+1)})$. In addition, as the regularity $p$ increases, the number of recoverable modes decreases. Note that, as the network's width increases, the added eigenmodes of the Gram matrix are of increasingly higher frequency and are associated with ever-smaller eigenvalues. Once the network's width exceeds a certain threshold, finite machine precision prevents the computation of new eigenmodes, even as the width grows arbitrarily large. This computational limit is typically disregarded in theoretical studies.    
\end{enumerate}

Another key factor that determines numerical stability is the \textbf{condition number} of the Gram matrix. A large condition number indicates that even minor errors from finite-precision arithmetic or numerical discretization can be significantly amplified, resulting in substantial deviations in the solution. Furthermore, it compromises the convergence of iterative optimization methods~\cite{caponnetto2007optimal,lin2020optimal}. The following result demonstrates the severity of this issue for a set of general quasi-evenly spaced biases.

\begin{theorem}\label{theorem_condition_number1}
 For $n=1,2,\dots,N$,  assume that $w_n=\pm 1$ and $\bb=(b_1,\dots,b_N)\in\mathbb{R}^N$  is quasi-evenly spaced on $[-1,1]$. With $\sigma=\ReLU^p$ any integer $0\leq k < p$, the condition number $\kappa(\bG_{\sigma}^{(k)}):=\lambda_{1}^{(k)}/\lambda^{(k)}_{N}$ satisfies
\begin{equation}
\kappa(\bG_{\sigma}^{(k)})=\Omega(N^{1+2(p-k)})\;,
\end{equation}
or equivalently, there exists some constant $C$ such that $\kappa(\bG_{\sigma}^{(k)})\geq CN^{1+2(p-k)}$.
\end{theorem}
\begin{proof} See Appendix~\ref{proof_theorem_condition_number1}.
\end{proof}

With evenly spaced $\bb$, i.e., $b_i= -1+2N^{-1}(n-1)$, $i=1,2,\dots,N$, we can have a sharper estimate of the spectrum.
\begin{theorem}\label{theorem_sharp_eigenvalue}
Suppose $\sigma=\ReLU^p$ and integer $0\leq k < p$. Let $\lambda_1^{(k)}\geq \cdots\geq \lambda_{N}^{(k)}$ be the eigenvalues of $\bG_\sigma^{(k)}$. The following 
\begin{equation}
\frac{\lambda_N^{(k)}}{\lambda_j^{(k)}}\sim \frac{N^{2+2(p-k)}}{j^{2+2(p-k)}}\;.
\end{equation}
holds as $N\to+\infty$. Consequently, the condition number 
\begin{equation}
\kappa(\bG_{\sigma}^{(k)}) = \Theta(N^{2+2(p-k)})\;,
\end{equation}
that is, there exist constant $c>0$, $C>0$, and an integer $N_0$ such that  $cN^{2+2(p-k)}\leq \kappa(\bG_{\sigma}^{(k)})\leq CN^{2+2(p-k)}$ holds for any $N\geq N_0$.
\end{theorem}
\begin{proof}
See Appendix~\ref{proof_theorem_sharp_eigenvalue}.
\end{proof}

In the following, we conduct numerical experiments to validate the above spectral analysis and discuss its implications. As there are only constant factor differences among $\bG_{\sigma}^{(k)}$ with different orders of derivative $k$, we will focus on $\bG_p:=\bG_{\sigma}^{(0)}$ with $\sigma=\ReLU^p$ for $p\geq 1$. We compute the entries of $\bG_p$, i.e., integrals of products of $\ReLU^p$ functions, using exact formulas collected in Appendix~\ref{sec_explicit_formula}. To perform the eigen-decomposition of $\bG_p$, we use QR-based function \texttt{eigh} from \texttt{Python} package \texttt{scipy}. This algorithm is backward stable, and the absolute errors of the computed eigenvalues are proportional to the float-point arithmetic accuracy $\epsilon_M$; using \texttt{float64}, i.e., double precision, which is approximately $1\times 10^{-16}$.

\begin{figure}
\begin{center}
\begin{tabular}{c@{\vspace{2pt}}c@{\vspace{2pt}}c}
(a)&(b)&(c)\\
\includegraphics[width=0.33\textwidth]{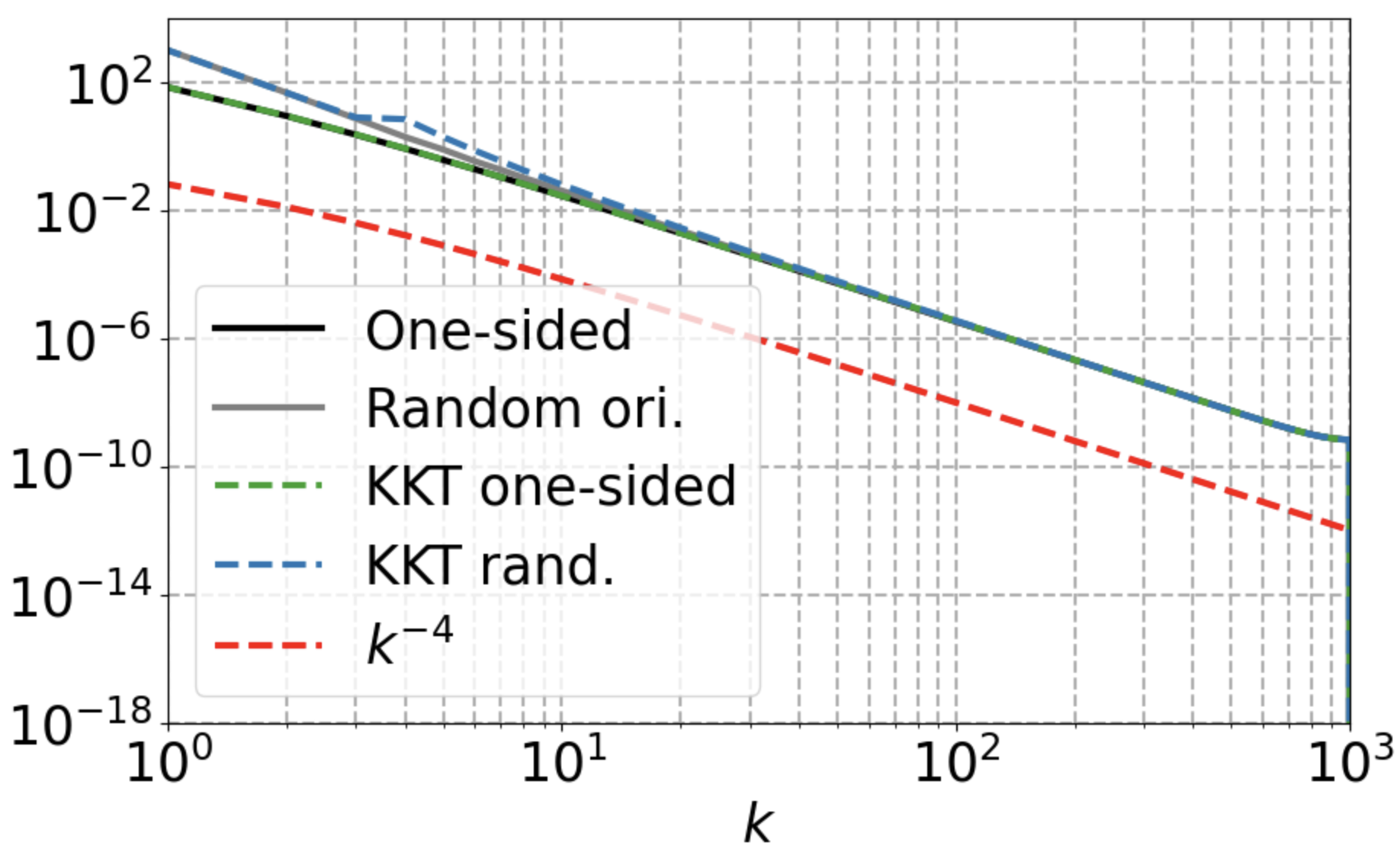}&
\includegraphics[width=0.33\textwidth]{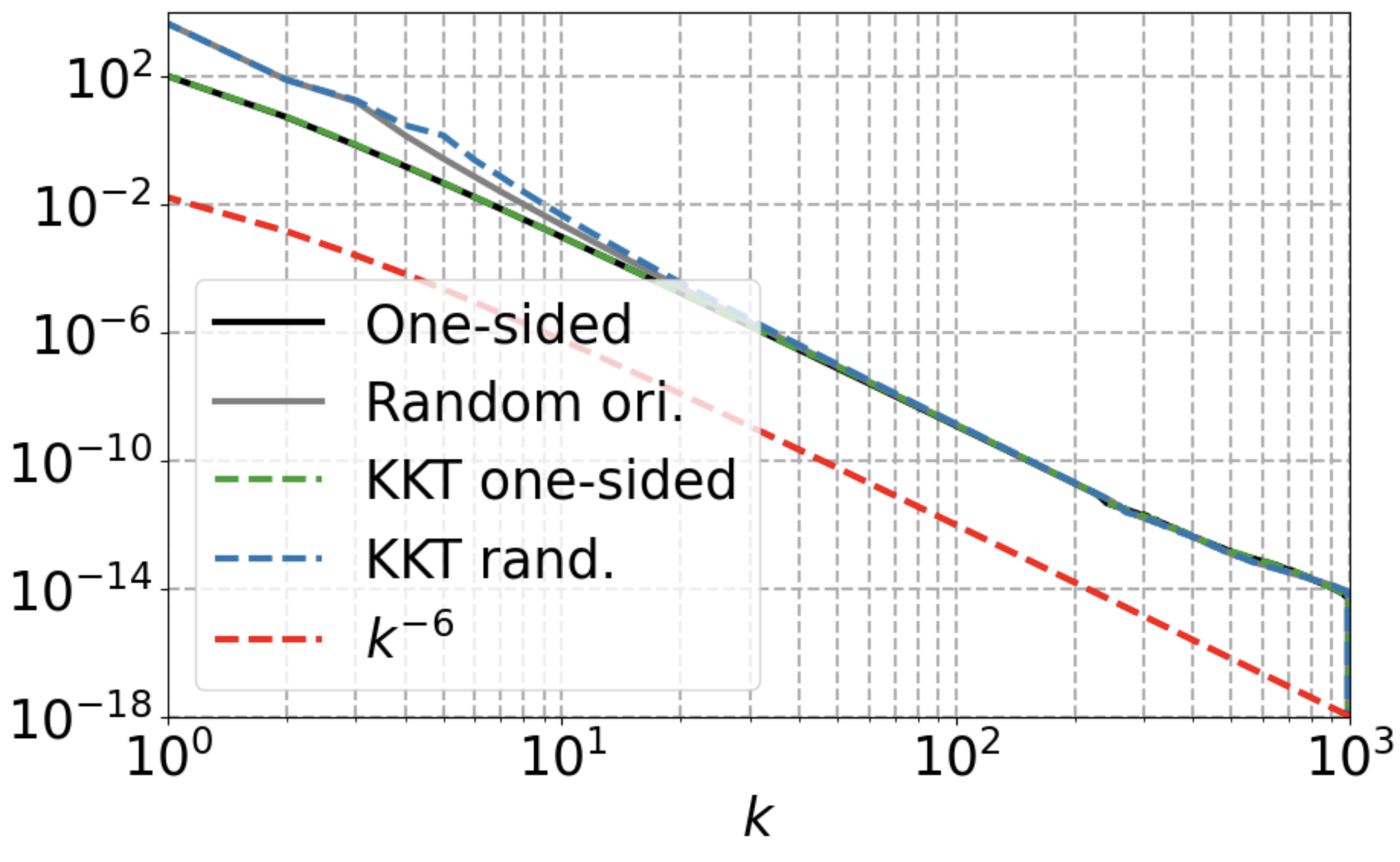}&
\includegraphics[width=0.33\textwidth]{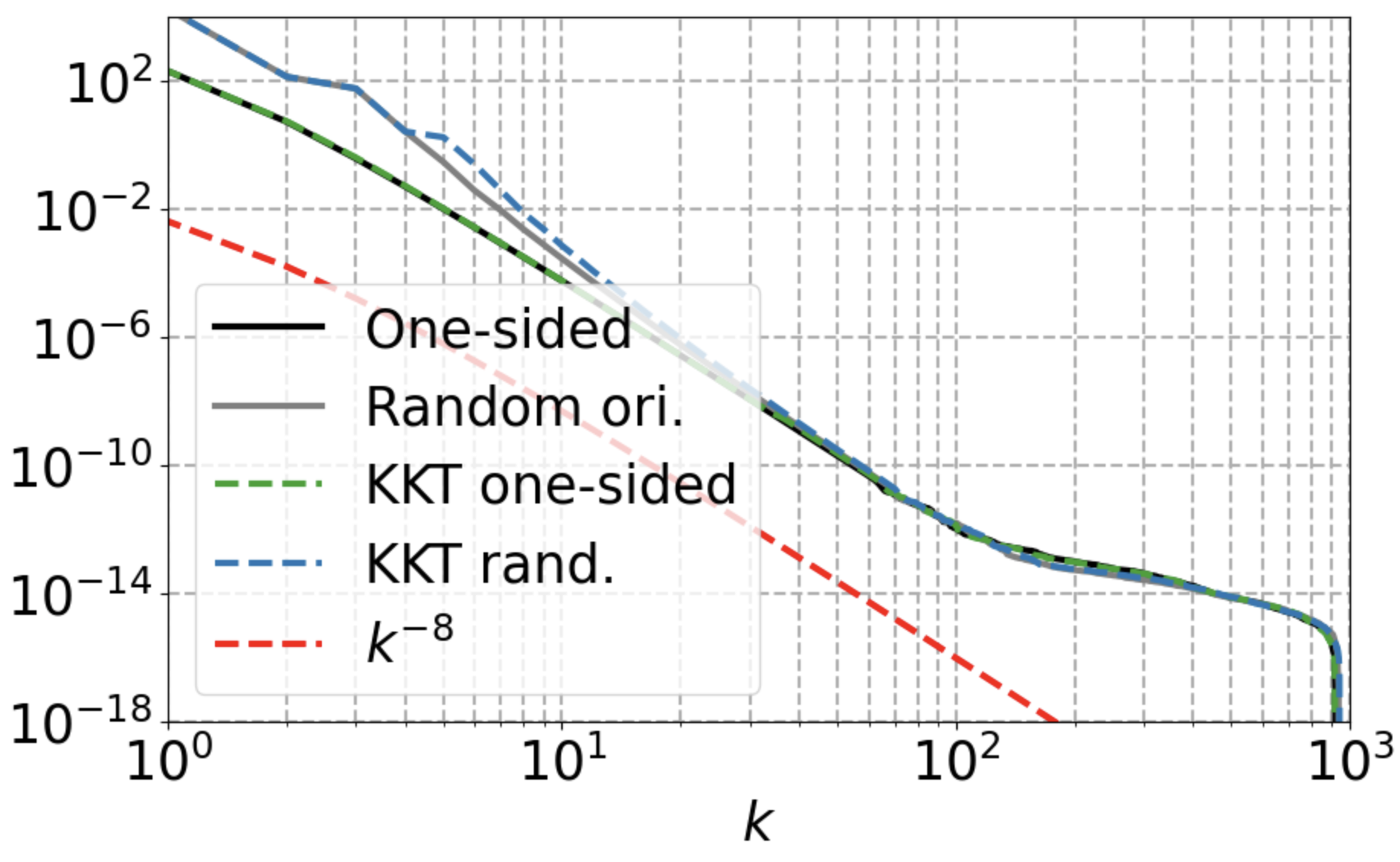}
\end{tabular}
\end{center}
\caption{Spectrum of the $1000\times1000$ Gram matrices associated with  ReLU$^{p}$ activation function for (a) $p=1$, (b) $p=2$, (c) $p=3$. The red dashed curves show the theoretically predicted decay rate $k^{-(2p+2)}$, for each $p$. }\label{fig_spectrum}
\end{figure}

In Figure~\ref{fig_spectrum}, we show the spectrum of $\bG_p\in\mathbb{R}^{1000\times 1000}$ in the descending order  with uniform bias $\bb$ and $\bw=(1,\cdots, 1)$ (One-sided) for $p=1$ in (a), $p=2$ in (b), and $p=3$ in (c). According to Theorem~\ref{theorem_sharp_eigenvalue}, the $k$-th eigenvalue of $\bG_p$ is of order $\cO(k^{-(2p+2)})$. To validate this statement, we also plot $k^{-(2p+2)}$ in red dashed lines for reference, and we observe that the decaying behaviors align perfectly with our analysis for eigenvalues above $10^{-12}$. Due to the finite precision of the floating-point arithmetic, the smaller eigenvalues are less accurate, exhibiting deviation from the prediction. Such instability is more significant when a greater power $p$ is used, as the decaying rate becomes faster. We also added the spectrum for the Gram matrix, where the weight parameters $\bw$ take the values $1$ or $-1$ with probability $1/2$ (Random ori.). The comparison shows that the decaying behaviors are similar. 

In Figure~\ref{fig_eigenvectors}, we plot the eigenvectors of the Gram matrix $\bG_2$, i.e., for shallow networks with $\ReLU^2$ as the activation function. From (a) to (c), we observe that the eigenvectors corresponding to the smaller eigenvalues exhibit higher-frequency oscillations, which aligns with our theoretical observation in~\eqref{eq_general_eigen}. Here we use $N=1000$ and note that it is not necessary to compute or display eigenvectors associated with even smaller eigenvalues. Indeed, the accuracy of the computed eigenvectors depends on the gaps between the corresponding eigenvalues and their nearest neighbors (see, for example, Theorem 11.7.1 in~\cite{parlett1998symmetric}).  Specifically, the angular difference between the computed $i$-th  normalized eigenvector $\widehat{\bv}_i$ and the corresponding true eigenvector $\bv_i$ can be as bad as $\cO(\epsilon_M/\text{gap}(\mu_i))$, where $\text{gap}(\mu_i):=\min_{j\neq i}|\mu_i-\lambda_j|$ with $\mu_i:= \widehat{\bv}_i^\top\bG_p\widehat{\bv}_i$ being the Ritz value associated with the $i$-th eigenvector.   By Theorem~\ref{theorem_sharp_eigenvalue}, $\text{gap}(\mu_i)$ is in  order of $\cO(i^{-3-2p})$ as the matrix size $N\to+\infty$, indicating that the number of bits should  grow as $\cO((3 + 2p)\log_{10}i)$ to maintain a prescribed level of accuracy. In particular, to accurately compute the eigenvector corresponding to the smallest eigenvalue, approximately $\cO(2 \log_{10} N)$ additional digits are needed when the regularity parameter $p$ is increased by one.

\begin{figure}
\begin{center}
\begin{tabular}{c@{\vspace{2pt}}c@{\vspace{2pt}}c}
(a)&(b)&(c)\\
\includegraphics[width=0.33\textwidth]{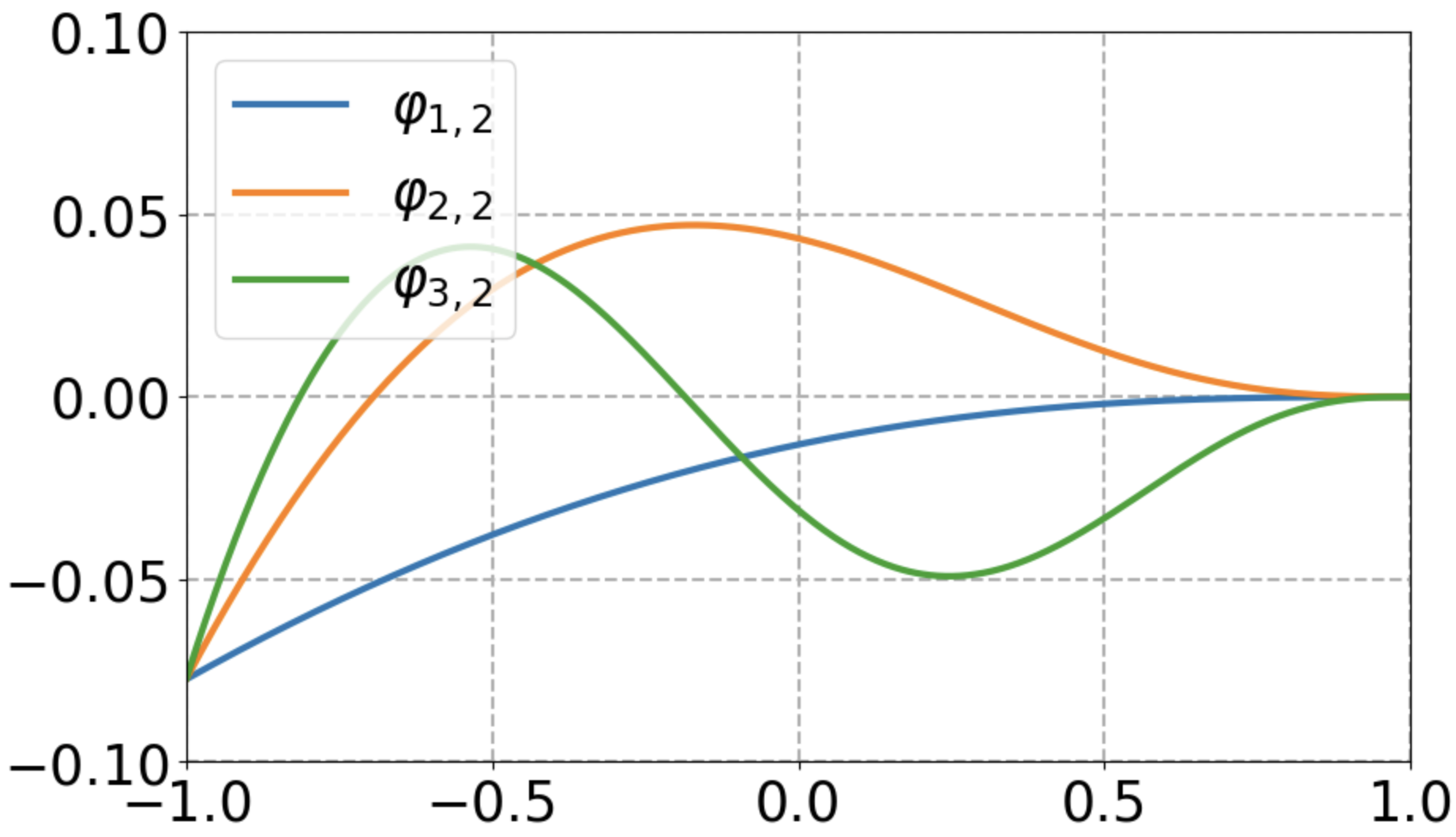}&
\includegraphics[width=0.33\textwidth]{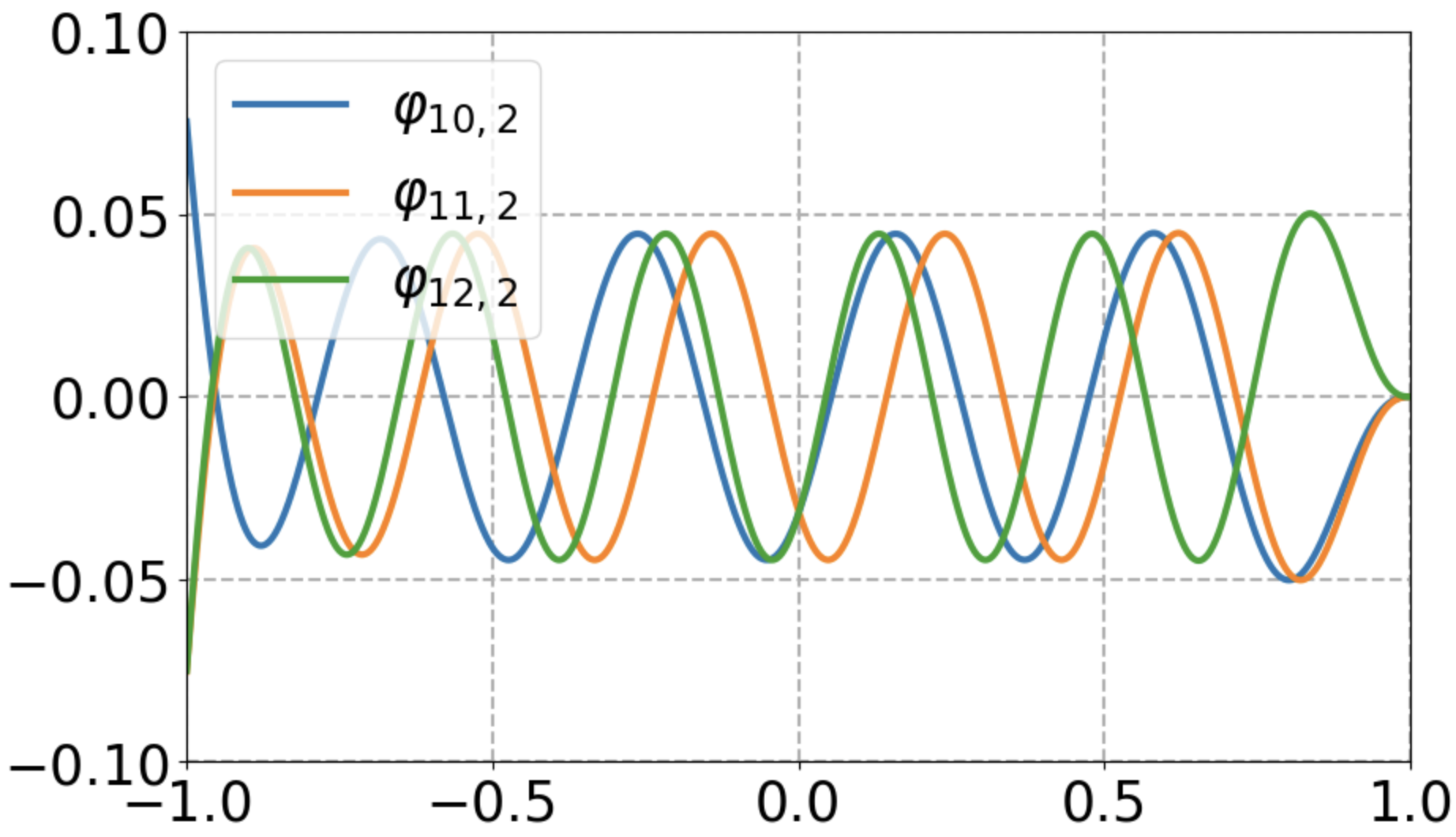}&
\includegraphics[width=0.33\textwidth]{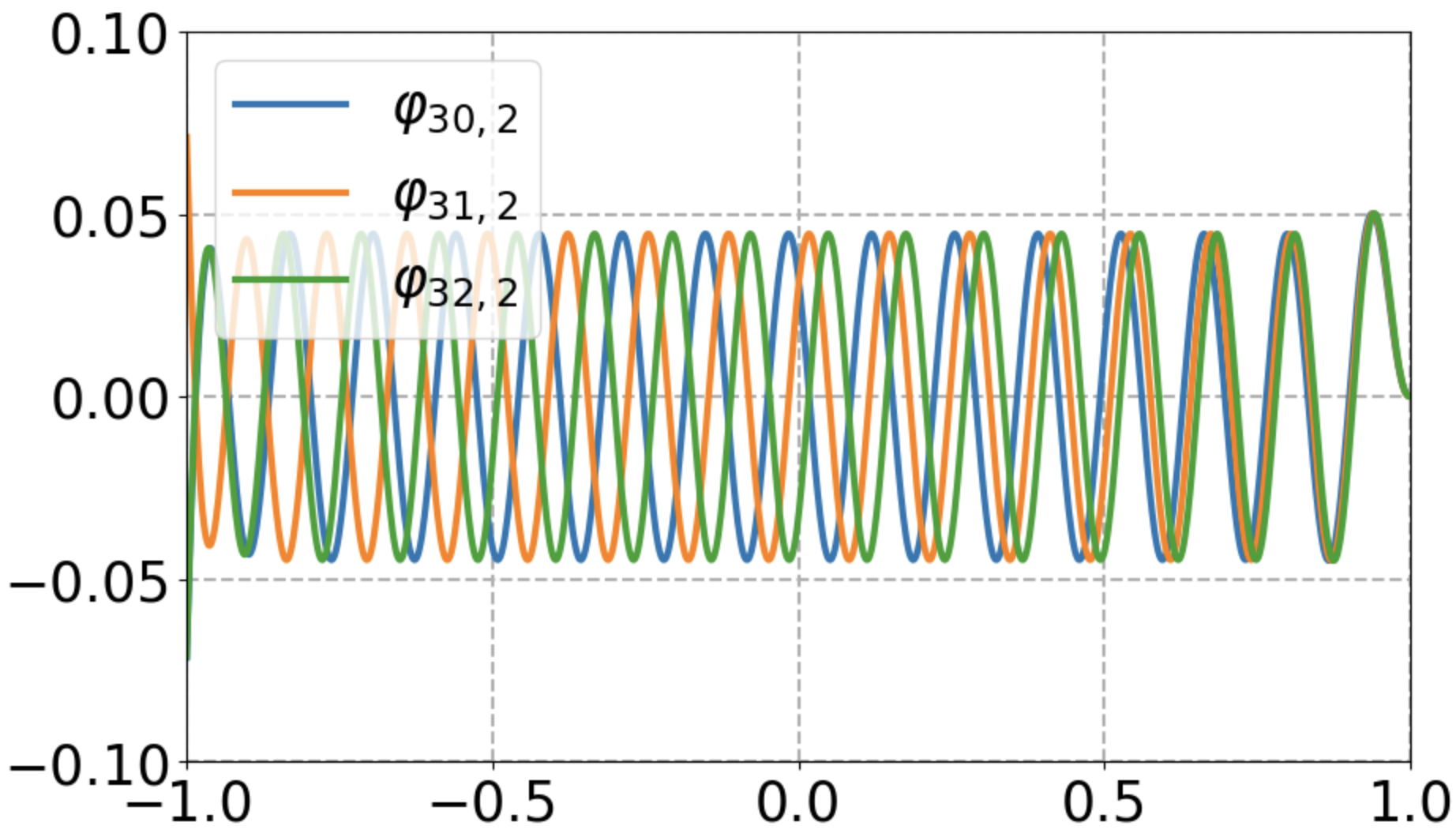}
\end{tabular}
\end{center}
\caption{Eigenvectors  of the Gram matrix $\bG_2$ corresponding to shallow networks with $\ReLU^2$ using $1000$ uniform biases $\bb$ with all-one $\bw$ ordered according to the descending eigenvalues. }\label{fig_eigenvectors}
\end{figure}

\subsection{Accuracy of direct solvers with boundary constraints}\label{sec_constraint_accuracy}

Based on the analysis in the previous section, we investigate how the spectral properties of the Gram matrices influence the direct solution of the quadratic problems~\eqref{eq_constrained_quad}. In particular, for the constrained optimization~\eqref{eq_constrained_quad}, the  Karush–Kuhn–Tucker (KKT) condition for the unique  minimizer $\ba^*\in\mathbb{R}^{N}$ is 
\begin{equation}\label{eq_opt_KKT}
\bK_F\bz_F^* :=\bK_F\begin{bmatrix}
\ba^*\\
\lambda^*_1\\
\lambda^*_2
\end{bmatrix} =
\begin{bmatrix}
 \bG_{F}&\bB^\top\\
\bB&\mathbf{0}
\end{bmatrix}\begin{bmatrix}
\ba^*\\
\lambda^*_1\\
\lambda^*_2
\end{bmatrix}=\begin{bmatrix}
\by_{F}\\
c_L\\
c_R
\end{bmatrix}\;,
\end{equation}
where  $\bK_F\in\mathbb{R}^{(N+2)\times(N+2)}$ is called the KKT matrix. Note that $\bK_F$ is nonsingular if $\bG_F$ is positive definite and $\bB$ defined in~\eqref{eq_B_matrix} with $\sigma = \ReLU^p$ activation functions has rank 2. The symmetric structure of $\bK_F$ implies it has exactly $N$ positive eigenvalues $\lambda_1(\bK_F)\geq \lambda_2(\bK_F)\geq \cdots\geq \lambda_N(\bK_F)>0$ and two negative eigenvalues $0>\lambda_{N+1}(\bK_F)\geq \lambda_{N+2}(\bK_F)$ related to the Schur complement of the block $\bG_F$. Moreover, since $\bB$ has rank $2$ and 
$$
\bK_F =\begin{bmatrix}\bG_F&\mathbf{0}\\
\mathbf{0}&\mathbf{0}
\end{bmatrix} + \begin{bmatrix}\mathbf{0}&\bB^\top\\
\bB&\mathbf{0}\;
\end{bmatrix}\;,
$$
by the eigenvalue interlacing theorem (See for example~\cite{carlson1983minimax}), we have
\begin{equation}
\lambda_{i}(\bG_F)\leq \lambda_{i}(\bK_F)\leq \lambda_{i-4}(\bG_F), \quad i=5,\dots,N.
\end{equation}

Combined with Theorem~\ref{theorem_sharp_eigenvalue}, we deduce the following spectral estimate.

\begin{proposition}\label{prop_decay_KKT}
For PINN with $\ReLU^p$ where $p\geq 2$, we have
\begin{equation}
\lambda_j(\bK_\PINN) =\Theta\left(j^{2-2p}\right)\;.
\end{equation}
For DRM with $\ReLU^p$ where $p\geq 1$, we have
\begin{equation}
\lambda_j(\bK_\DRM) =\Theta\left(j^{-2p}\right)\;.
\end{equation}
\end{proposition}

In Figure~\ref{fig_spectrum}, we also include the spectrum of the KKT matrix for the $\ReLU^p$ activation function with $p=1,2$, and $3$, which exhibits the same decay property as described in Proposition~\ref{prop_decay_KKT}. We can be more specific as follows. 
\begin{lemma} \label{lemma_eigenvector_KF}There are at least $N-2$ eigenvectors of $\bK_F$ of the form $\widetilde{\bx}^\top = [\bx^\top, \mathbf{0}^\top]$ with eigenvalue $\lambda>0$, where $\bx$ is an eigenvector of  $\bG_F$ corresponding to eigenvalue $\lambda$. 
\end{lemma}
\begin{proof}
See Appendix~\ref{proof_lemma_eigenvector_KF}.
\end{proof}
\begin{remark}
The other at most four eigenvectors of $\bK_F$ are of the form $\widetilde{\bx}^\top = [\bx^\top,\bb^\top]$ where $\bb = \mu^{-1} \bB\bx$ and $\mu = (\lambda\pm\sqrt{\lambda^2+4\gamma})/2$ where $\bx$ is an eigenvector of $\bG_F$ for the eigenvalue $\lambda$, and $\gamma$ is an eigenvalue of $\bB^\top\bB$. 
\end{remark}

The analysis above has important implications for the numerical solutions of~\eqref{eq_opt_KKT}. Since $\bK_F$ is ill-conditioned, it is common to solve~\eqref{eq_opt_KKT} by truncated singular value decomposition (SVD). Specifically, let $\bU \bSigma\bV^\top$ be the SVD of $\bK_F$ where $\bU\in\mathbb{R}^{(N+2)\times (N+2)}$ and $\bV\in\mathbb{R}^{(N+2)\times (N+2)}$ are unitary matrices, and $\bSigma$ is a diagonal matrix recording singular values of $\bK_F$. For some prespecified tolerance parameter $\varepsilon>0$, let $\widetilde{\bSigma}$  be a matrix obtained by substituting diagonal elements with magnitude smaller than $\varepsilon\lambda_{1}(\bK_F)$ with $0$. Then the reduced SVD solution of~\eqref{eq_opt_KKT} is computed by
\begin{equation}
\widetilde{\bz}_F^*:=\begin{bmatrix}
\widetilde{\ba}^*\\
\widetilde{\lambda}_1^*\\
\widetilde{\lambda}_2^*
\end{bmatrix} = \bV\widetilde{\bSigma}^\dagger\bU^\top \begin{bmatrix}
\by_{F}\\
c_L\\
c_R
\end{bmatrix}\;,
\end{equation}
where $\widetilde{\bSigma}^\dagger$ is the pseudo-inverse of $\widetilde{\bSigma}$. We can characterize the error caused by SVD truncation.

\begin{corollary}\label{corollary_trunc_SVD_error}Define the set $\cI_N(\varepsilon)=\{j\in \cI_1~|~|\lambda_j(\bK_F)|<\varepsilon\lambda_1(\bK_F)\} $ where $\cI_1$ with $|\cI_1|\geq N-2$ is the set of indices  of eigenvectors of $\bK_F$ obtained by lifting those of $\bG_F$, then
\begin{equation}
\|\widetilde{\bz}_F^* - \bz_F^*\|^2\geq\sum_{i\in \cI_N(\varepsilon)} \left(\frac{\bu_i^\top\by_F}{\lambda_i(\bK_F)}\right)^2\;.
\end{equation}
\end{corollary}
\begin{proof}
See Appendix~\ref{proof_corollary_trunc_SVD_error}.
\end{proof}

\begin{figure}
\centering
\begin{tabular}{c@{\vspace{2pt}}c@{\vspace{2pt}}c}
\toprule
$N=100$&$N=300$&$N=500$\\\midrule
(a)&(b)&(c)\\
\includegraphics[width=0.33\textwidth]{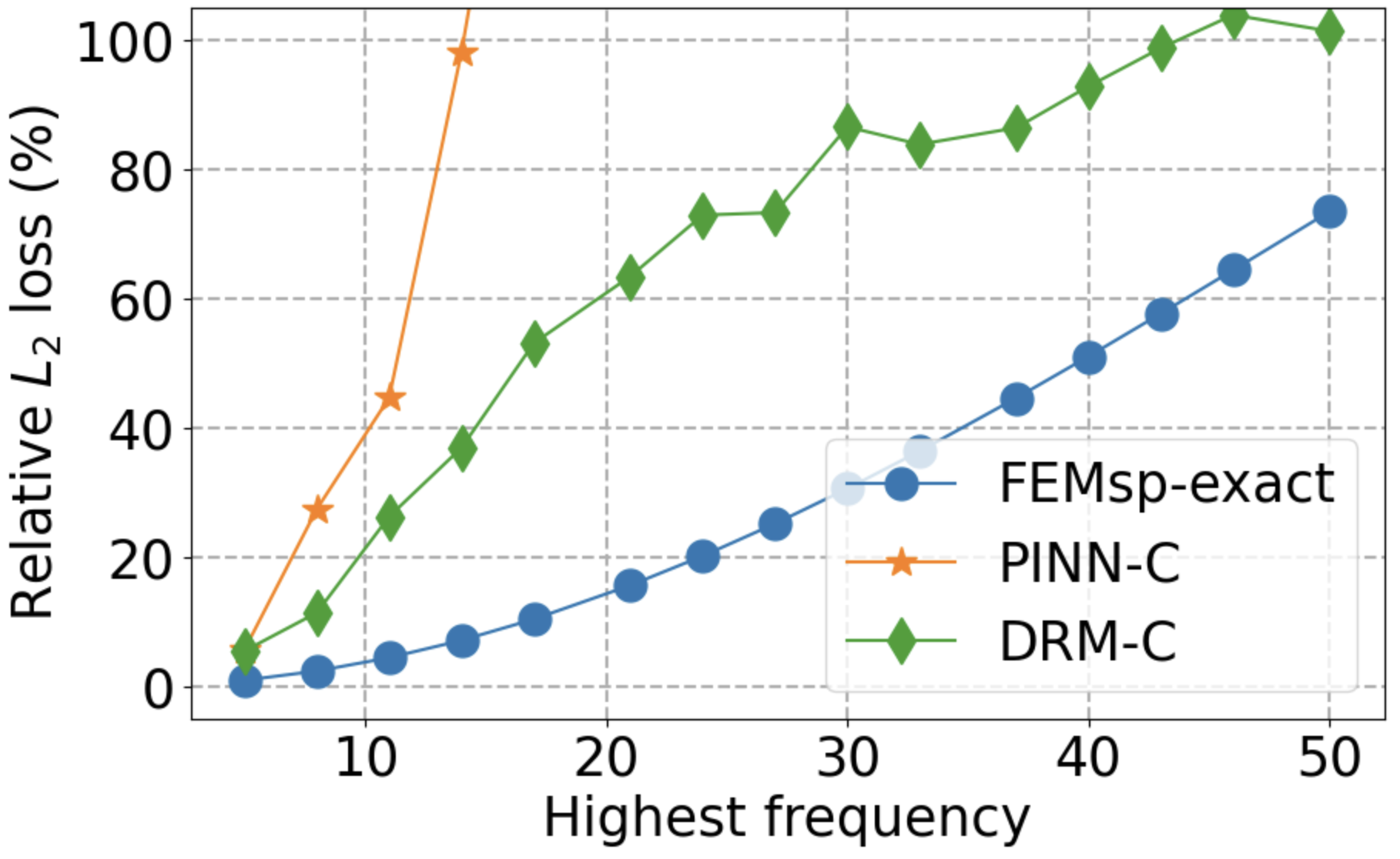}&
\includegraphics[width=0.33\textwidth]{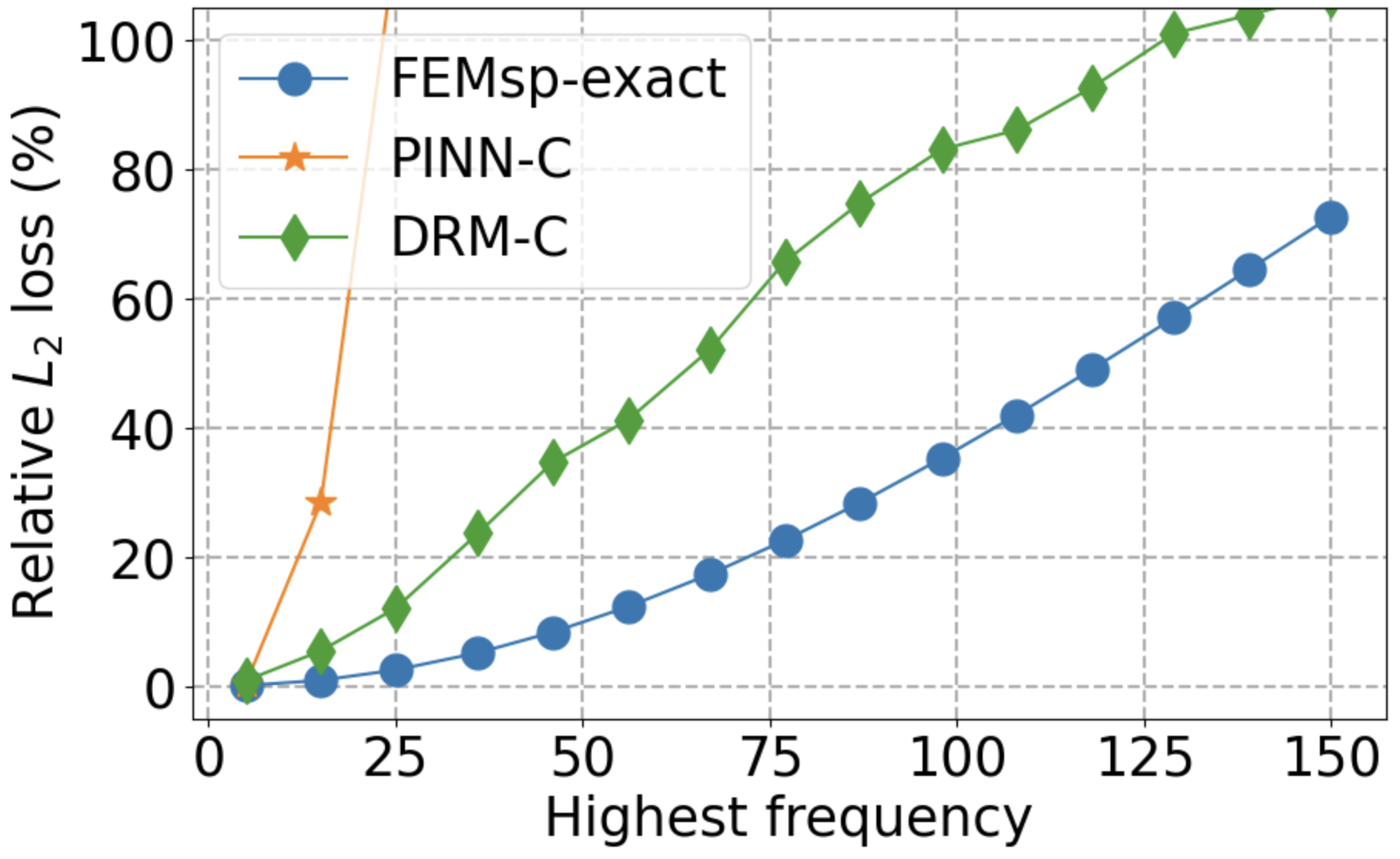}&
\includegraphics[width=0.33\textwidth]{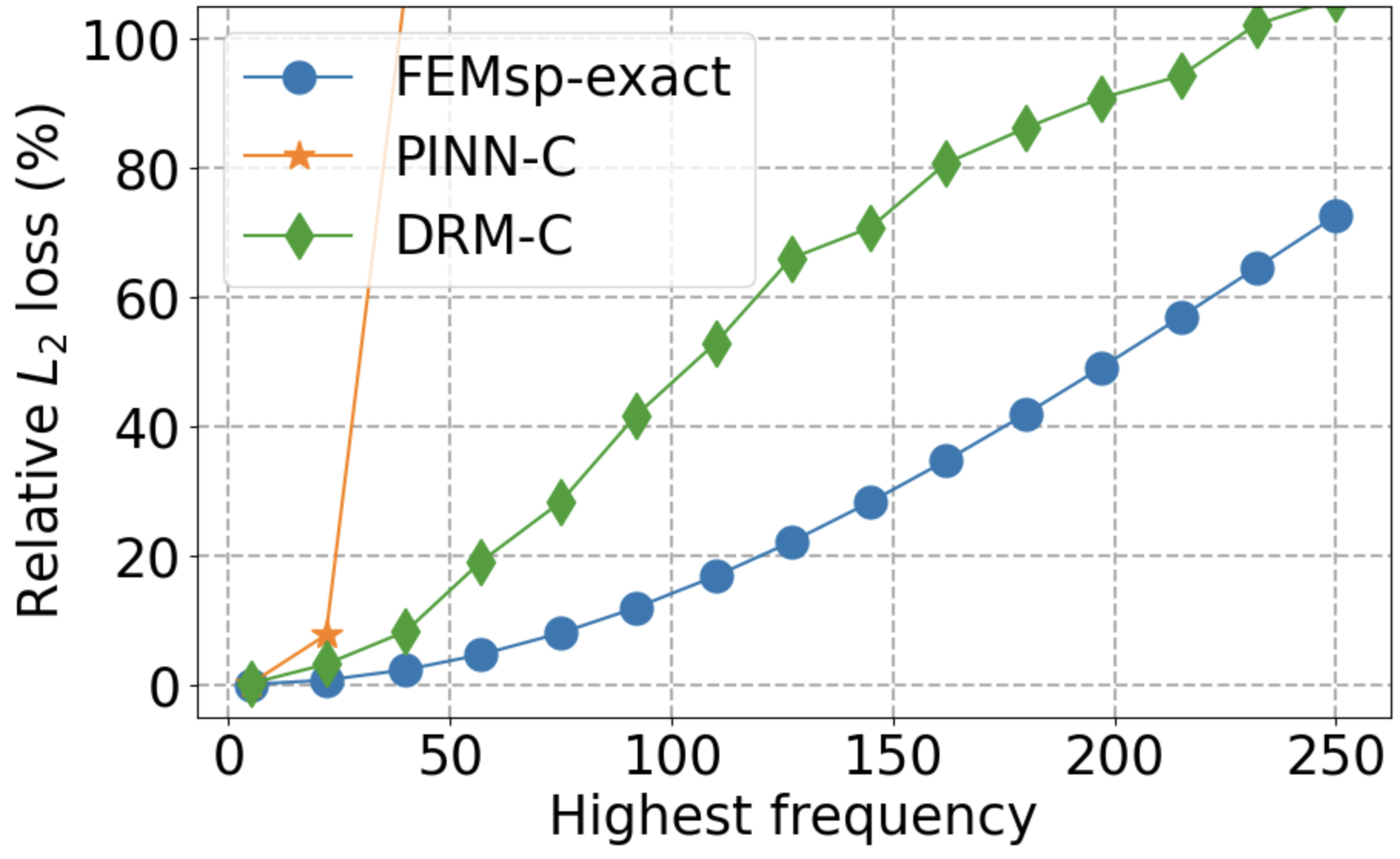}\\
(d)&(e)&(f)\\
\includegraphics[width=0.28\textwidth]{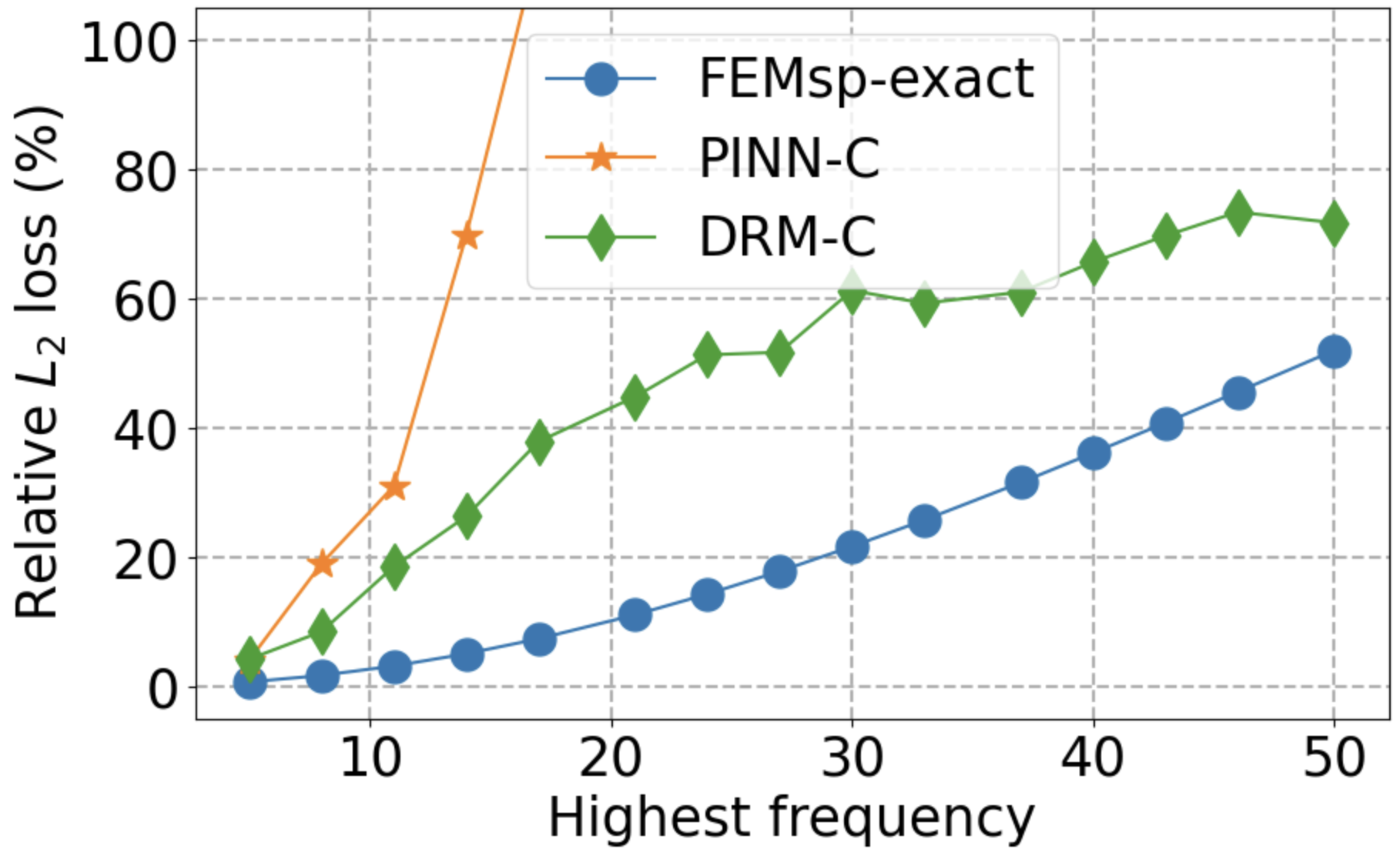}&
\includegraphics[width=0.28\textwidth]{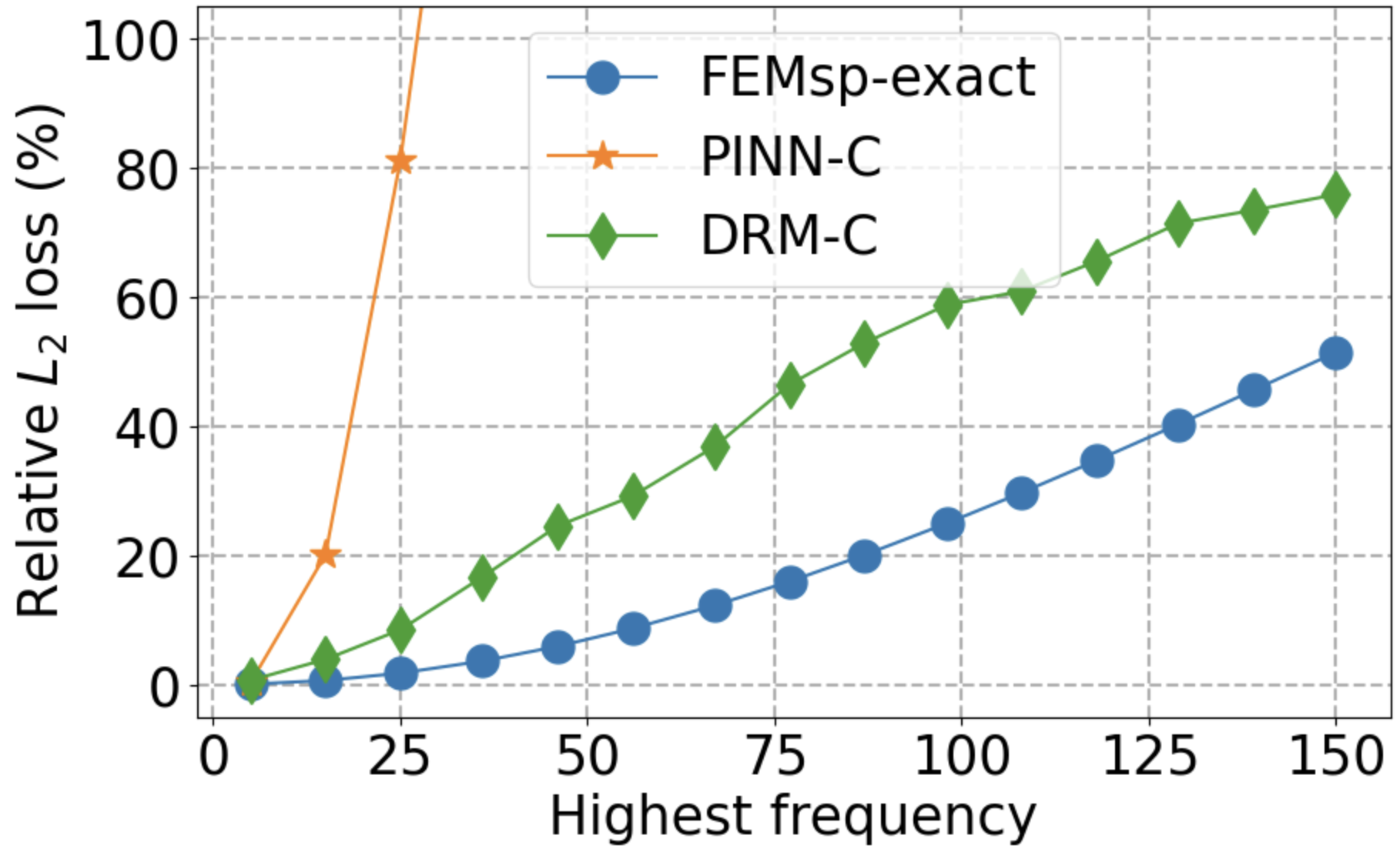}&
\includegraphics[width=0.28\textwidth]{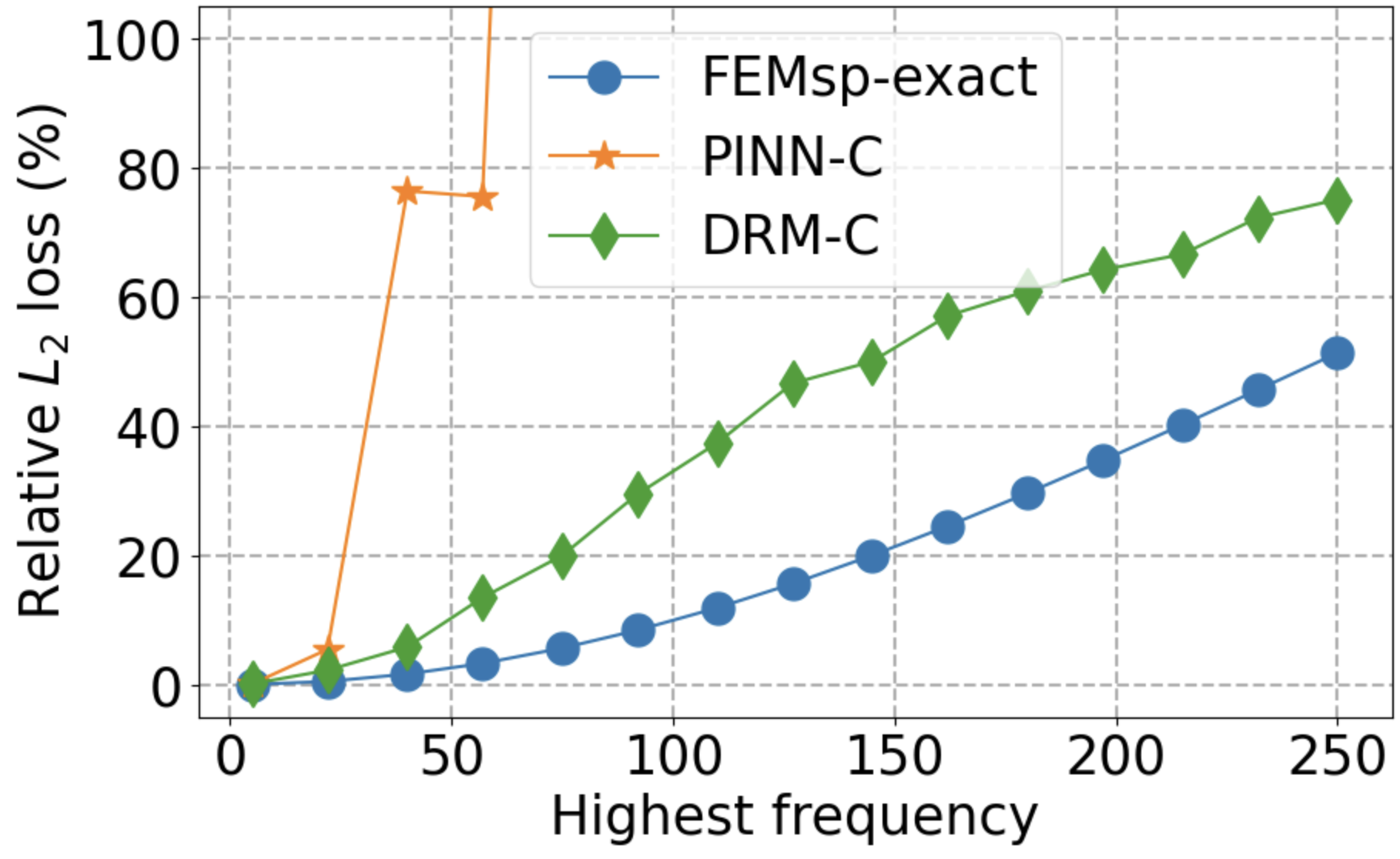}
\end{tabular}
\caption{Relative $L_2$ error of the solutions by FEMsp with linear bases, PINN with $\ReLU^2$, and DRM with $\ReLU$ and different numbers of bases (corresponding to evenly spaced grid points) $N$. For (a)-(c), the underlying solution has single Fourier mode: $\sin(k_{\max}\pi x-2\pi/3)$, while for (d)-(f), the  solution has two Fourier modes: $\sin(2\pi x+3\pi/5)+\sin(k_{\max}\pi x-2\pi/3)$ where $k_{\max}$ increases. Here, the Dirichlet boundary conditions are imposed as constraints, and a direct linear solver is used to solve the linear system.  }\label{L2_frequency_error_p1}
\end{figure}

\begin{figure}
\centering
\begin{tabular}{c@{\vspace{2pt}}c@{\vspace{2pt}}c}
\toprule
$N=100$&$N=300$&$N=500$\\\midrule
(a)&(b)&(c)\\
\includegraphics[width=0.33\textwidth]{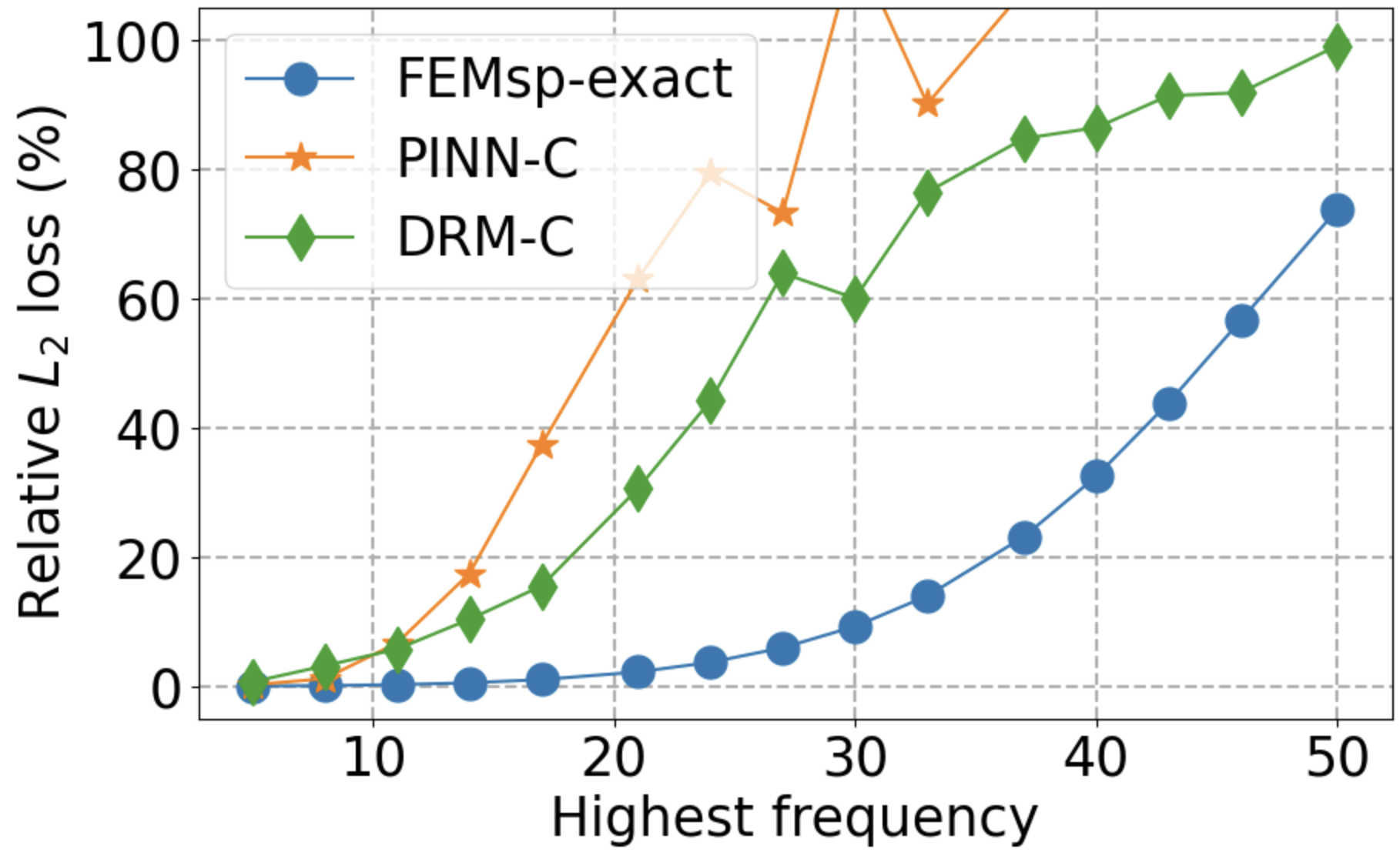}&
\includegraphics[width=0.33\textwidth]{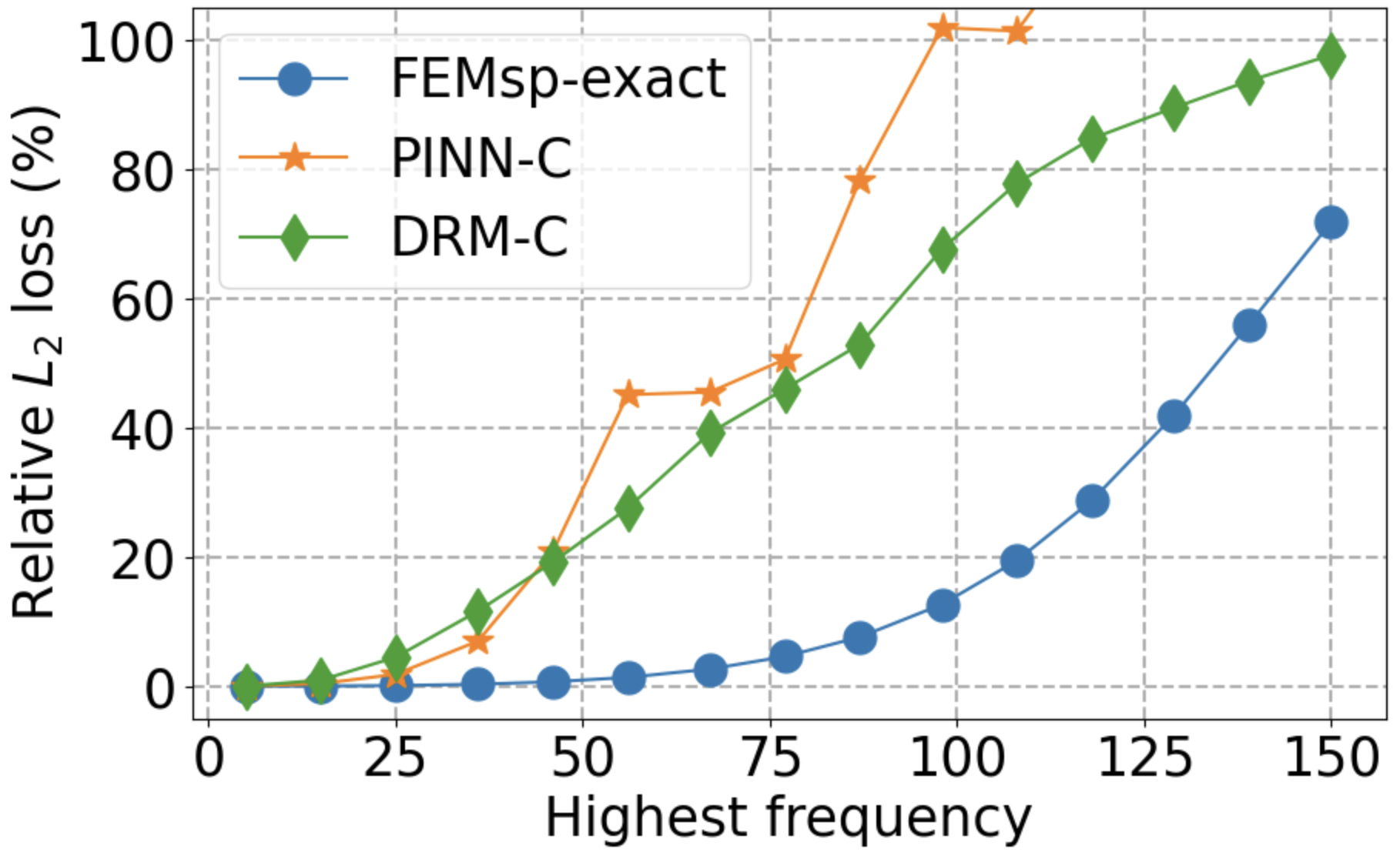}&
\includegraphics[width=0.33\textwidth]{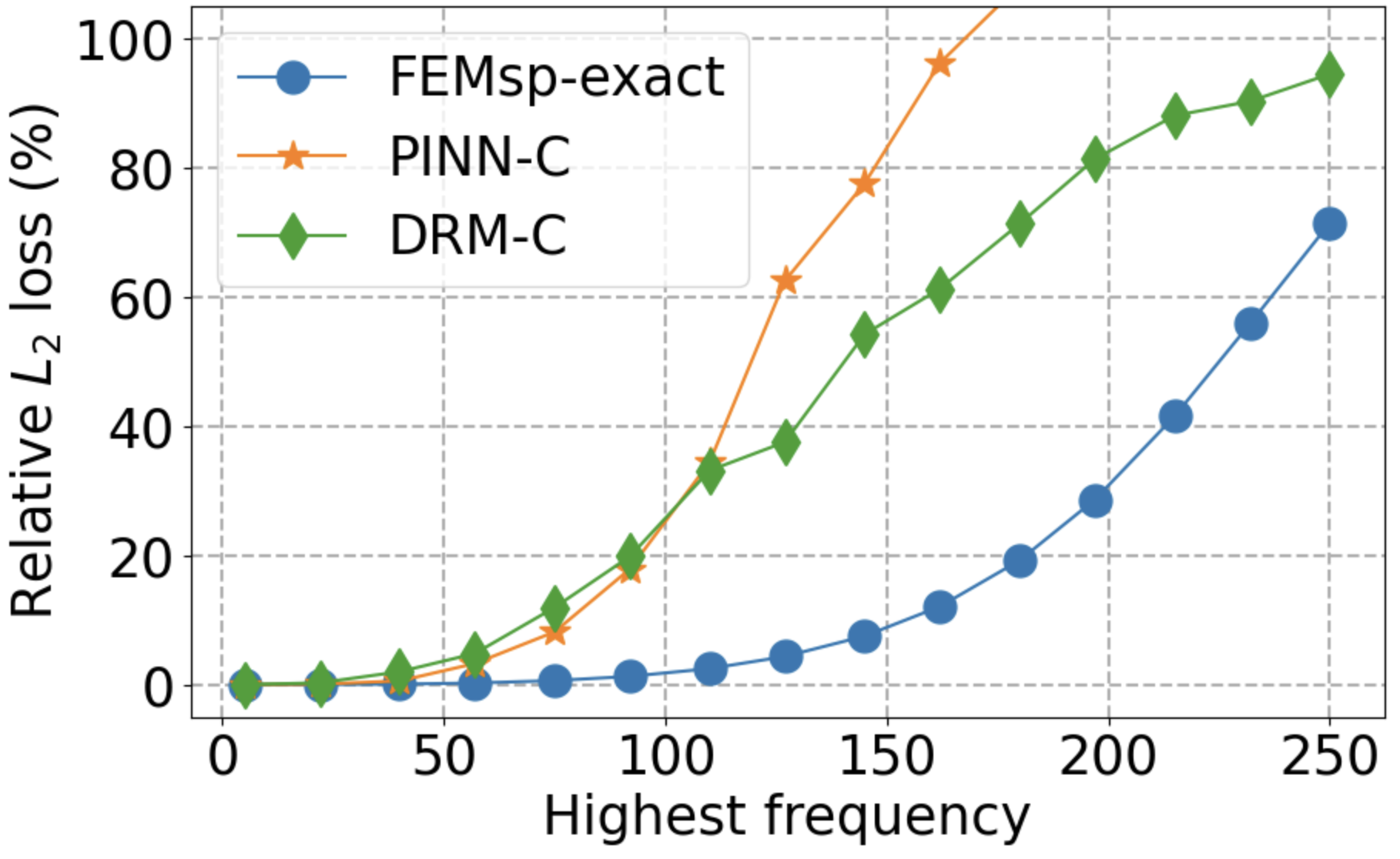}\\
(d)&(e)&(f)\\
\includegraphics[width=0.33\textwidth]{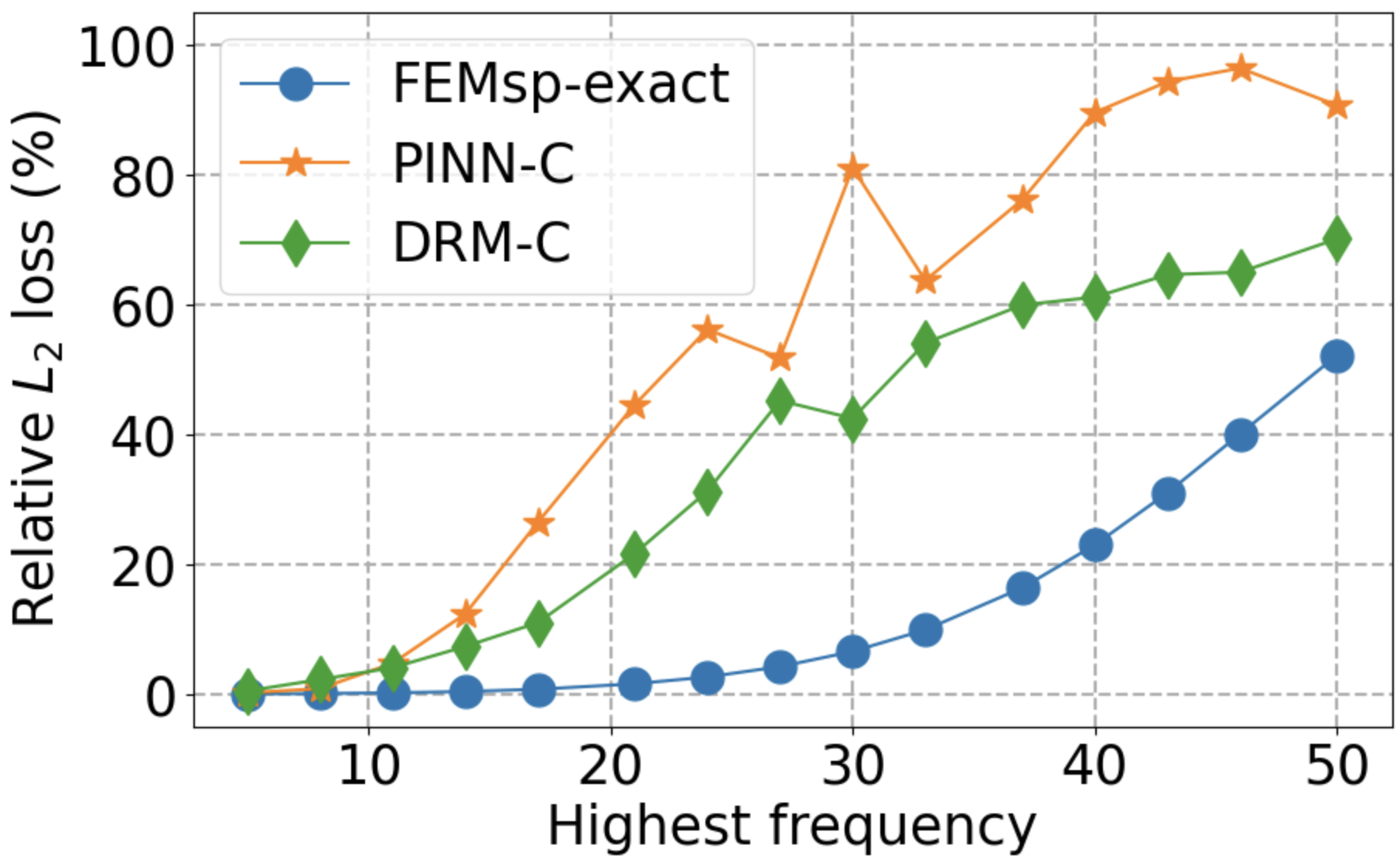}&
\includegraphics[width=0.33\textwidth]{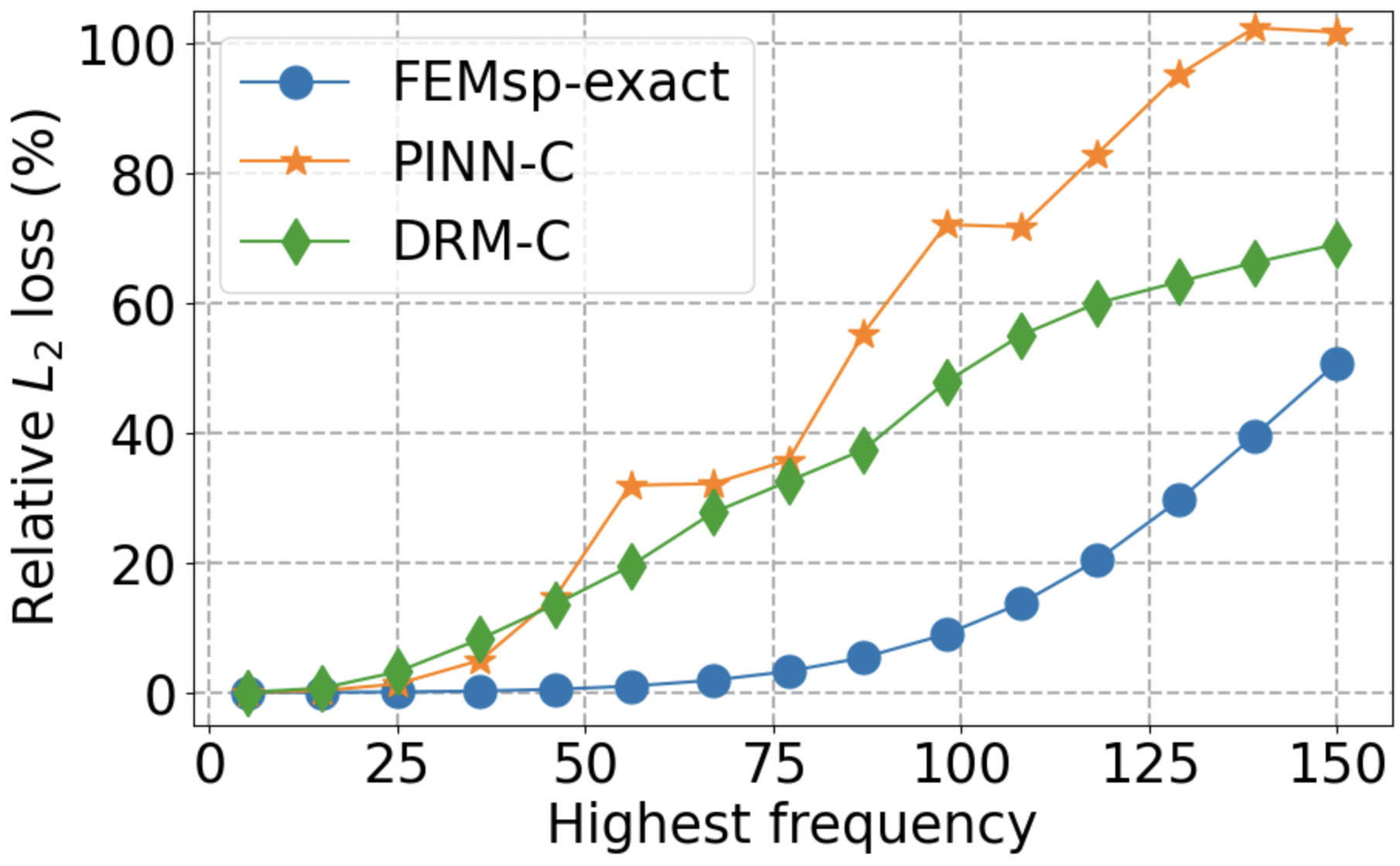}&
\includegraphics[width=0.33\textwidth]{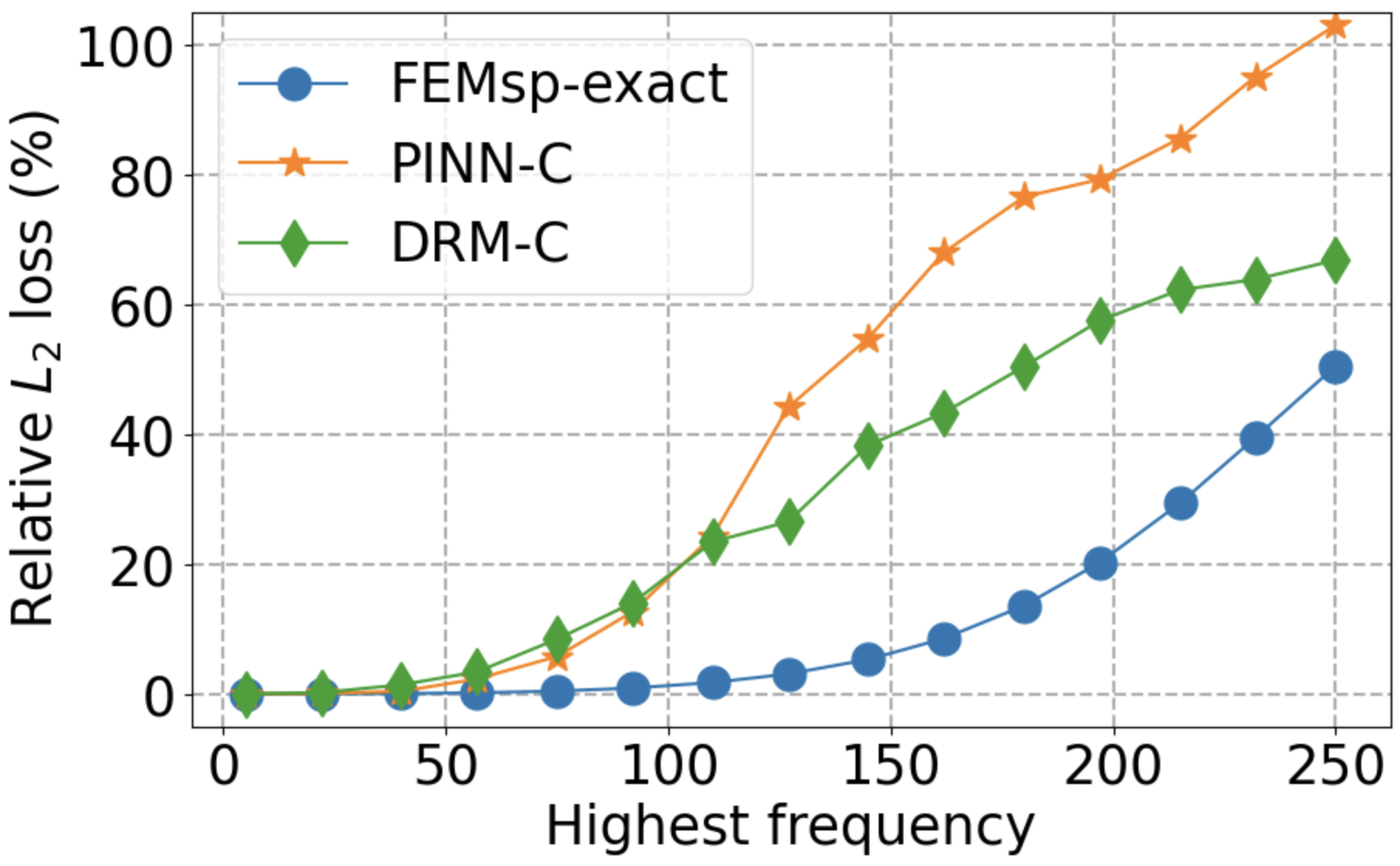}
\end{tabular}
\caption{Relative $L_2$ error of the solutions by FEMsp with quadratic bases, PINN with $\ReLU^3$, and DRM with $\ReLU^2$ and different numbers of bases (corresponding to evenly spaced grid points) $N$. The other experimental set-ups are identical with those in Figure~\ref{L2_frequency_error_p1}. }\label{L2_frequency_error_p2}
\end{figure}

Now we discuss the implications of these spectral analyses in numerical computations and verify these understandings using numerical tests. In the current setup, linear representations are used for FEMsp, PINN, and DRM. Mathematically, FEMsp and DRM are equivalent in terms of weak formulations. They both minimize the $H_1$ semi-norm of the difference between the true solution and its approximation in the space spanned by the basis functions. PINN adopts a strong formulation which minimizes the $H_2$ semi-norm of the difference between the true solution and its approximation in the space spanned by the basis functions. However, the spectral properties will reveal different frequency biases and numerical errors in real computation. Assume we are solving the one-dimensional Poisson equation on a uniform grid, i.e., $b_i$ are evenly spaced on $[-1, 1]$ with a grid size of $h$. For the FEM, if $B^p(x)$ (B splines of degree $p$) is used as the finite element space, the stiffness matrix has eigenvalues $\lambda_k=\Theta\left(k^2\right)$ for frequency $k$. Numerical inversion of the system is stable, and the numerical error is of order $(kh)^{p+1}$ when the Fourier mode $k$ is the PDE solution. For DRM using $\ReLU^{p}$, which spans the same space as $B^p(x)$, the Gram (stiffness) matrix $K_{\DRM}$ has eigenvalues $\lambda_k(\bK_\DRM)=\Theta\left(k^{-2p}\right)$ for frequency $k$ (Proposition~\ref{prop_decay_KKT}). Numerical inversion of the system will amplify the error in frequency $k$ by $k^{2p}$ by Proposition~\ref{corollary_trunc_SVD_error}. For PINN using $\ReLU^p$, the Gram matrix $K_{\PINN}$ has eigenvalues $\lambda_k(\bK_\DRM) =\Theta\left(k^{2-2p}\right)$ for frequency $k$ (Proposition~\ref{prop_decay_KKT}). Numerical inversion of the system will amplify the error in frequency $k$ by $k^{2-2p}$. 

In the following tests, the finite element method is used as the reference to verify the numerical implications stated above. In all tests, all integrals are computed exactly and solved by a direct solver for the linear system with boundary constraints. In the first test, we show the relative $L_2$ error of the numerical solution using linear basis for FEM, $\ReLU$ for DRM, and $\ReLU^2$ for PINN in Figure~\ref{L2_frequency_error_p1}. As we can see, the numerical error for FEM behaves like $\cO(k^2)$ as expected. The DRM produces increasingly significant errors as the solution contains higher and higher Fourier modes. PINN with $\ReLU^2$ does not work due to the use of a strong formulation. In the second test, we show the relative $L_2$ error of the numerical solution using quadratic B splines for FEM, $\ReLU^2$ for DRM, and $\ReLU^3$ for PINN in Figure~\ref{L2_frequency_error_p2}. As expected, the numerical error for FEM behaves like $\cO(k^3)$. The DRM produces increasingly significant errors as higher Fourier modes are introduced into the solution. It is also interesting to note the difference between PINN with $\ReLU^3$ and DRM with $\ReLU^2$ in this linear setting. Although their  Gram matrices have identical spectral decay rates, the strong formulation of PINN is equivalent to solving $-\Delta^2 u=\Delta f$.
In contrast, the weak formulation of DRM is equivalent to solving the original PDE: $-\Delta u=f$. Hence, using $\ReLU^3$ in PINN results in one order higher approximation error than using $\ReLU^2$ in DRM, as shown in the test. Also note that these results indicate that numerical errors are relatively stable with respect to the dimensionless number $kh$ (or $k_{\max}/N$), as is typically the case in classical numerical analysis. 

\begin{figure}
\centering
\begin{tabular}{c@{\vspace{2pt}}c@{\vspace{2pt}}c}
(a)&(b)&(c)\\
\includegraphics[width=0.33\textwidth]{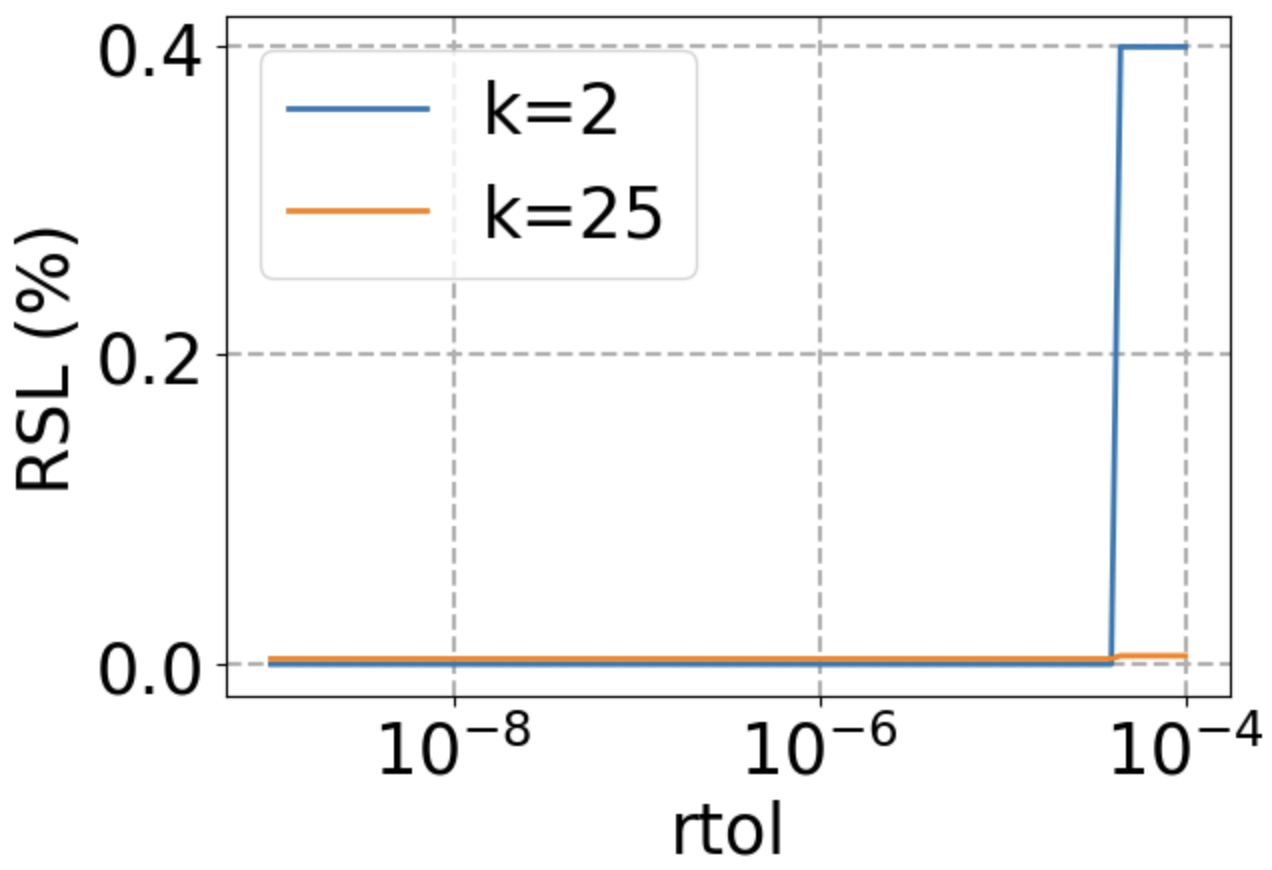}&
\includegraphics[width=0.33\textwidth]{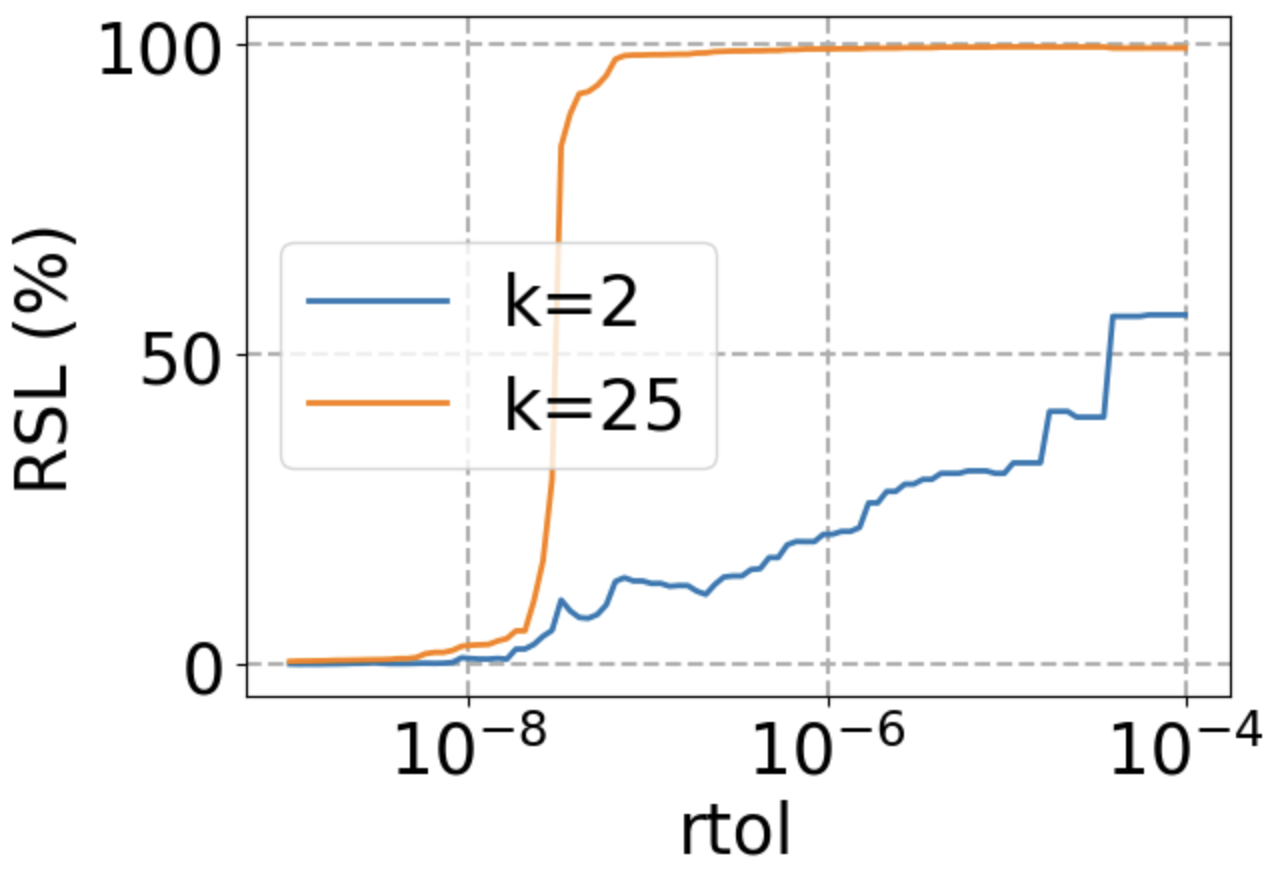}&
\includegraphics[width=0.33\textwidth]{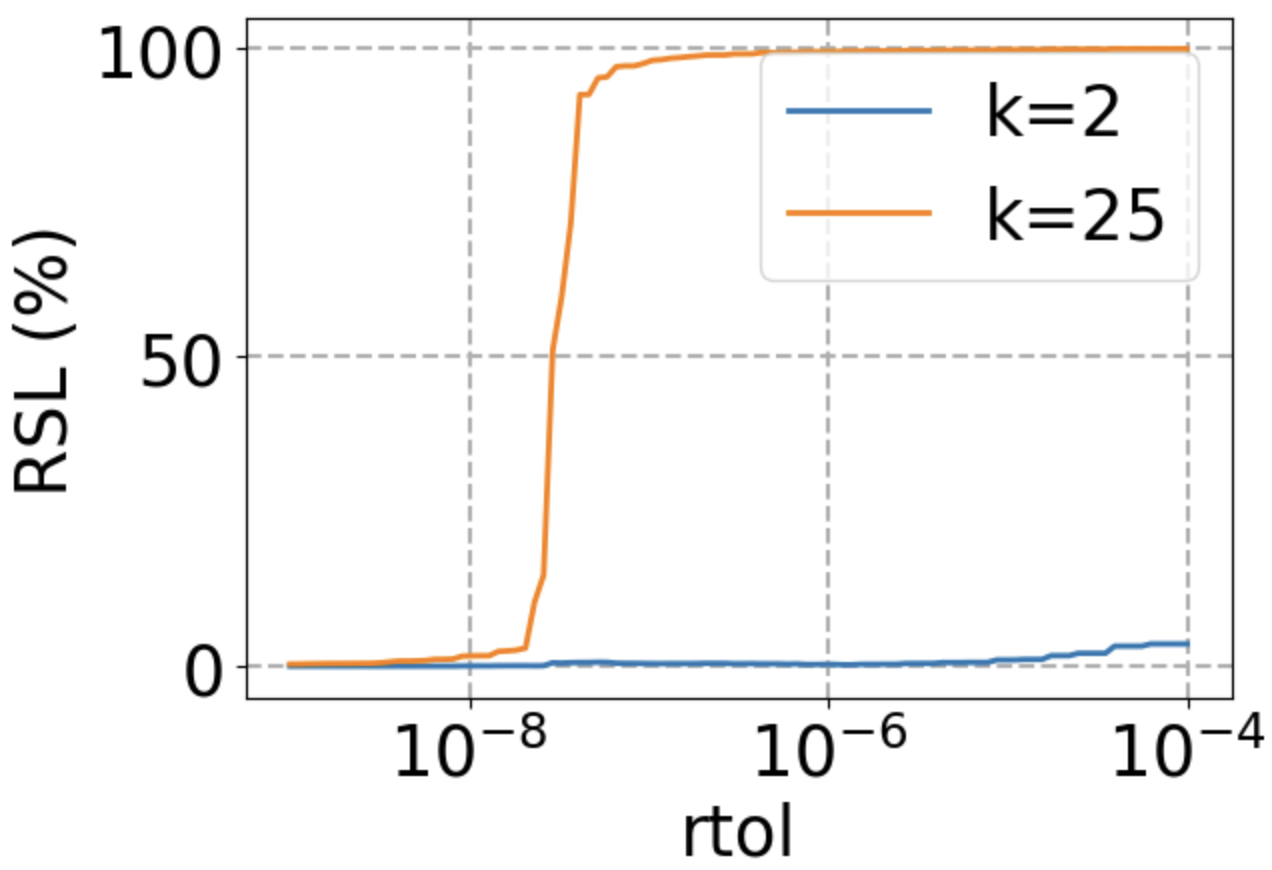}
\end{tabular}
\caption{Effects of truncation on the singular values in the inversion of the Gram matrices. (a) FEMsp, (b) PINN, (c) DRM. The underlying solution is $\sin(2\pi x+4\pi/5)+\sin(25\pi x-3/4\pi)$. Here, all methods are computed using exact integration, and the boundary condition is enforced as constraints. The test data contains $3000$ points, evenly sampled from $[-1, 1]$. The bases for FEMsp are quadratic, the activation function for PINN is $\ReLU^3$, and that for DR is $\ReLU^2$. }\label{fig_trunc_effects}
\end{figure}

Next, we use another experiment, which involves truncating small singular values, to demonstrate the frequency bias and its consequences for the three different methods. We illustrate the introduction of approximation errors for specific frequency components as the truncation threshold is varied. These components correspond to the eigenvectors associated with the attenuated portion of the spectrum.  For PINN and DRM, Proposition~\ref{prop_decaying} indicates that their eigenvalue structures (with power ReLU activations) are biased against high frequencies. Specifically, eigenvectors linked to large eigenvalues represent low frequencies, while those related to small eigenvalues represent high frequencies.
In contrast, the eigenvalue structure of the FEM is biased against low frequencies. As a result, we anticipate that progressively truncating the singular values of $\bK_F$ from small to large will initially cause substantial errors in the approximation of low-frequency components for FEM. For both PINN and DRM, this truncation will initially lead to significant errors in approximating high-frequency components instead.

To validate the above statement, in Figure~\ref{fig_trunc_effects}, we show the relative spectral loss (RSL) defined by
\begin{equation}\label{eq_RSL}
\text{RSL}(f,\zeta) := \frac{|\widehat{f}(\zeta) - \widehat{u^*}(\zeta)|}{|\widehat{u^*}(\zeta)|}\times 100 \%
\end{equation}
for approximations using FEM with quadratic bases, PINN with $\ReLU^3$, and DRM with $\ReLU^2$ as the relative tolerance for the singular values increases; in~\eqref{eq_RSL}, $f$ is the approximate solution, $u^*=\sin(2\pi x+4\pi/5)+\sin(25\pi x-3\pi/4)$  is the exact solution, $\widehat{\cdot}$ denotes the Fourier transform, and $\zeta\in\mathbb{R}$ is the examined frequency. For all methods, the number of bases is $N=400$. For both PINN and DRM, the fixed weight parameters $\bw$ are independently sampled from $\{-1,1\}$ with equal probability, and the biases $\bb$ are independently sampled from the uniform distribution $\cU(-1,1)$.  We observe that for FEM, the approximation errors remain low across a wide range of relative threshold values. As the threshold increases further, the error for the low-frequency component increases first while the error for the high-frequency component remains almost unchanged. In contrast, for both PINN and DRM,  the error for the high-frequency approximation becomes prominent first, followed by the error for the low-frequency component. They confirm our statement.

\subsection{Accuracy of direct solvers with boundary regularization}\label{sec_accuracy_direct_bd_regularization}

\begin{figure}
\centering
\begin{tabular}{c@{\vspace{2pt}}c@{\vspace{2pt}}c}
(a)&(b)&(c)\\
\includegraphics[width=0.33\textwidth]{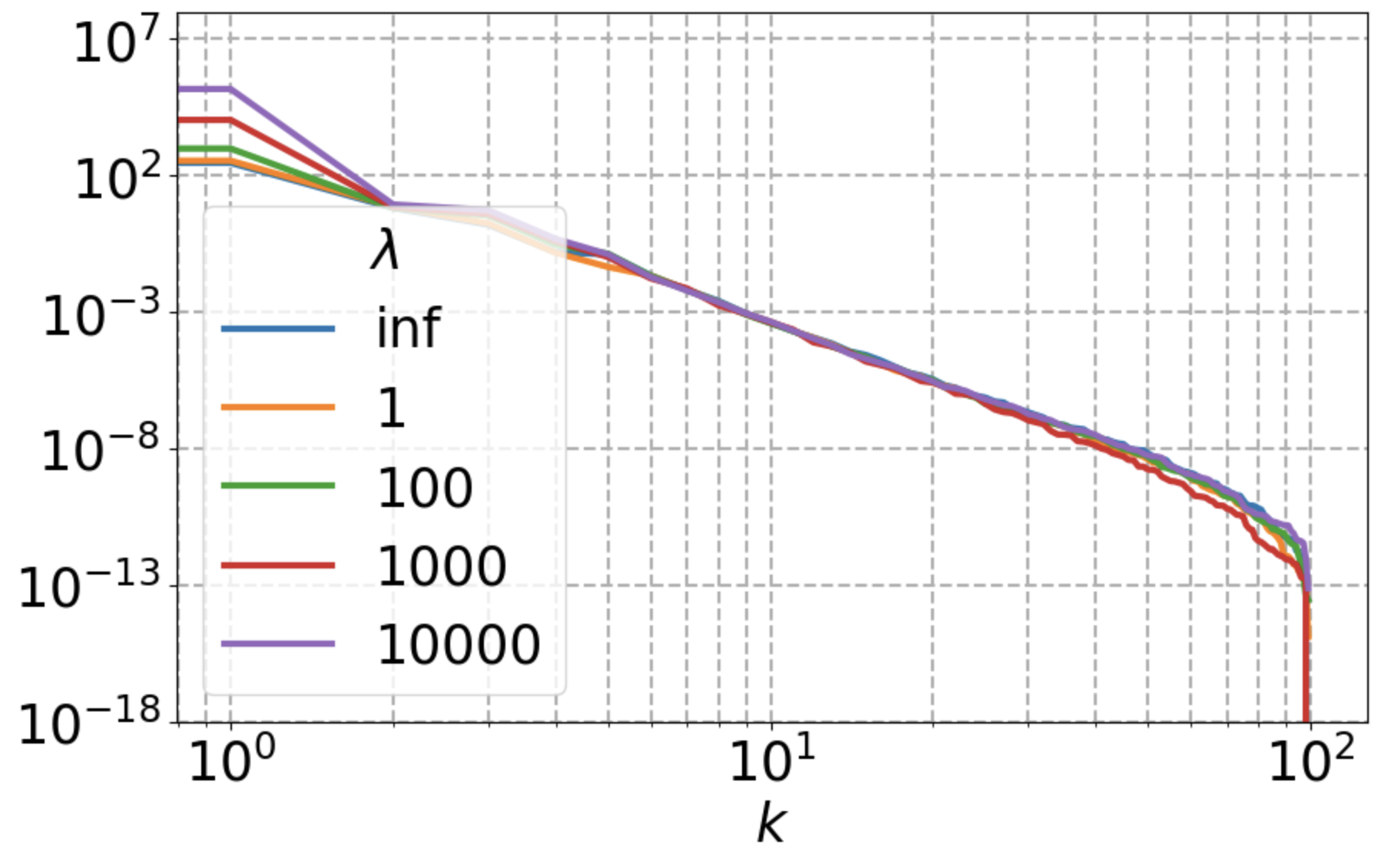}&
\includegraphics[width=0.33\textwidth]{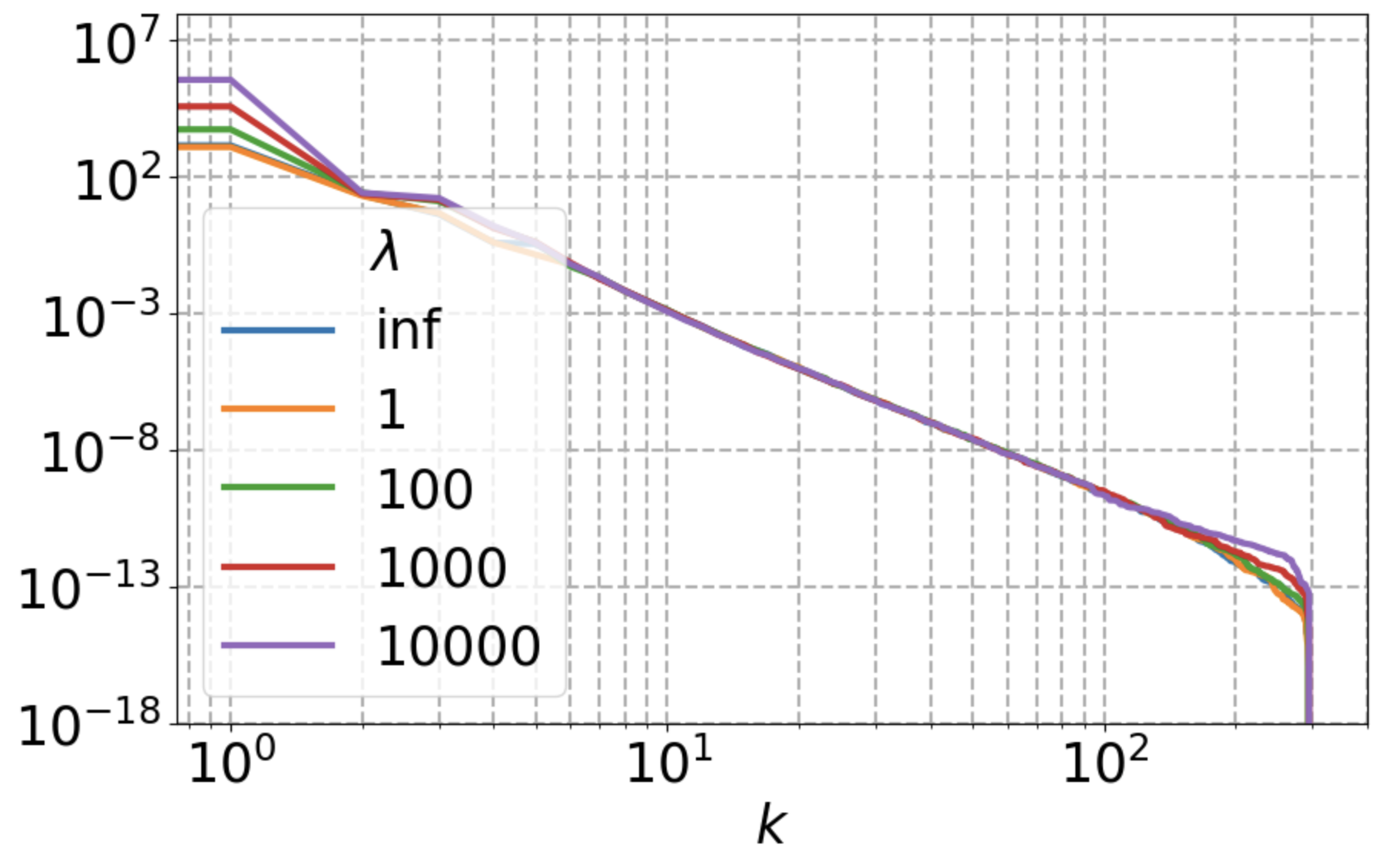}&
\includegraphics[width=0.33\textwidth]{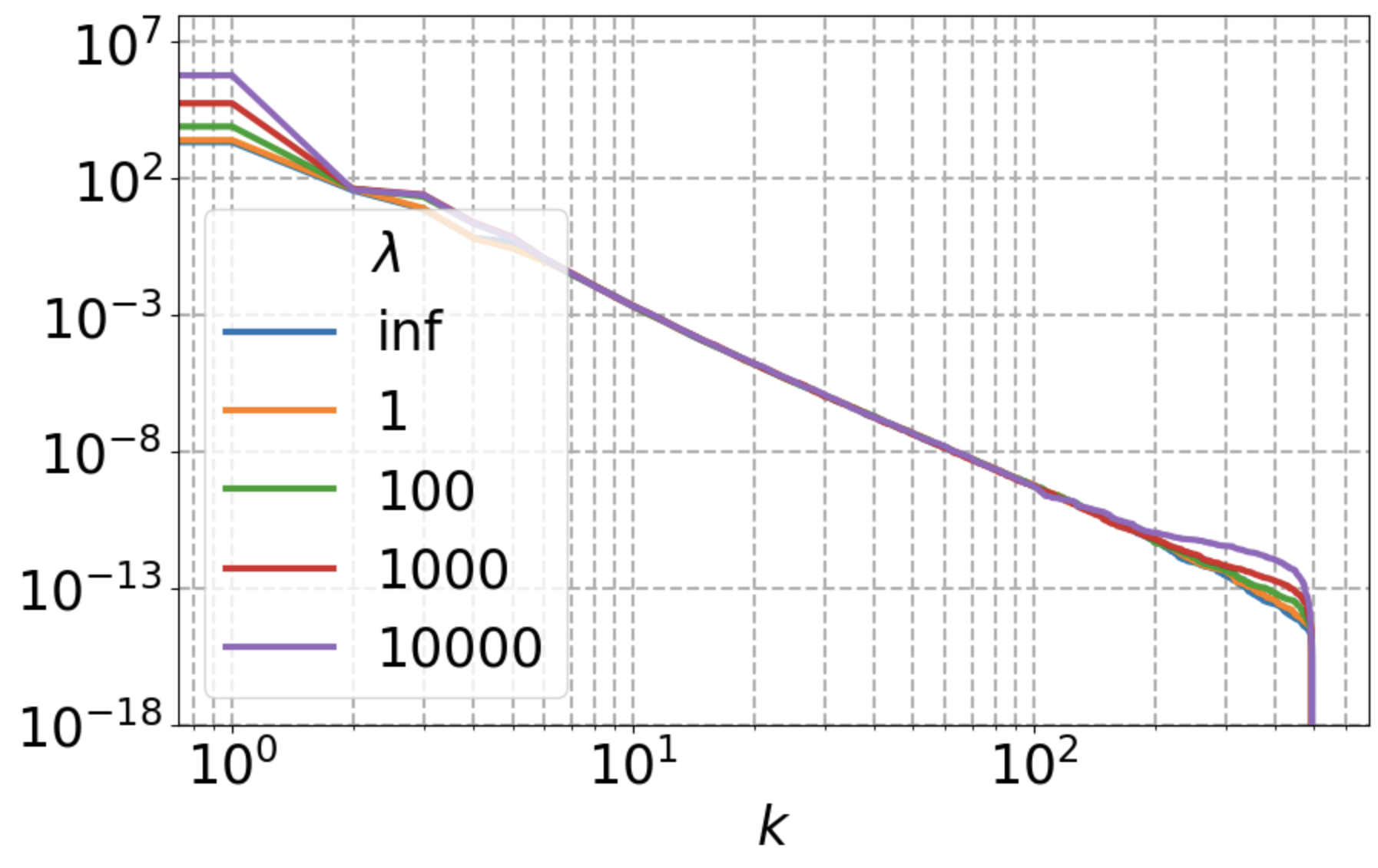}
\end{tabular}
\caption{Spectrum of $\bR_F$ in~\eqref{eq_optimality_regularization} with $\ReLU^2$, (a) $N=100$, (b) $N=300$, (c) $N=500$ and with varying values of regularization parameter $\lambda$. $\lambda = $\texttt{inf} corresponds to the case when the boundary condition is imposed as a constraint. The biases are uniform samples from $(-1,1)$ and weights are either $1$ or $-1$ with equal probability.}\label{R_F_spectrum}
\end{figure}

In practice, exact boundary conditions can be difficult to enforce when NNs are used, especially when the boundary geometry is complex. In the following, we study the solutions obtained by solving~\eqref{eq_regularization_quad}, which corresponds to imposing the boundary conditions as regularization. The optimality condition for the solution $\ba_F^*$ in this case is
\begin{equation}\label{eq_optimality_regularization}
\bR_F\ba_F^*:=(\bG_F + \lambda\bB^\top\bB)\ba^*_F=\by_F +\lambda\bB^\top\bc\;.
\end{equation}
The matrix $\bR_F$ is positive definite, and by Weyl's inequality, we have
\begin{equation}
\lambda_i(\bG_F) + \lambda\cdot\lambda_{\min}(\bB^\top\bB) \leq \lambda_i(\bR_F)\leq \lambda_i(\bG_F) + \lambda\cdot\lambda_{\max}(\bB^\top\bB)
\end{equation}
where $ \lambda_{\max}(\bB^\top\bB)$ and $ \lambda_{\min}(\bB^\top\bB)$  are the maximum and minimum eigenvalues of $\bB^\top\bB$, respectively, and $\lambda_i(\bG_F)$ is the $i$-th eigenvalue of $\bG_F$ in the descending order.  
This  shows that the spectrum of $\bR_F$ decays similarly as $\bG_F$, which suggests that regularized PINN and DRM have the same asymptotic frequency bias as their respective counterparts with constrained boundary conditions. However,  we note that this boundary regularization introduces additional subtleties in balancing the boundary error and interior error to achieve overall accuracy.

In Figure~\ref{R_F_spectrum}, we plot the spectrum of $\bR_F$ with $\ReLU^2$ activation, varying width $N$, and different values of the regularization parameter $\lambda$. We also include the case $\lambda=\infty$, which corresponds to $\bK_F$ arising from imposing the boundary condition as a constraint. The bases are randomly sampled from a uniform distribution in $(-1,1)$ and each weight takes either $1$ or $-1$ with equal probability. We observe that the decaying behaviors of $\bG_F$ are similar to those of $\bR_F$, suggesting that solving for $\ba_F^*$ directly from~\eqref{eq_optimality_regularization} can incur similar numerical issues as discussed in Section~\ref{sec_constraint_accuracy}. We also observe that as the regularization parameter $\lambda$ increases, the largest eigenvalue of $\bR_F$ grows, and the tails of the spectrum only lift slightly. This suggests that imposing stronger regularization with a larger $\lambda$ leads to a more ill-conditioned $\bR_F$, posing a numerical challenge in balancing the interior and boundary errors, which ultimately determine the overall accuracy. The stiffness of~\eqref{eq_optimality_regularization} can become more severe when using wider networks.

\begin{figure}
\centering
\begin{tabular}{c@{\vspace{2pt}}c@{\vspace{2pt}}c}
(a)&(b)&(c)\\
\includegraphics[width=0.33\textwidth]{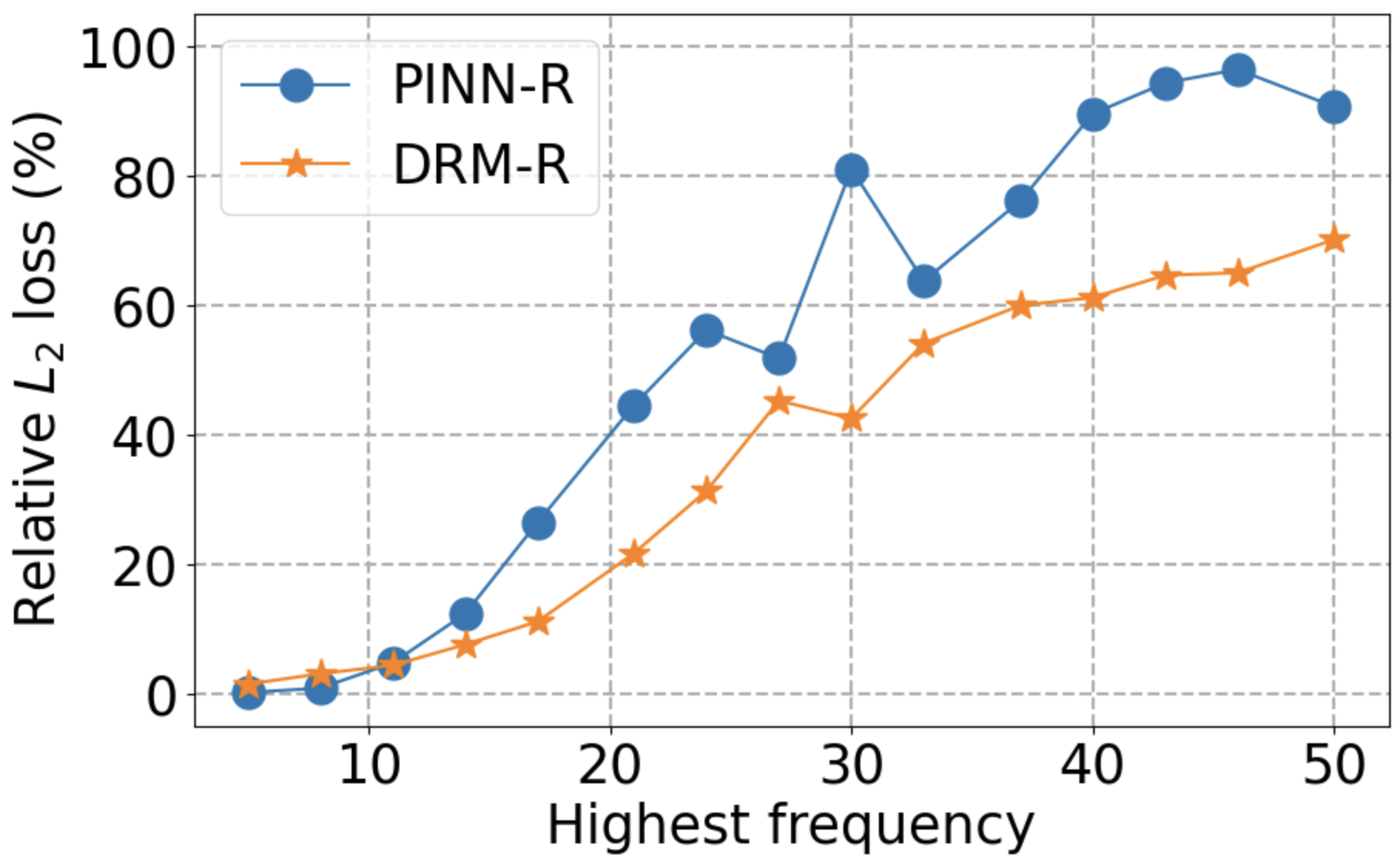}&
\includegraphics[width=0.33\textwidth]{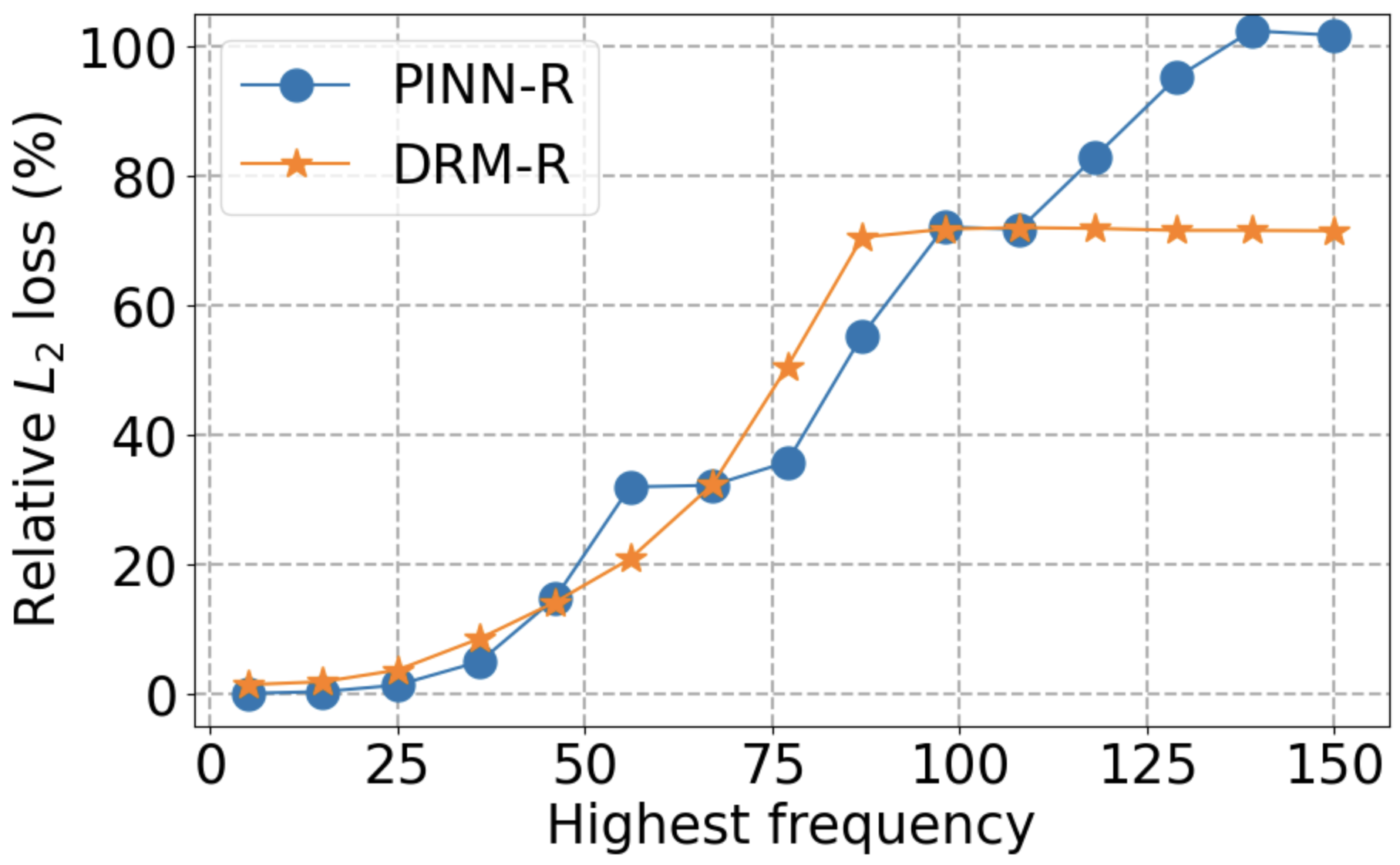}&
\includegraphics[width=0.33\textwidth]{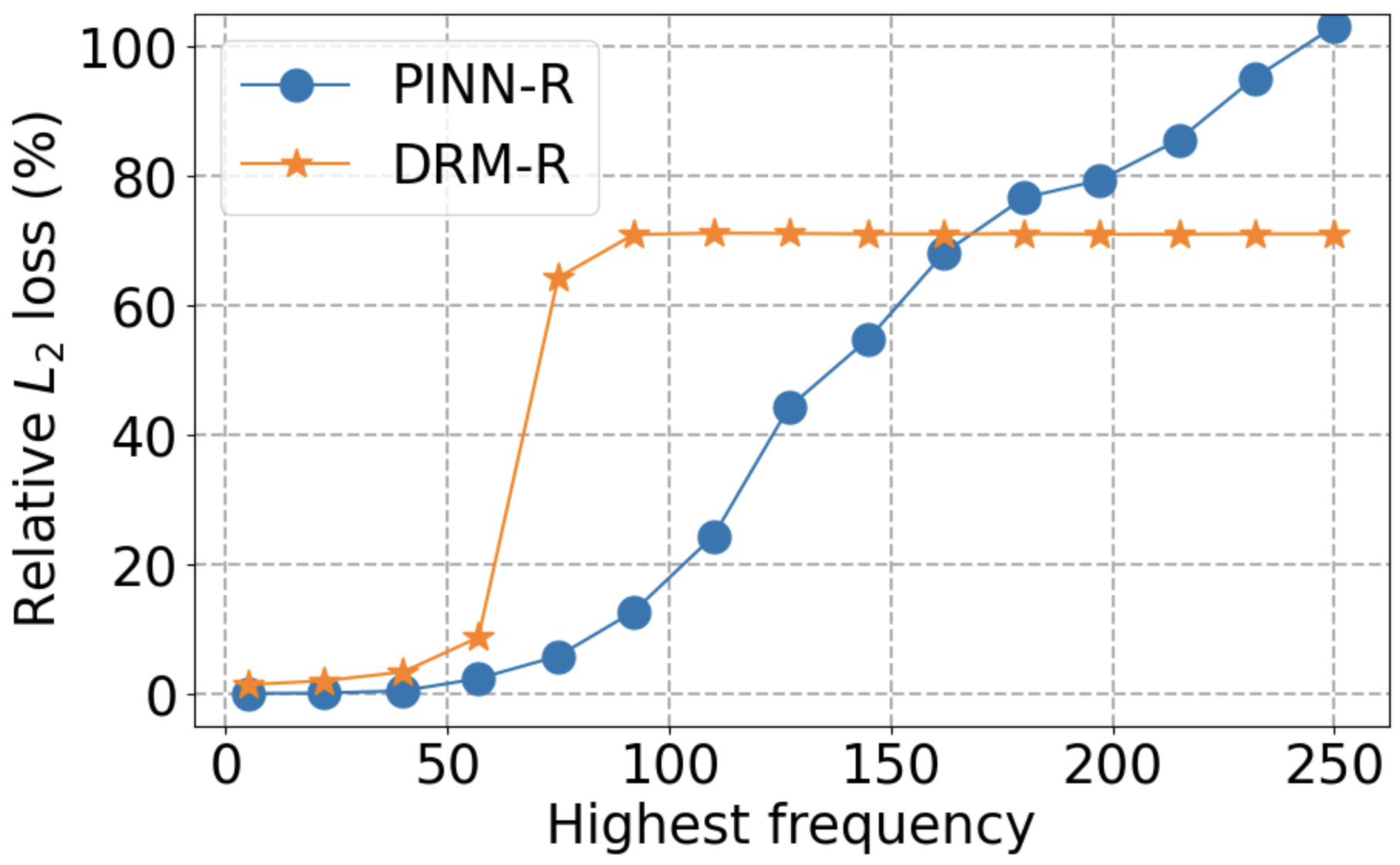}
\end{tabular}
\caption{Relative $L_2$ error of the solutions by PINN and DR when the number of bases is (a) $N=100$, (b) $N=300$, (c) $N=500$, as the highest frequency $k_{\max}$ of the solution $\sin(2\pi x+3\pi/5)+\sin(k_{\max}\pi x-2\pi/3)$ increases. The Dirichlet boundary conditions are imposed as regularization. For PINN, we use $\ReLU^3$, $\lambda = 0.1k_{\max}$,  and for DRM, we use $\ReLU^2$, $\lambda = 50k_{\max}$.}\label{acc_PINN_DRM_reg}
\end{figure}

In Figure~\ref{acc_PINN_DRM_reg}, we show the relative $L_2$ error of the solutions by PINN and DRM where the boundary condition is imposed as regularization. The network width is $N=100$ in (a), $N=300$ in (b), and $N=500$ in (c), and the solution is $\sin(2\pi x + 3\pi/5) + \sin(k_{\max}\pi x -2\pi /3)$ where $k_{\max}$ is the highest frequency varying from $5$ to $N/2$. For PINN, we use $\ReLU^3$ as the activation function and set $\lambda=0.1k_{\max}$; and for DRM, we use $\ReLU^2$ as the activation function and set  $\lambda = 100k_{\max}$.  In (a)-(c), we observe that both PINN and DRM perform similarly when $k_{\max}$ is small, and as $N$ increases, the accuracy of both methods generally improves. 
In addition to errors caused by inexact boundary conditions, as the highest frequency $k_{\max}$ increases, the approximation accuracy deteriorates; this behavior is consistent with that observed in previous tests when the boundary condition is enforced as a constraint.

We further examine the effect of the regularization parameter $\lambda$ on both PINN and DRM and compare them to the results with boundary conditions enforced as constraints. For PINN, we use $\ReLU^p$ activation functions, while for DRM, we use $\ReLU^{p-1}$ so that their corresponding Gram matrices have the same asymptotic spectral decay rate. In both methods, the two-layer networks contain $N=300$ neurons, and the underlying solution is fixed as $\sin(2\pi x + 3\pi/5) + \sin(25\pi x -2\pi /3)$. We measure the relative $L_2$ error of the solution approximation for different values of $\lambda$ and $p$, reporting the results in Figure~\ref{lamda_effect_PINN_DRM}.

For $p=2$, we observe that PINN basically failed across all regularization parameters. This is because PINN uses a strong formulation while $\ReLU^2$ does not have a continuous second derivative.
While DRM using $\ReLU$ employs a weak formulation. As we can see, when  $\lambda$ increases, the boundary condition is enforced more accurately, eventually leading to a numerical result consistent with that of treating the boundary condition as a constraint. 

When we increase the regularity to $p=3$, we observe that PINN performs satisfactorily over a wide range of the regularization parameter $\lambda$, whereas DRM is only effective within a specific range. We also see that the regularization controlled by $\lambda$ introduces a bias. As shown in Figure~\ref{R_F_spectrum},  as $\lambda$ becomes large, the first two leading singular values increase drastically. Moreover, the use of a smoother of $\ReLU$ power activation makes the spectral decay faster ($\Theta\left(k^{-4}\right)$). The combination of these two factors makes the numerical results more sensitive to the choice of $\lambda$. The above phenomena are corroborated when taking $p=4$ as shown in Figure~\ref{lamda_effect_PINN_DRM}(c).  In this case, the spectral decay of the Gram matrix is even faster ($\Theta\left(k^{-6}\right)$), which makes the dependence on $\lambda$ even more sensitive. 

In these tests, one also observes that DRM is more sensitive to the regularization constant for boundary conditions than PINN. This minimization problem~\eqref{eq_DRM_in_loss} is equivalent to the weak formulation of the original PDE when optimizing among all functions in $H_1[-1,1]$ that satisfy the Dirichlet boundary condition. In classical FEM, the boundary condition is explicitly enforced. However, in DRM using boundary regularization, the boundary condition is part of the optimization. In other words, it is not guaranteed that one is minimizing the loss function among all NN representations that satisfy the Dirichlet boundary condition. 

Due to this fact, one may be inclined to increase $\lambda$ to emphasize the boundary condition strongly. However, as discussed above, the strong bias introduced by the boundary condition and the ill-conditioning of the system may lead to increasing numerical errors. The numerical results in Figure~\ref{lamda_effect_PINN_DRM}(c) indicate that the DRM requires a sufficiently large $\lambda$ to enforce the boundary conditions for accurate results. However, an excessively large $\lambda$ can also introduce significant error, particularly in more ill-conditioned systems. This makes DRM more sensitive to the choice of $\lambda$ compared with PINN.

\begin{figure}
\centering
\begin{tabular}{c@{\vspace{2pt}}c@{\vspace{2pt}}c}
(a)&(b)&(c)\\
\includegraphics[width=0.33\textwidth]{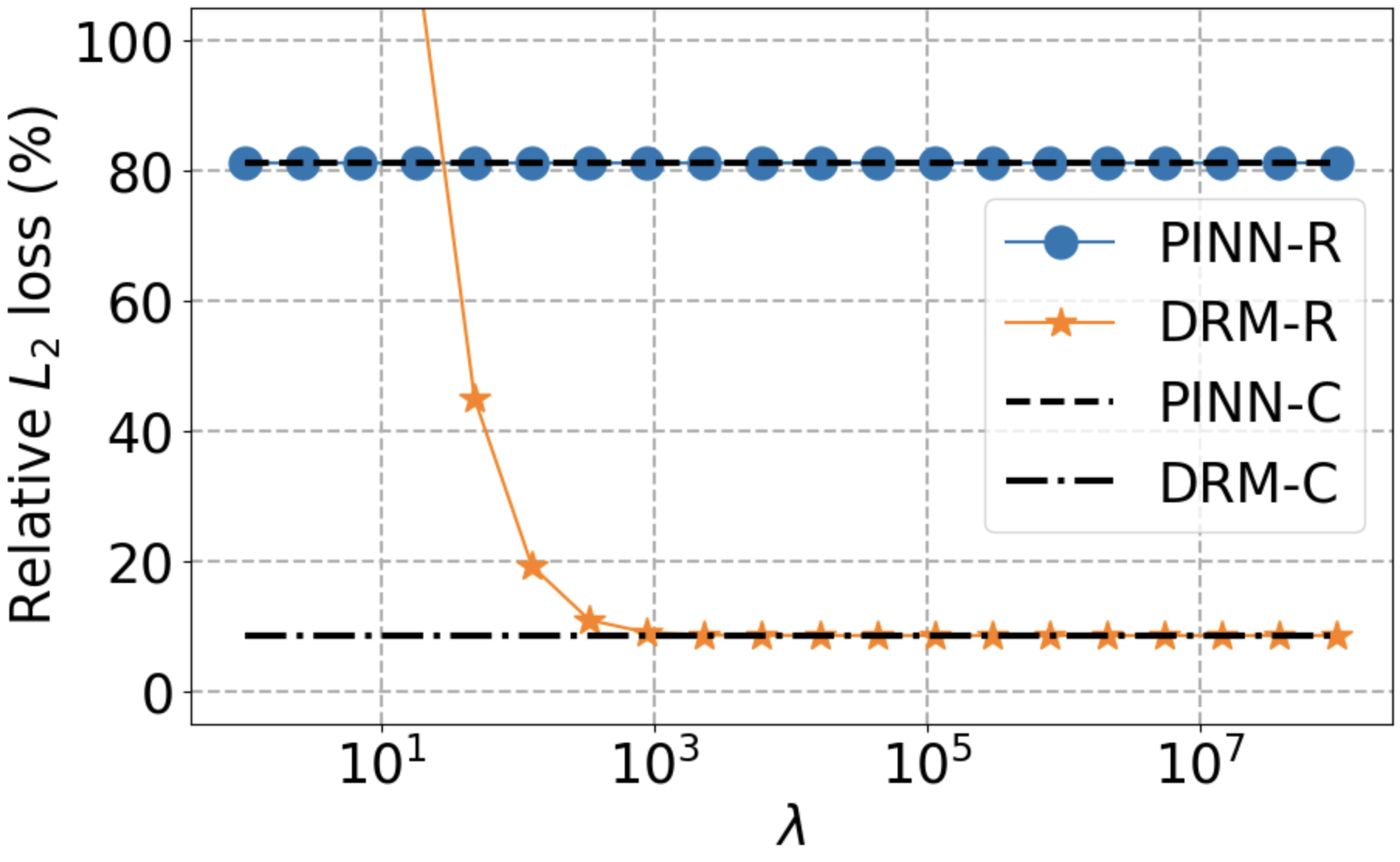}&
\includegraphics[width=0.33\textwidth]{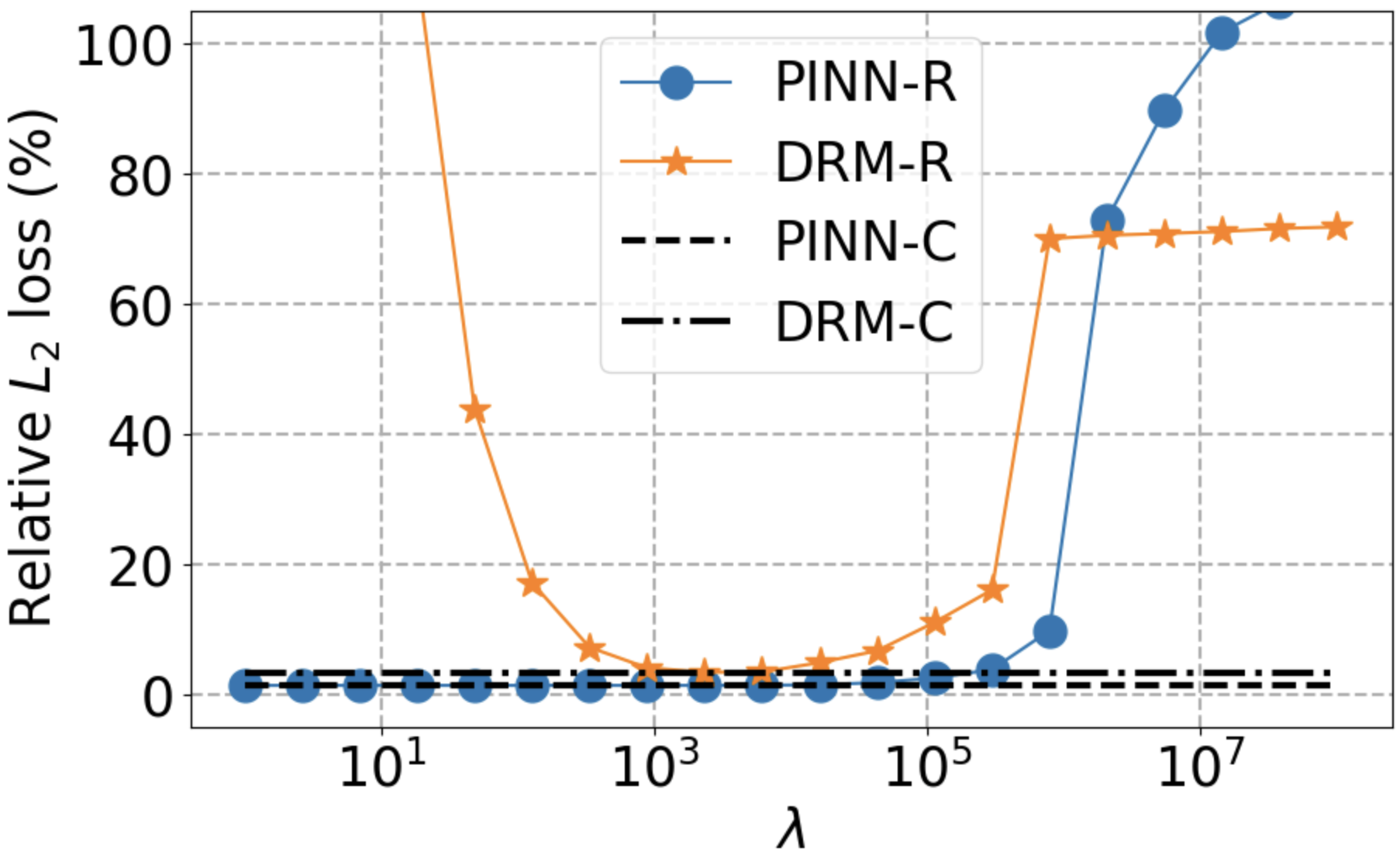}&
\includegraphics[width=0.33\textwidth]{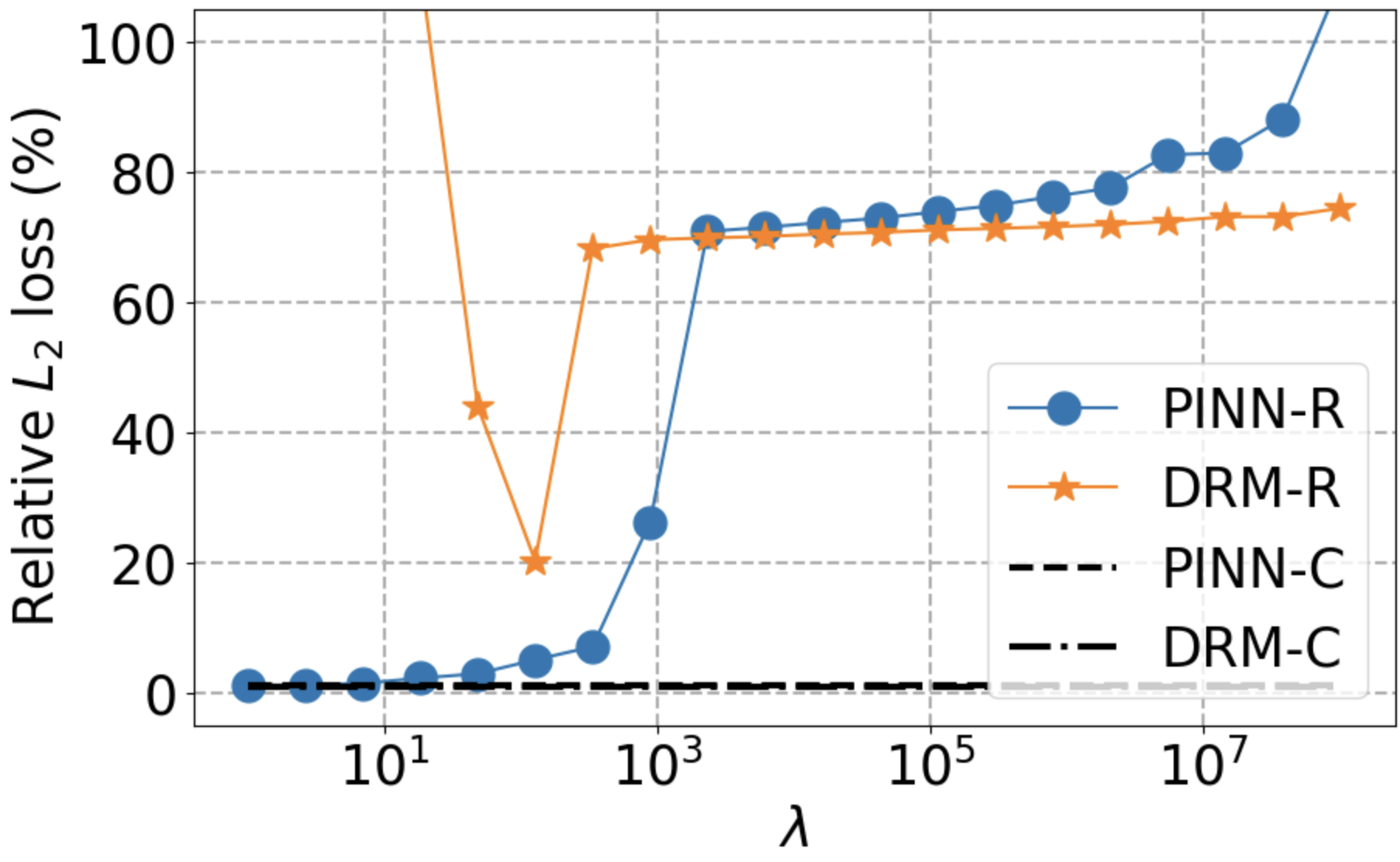}
\end{tabular}
\caption{Relative $L_2$ error of the solution approximation by PINN and DRM when the regularization parameter changes. For PINN, we use $\ReLU^{p}$ activation, and for DRM, we use $\ReLU^{p-1}$, where (a) $p=2$, (b) $p=3$, and (c) $p=4$.  The dashed lines show the relative $L_2$ errors for the solutions obtained when the boundary condition is imposed as a constraint.  The network width $N=300$ and the solution $\sin(2\pi x+3\pi/5)+\sin(25\pi x-2\pi/3)$ are fixed. In this figure, results labeled by PINN-R and DRM-R are obtained by boundary regularization, and the results by PINN-C and DRM-C are obtained by boundary constraint.}\label{lamda_effect_PINN_DRM}
\end{figure}

\subsection{Convergence of iterative solvers}\label{sec_convergence_bd_constraint}

\begin{algorithm}[t]
\caption{Projected Gradient Descent (PGD)}\label{algo_pgd}
\KwData{Learning rate $\gamma > 0$}
\KwResult{$\ba^* = \ba^{T_{\max}}$}
$\ba^{0}\sim \cN(0, \sqrt{2/N})$\;

\For{$t = 0,1,\dots, T_{\max}-1$}{
    $\bz^{t} = \ba^t -\gamma(\bG_F\ba^t-\by_F)$\;

    $\ba^{t+1} = \bz^t-\bB^\top(\bB\bB^\top)^{-1}(\bB\bz^t-\bc)$
}
\end{algorithm}

In practice, finding the solution approximation via PINN or DRM relies on first-order methods such as Stochastic Gradient Descent (SGD). We demonstrate the effect of the spectral decay and frequency bias of Gram matrices on the behavior of the gradient descent.  We focus on the case where the boundary condition is imposed as constraint~\eqref{eq_constrained_quad} first and consider the projected gradient descent (PGD) where only $\ba$ is optimized.

For the constrained optimization~\eqref{eq_constrained_quad} the PGD described in Algorithm~\ref{algo_pgd}, each iteration can be reduced to a single updating formula
\begin{equation}
\ba^{t+1} = \ba^t -\gamma\left(\bI-\bB^\top(\bB\bB^\top)^{-1}\bB\right)(\bG_F\ba^t-\by_F)\;,
\end{equation}
and if we denote $\be^t = \ba^t - \ba^*$ where $\ba^*$ is the solution to~\eqref{eq_constrained_quad}, then we have
\begin{equation}\label{eq_et_evolution}
\be^{t+1} = \left(\bI-\gamma\bP\bG_F \bP\right)\be^t\;.
\end{equation}
where $\bP:= \bI-\bB^\top(\bB\bB^\top)^{-1}\bB$ is the projection to the null space of $\bB$. Hence,
\begin{equation}
\|\be^{t+1}\|\leq \max\{|1-\gamma\lambda_{\max}(\bP\bG_F\bP)|,|1-\gamma\lambda_{\min}(\bP\bG_F\bP)| \}\cdot \|\be^t\|\;.
\end{equation}
As $\bB$ has rank $2$, $\lambda_{\min}(\bP\bG_F\bP) = 0$. If $0<\gamma<2/\lambda_{\max}(\bP\bG_F\bP)$, then PGD converges with rate $|1-\gamma\lambda_{\max}(\bP\bG_F\bP)|<1$. Since the dimension of the null space of $\bB$ is $N-2$, as $N$ gets sufficiently large, we expect  $\bP$ to be approximated by an identity matrix $\bI\in\mathbb{R}^{N\times N}$. More precisely, we have
$$
\mathbb{E}_{\bx \sim\cU(\mathbb{S}^{N-1})}\|(\bI - \bP)\bx\|_2 = \sqrt{2/N}
$$
where $\cU(\bbS^{N-1})$ is the uniform distribution over the unit sphere $\bbS^{N-1}$.  Suppose $\bG_F = \bU_F\bLambda_F\bU_F^\top$ is the eigendecomposition of $\bG_F$, then we can deduce from~\eqref{eq_et_evolution} the relation
\begin{equation}\label{eq_eigen_constraint}
\bU_F^\top\be^{t+1} = (\bI-\gamma\bU_F^\top\bP\bU_F\bLambda_F\bU_F^\top\bP\bU_F)\bU_F^\top\be^t 
\end{equation}
which is approximately $(\bI-\gamma\bLambda_F)\bU_F^\top\be^t$ as $N\to\infty$. This observation reveals how approximation errors evolve within the eigenspace of the Gram matrix $\bG_F$ for sufficiently wide networks.  For instance, the error at iteration $t+1$ projected onto the $j$-th eigenvector of $\bG_F$ depends approximately on the error at iteration $t$ projected onto the $j$-th eigenvector, scaled by $1-\gamma\lambda_j$, where $\lambda_1\geq \lambda_2\geq \cdots\geq\lambda_N$ are the eigenvalues of $\bG_F$. Combined with the spectral decay (Theorem~\ref{theorem_sharp_eigenvalue}) and characterization of the eigenfunctions (Proposition~\ref{prop_decaying}), this indicates that errors in high frequency decays very slowly during the iterations. Hence, the ill-conditioning and frequency bias against high frequency in the representation lead to slow learning dynamics for high frequency components.
Moreover, using $\ReLU^p$ with a larger power $p$ leads to a faster spectral decay of the Gram matrix, a stronger frequency bias, and consequently, slower learning dynamics for high-frequency components.

\begin{figure}
\begin{center}
\begin{tabular}{c@{\vspace{2pt}}c@{\vspace{2pt}}c}
(a)&(b)&(c)\\
\includegraphics[width=0.33\textwidth]{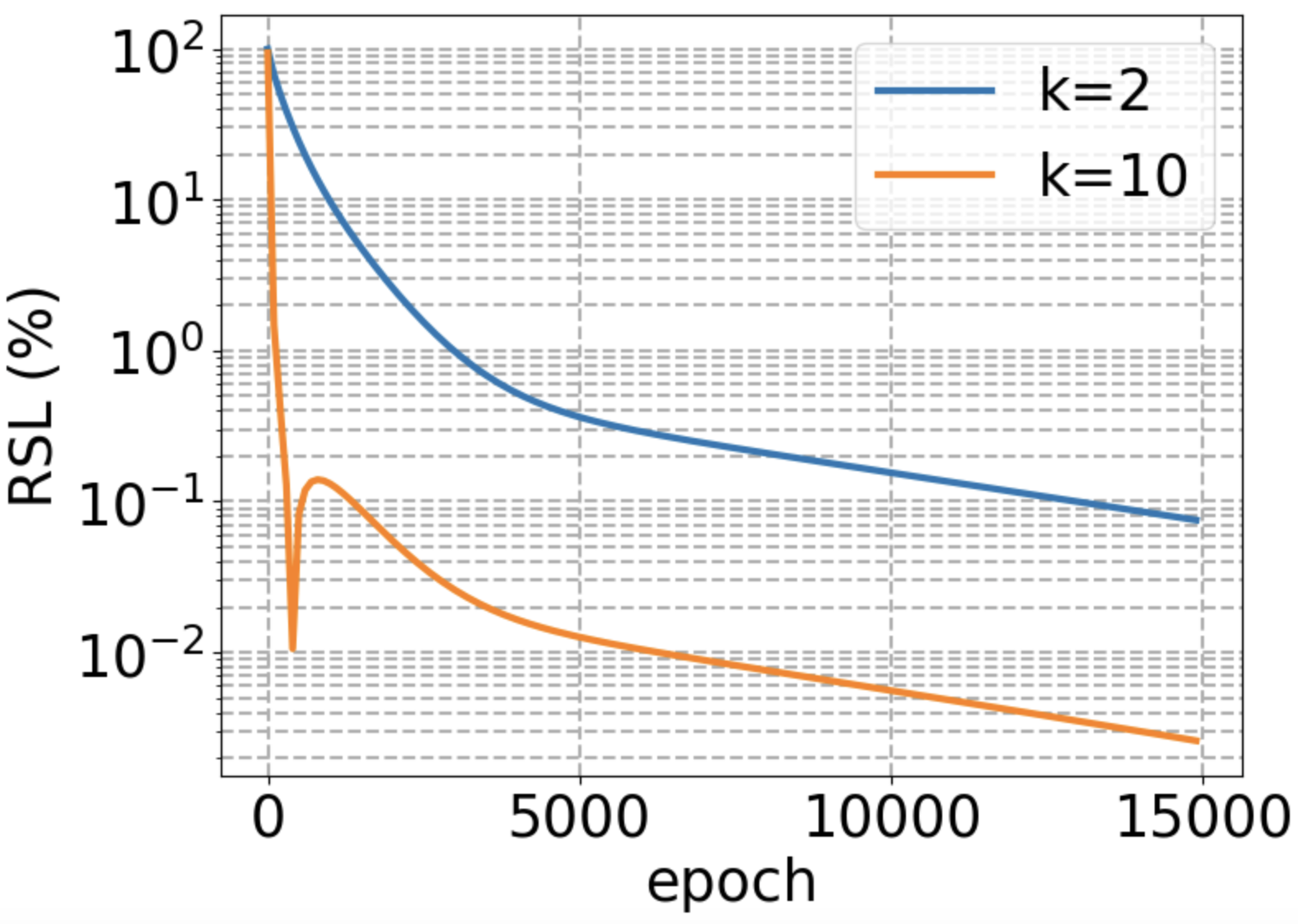}&
\includegraphics[width=0.33\textwidth]{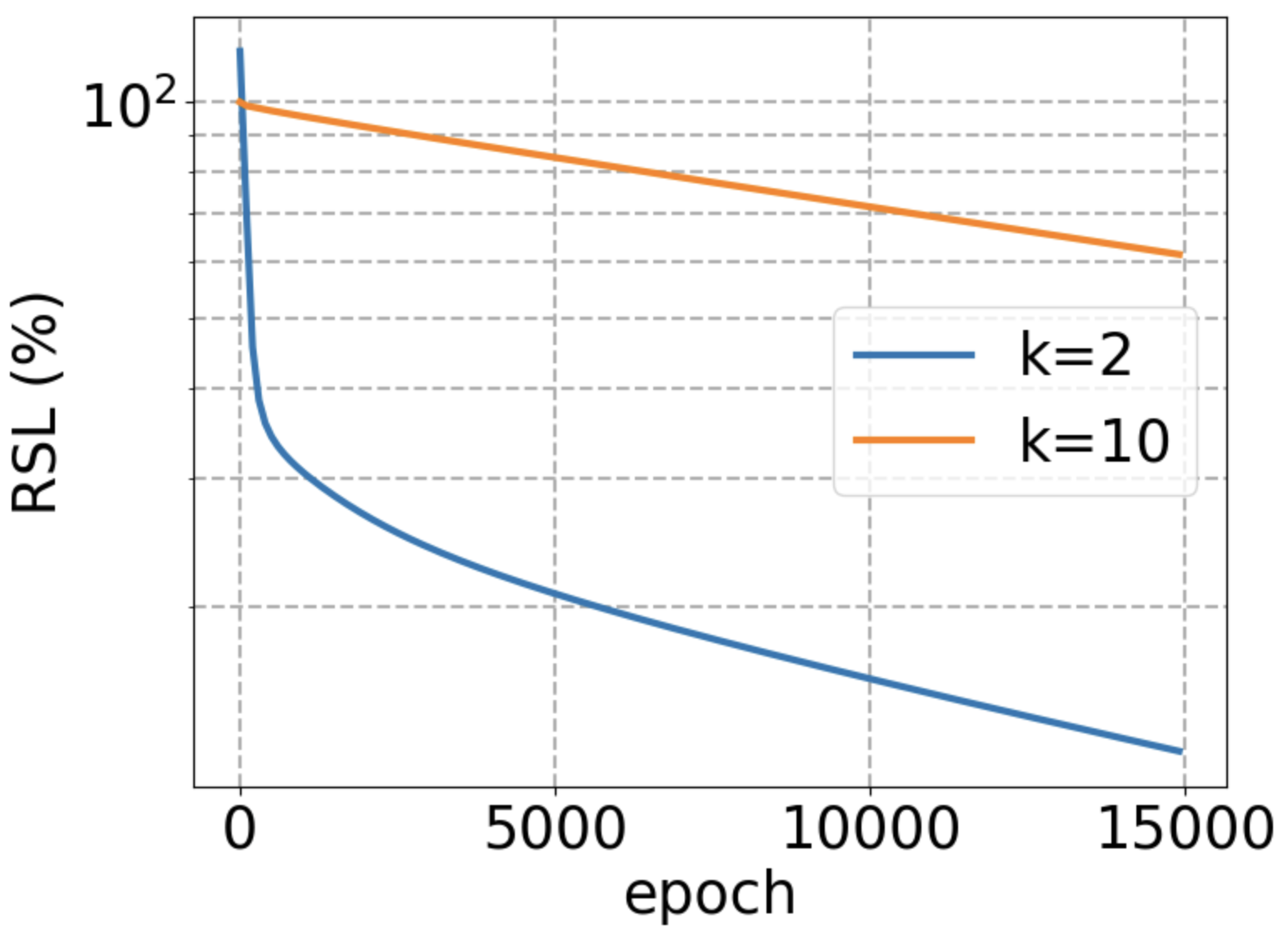}&
\includegraphics[width=0.33\textwidth]{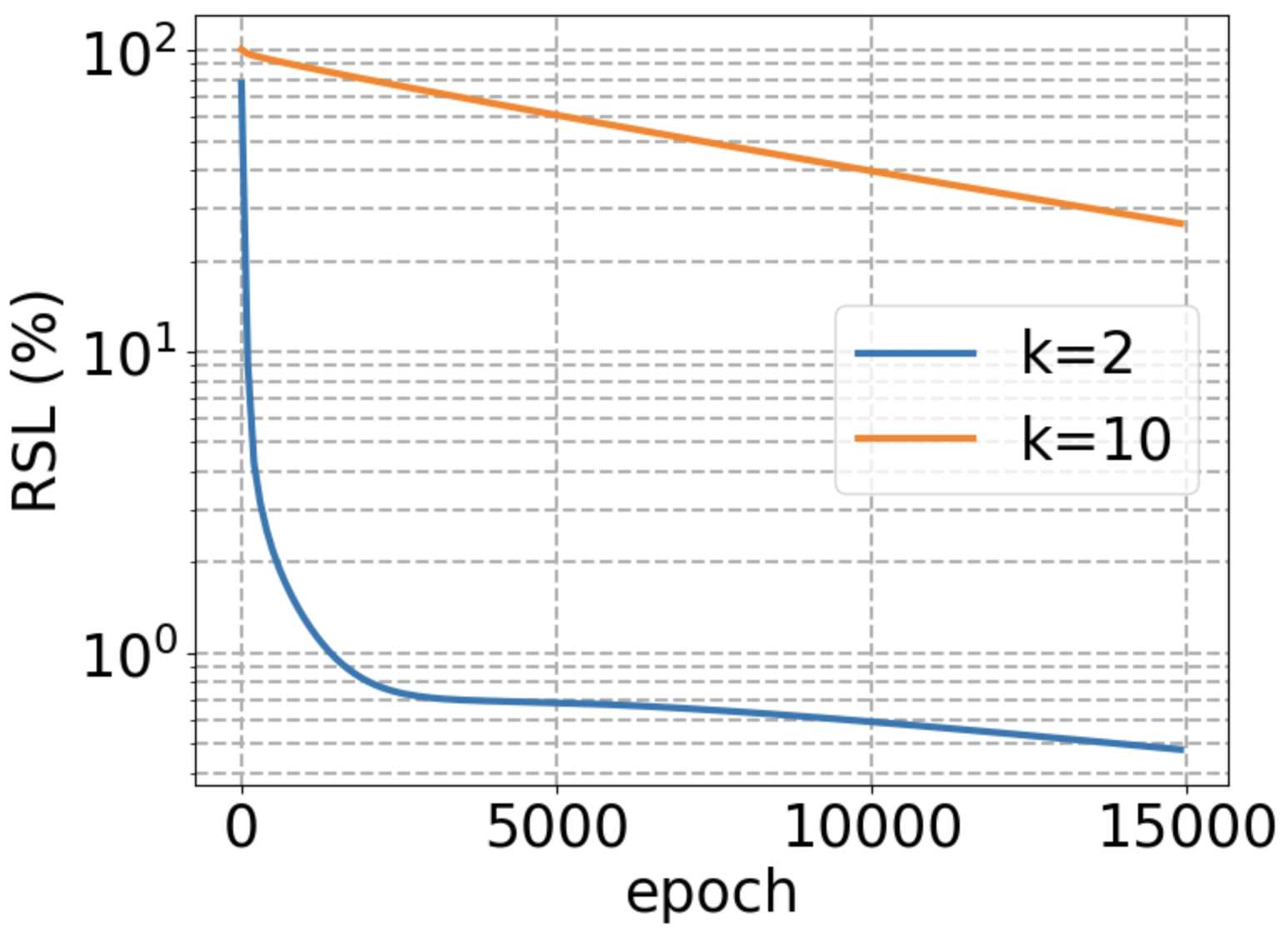}
\end{tabular}
\end{center}
\caption{Relative spectral loss~\eqref{eq_RSL} in the PGD solutions of (a) quadratic spline-based  FEM, (b) PINN  with $\ReLU^3$, and (c) DRM with $\ReLU^2$. The solution is $\sin(2\pi x)+\sin(10\pi x)$. For all methods, the learning rate is set as $1.99/\lambda_{\max}(\bP\bG_F\bP)$ and the number of bases is $N=200$. }\label{eq_constrained_bd_spectral_bias}
\end{figure}

Figure~\ref{eq_constrained_bd_spectral_bias} evaluates the RSL defined in~\eqref{eq_RSL} throughout the PGD iterations for three methods: (a) quadratic spline-based FEM, (b) PINN with $\ReLU^3$, and (c) DRM with $\ReLU^2$. The PDE solution is set to $\sin(2\pi x)+\sin(10\pi x)$. All methods use a learning rate of $\gamma = 1.99/\lambda_{\max}(\bP\bG_F\bP)$ and a number of bases (or network width) of $N=200$. For FEM, we observe that the high-frequency error decreases faster than the low-frequency error. In contrast, the low-frequency errors for both PINN and DRM decrease faster than the high-frequency errors, verifying our analysis. Moreover, since the spectral decay of Gram matrices for both PINN and DRM ($\Theta\left(k^{-4}\right), k=1,2, \ldots, N$, $k$ is the $k$-th mode) is faster than that for FEM ($\Theta\left((\frac{k}{N})^2\right)$), the error decay for PINN and DRM is slower than that for FEM using gradient descent. 

Here are the contrasting numerical implications:
\begin{itemize}
\item For the FEM method, although the representation has no frequency bias, the stiffness matrix inherits the spectral property from the differential operator, which results in ill-conditioning and frequency bias against low frequencies. If an iterative method is used, smooth components of the solution require many iterations to recover, which makes them useless in real practice if no preconditioning is applied.
\item 
For PINN or DRM, the corresponding system inherits the ill-conditioning and frequency bias against high frequency from the NN representation. When using the gradient-based method to solve the optimization problem, smooth components of the solution may be recovered quickly, while capturing the high-frequency components leads to a computational challenge. Moreover, since the Gram matrix, or the Jacobian of the loss function in general, is dense, it is very costly, if not impossible, to achieve high accuracy if the solution contains significantly high-frequency components.
    
\end{itemize}

\subsection{Neumann boundary condition}
\begin{figure}
\centering
\begin{tabular}{c@{\vspace{2pt}}c@{\vspace{2pt}}c}
(a)&(b)&(c)\\
\includegraphics[width=0.33\textwidth]{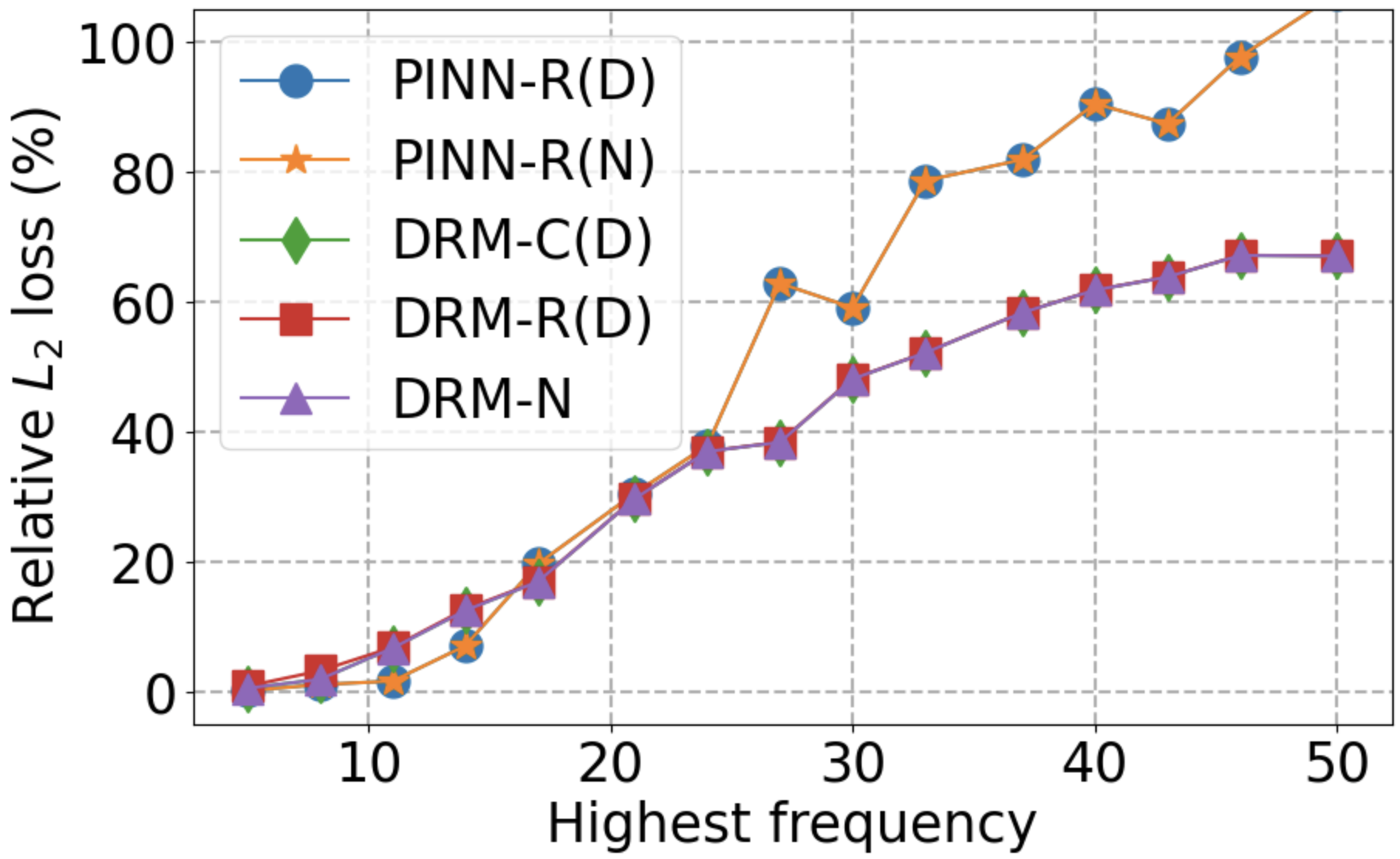}&
\includegraphics[width=0.33\textwidth]{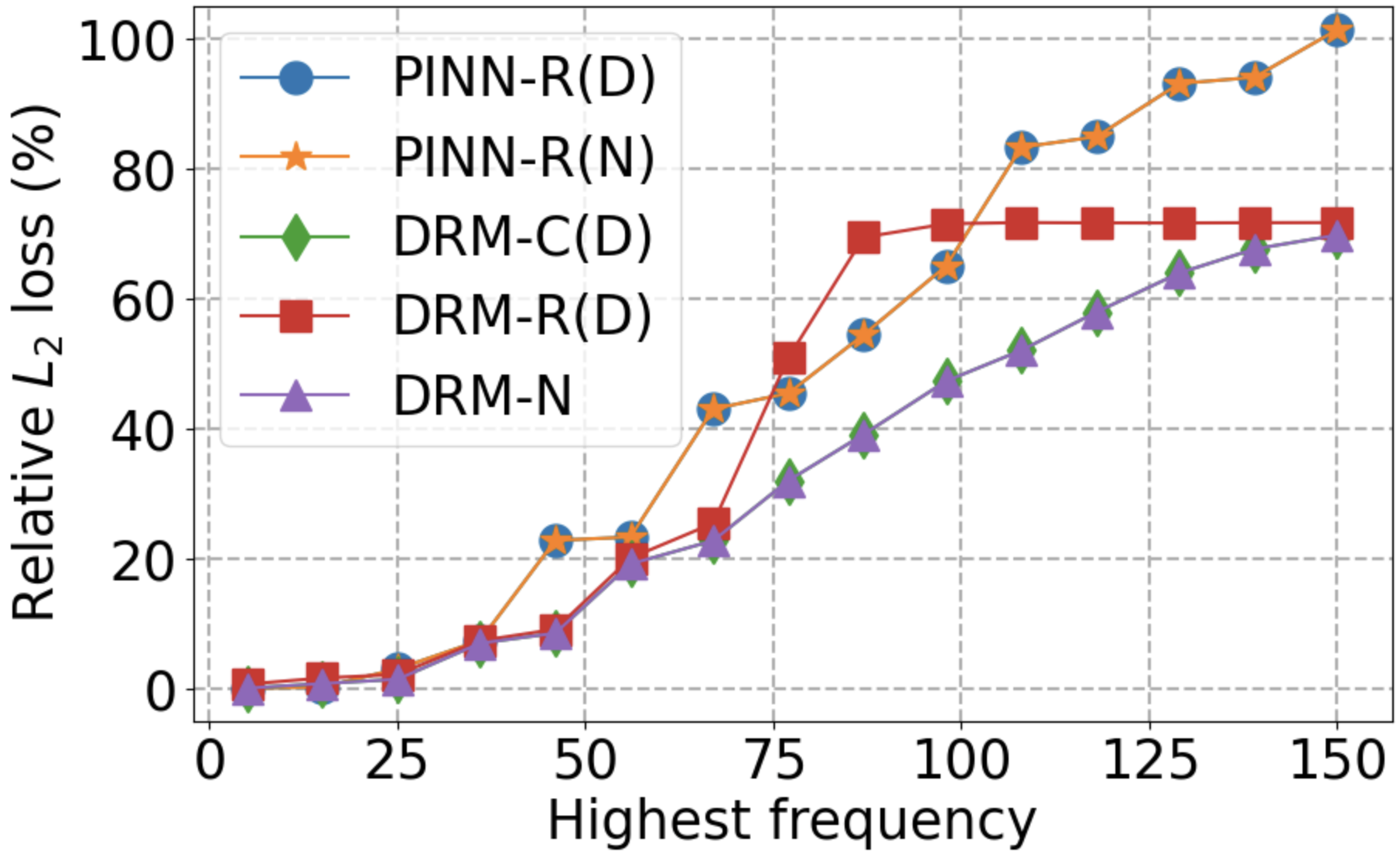}&
\includegraphics[width=0.33\textwidth]{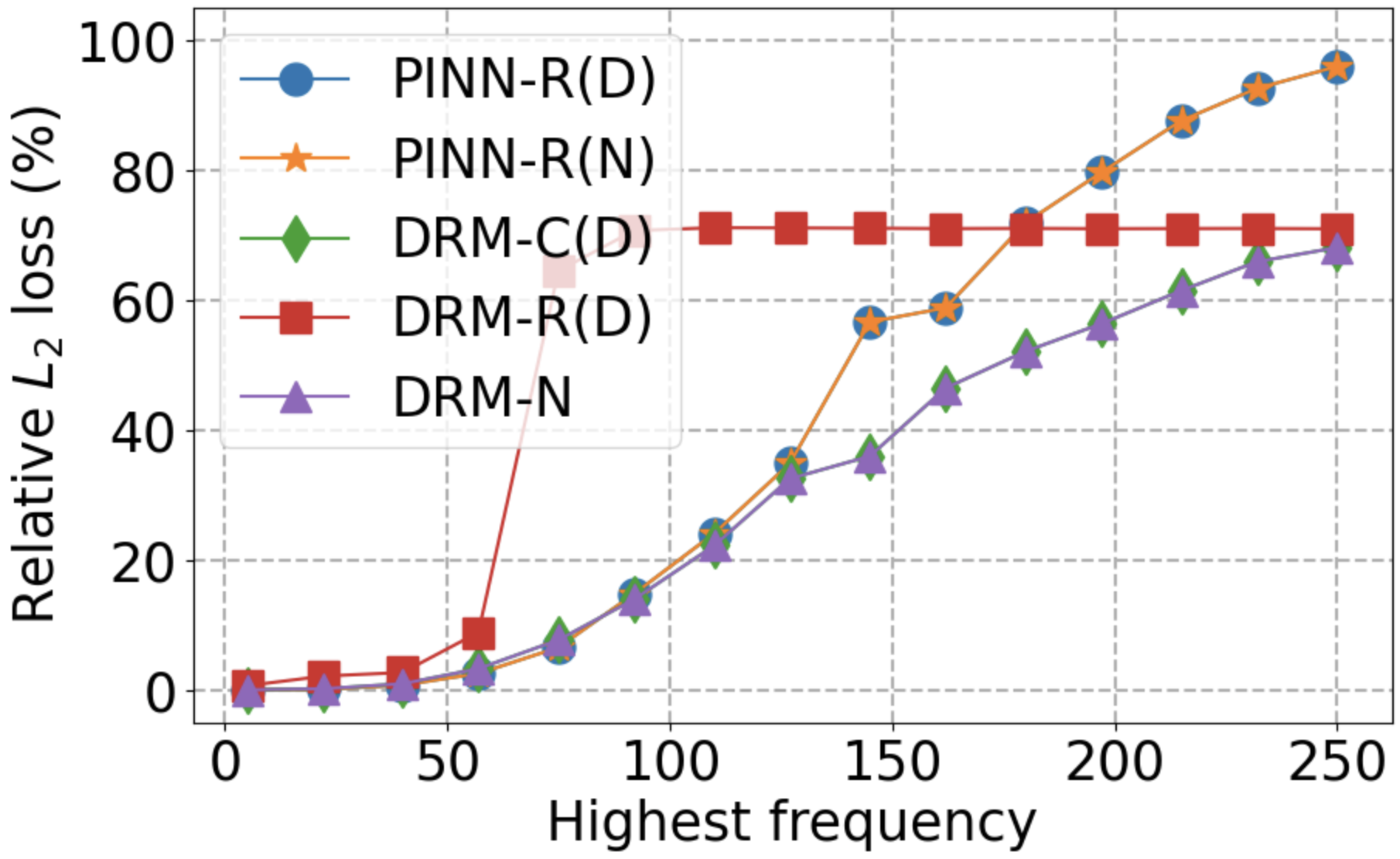}
\end{tabular}
\caption{Relative $L_2$ error of the solutions by PINN and DRM  with Dirichlet and Neumann boundary conditions when the number of bases is (a) $N=100$, (b) $N=300$, (c) $N=500$, as the highest frequency $k_{\max}$ of the solution $\sin(2\pi x+3\pi/5)+\sin(k_{\max}\pi x-2/3\pi)$ increases. For PINN-D and PINN-N, we use $\ReLU^3$ and impose the boundary conditions using regularization with $\lambda = 0.1k_{\max}$. For DRM-D, we use $\ReLU^2$ and impose the boundary conditions using regularization $\lambda = 50k_{\max}$. For DRM-N, the boundary condition is naturally included in the  loss function~\eqref{eq_neumann_drm}.}\label{acc_PINN_DRM_reg_neumann}
\end{figure}
In this section, we compare PINN and DRM for the Poisson equation with Dirichlet boundary conditions with their respective counterparts with the Neumann boundary condition. That is
\begin{equation}\label{eq_neumann_PDE}
\begin{aligned}
u_{xx} &= -f\\
u'(-1) = c'_L,&~u'(1)=c'_R 
\end{aligned}\;.
\end{equation}
In this case, the linear PINN  with boundary regularization is formulated as
\begin{equation}\label{eq_neumann_pinn}
\min_{\ba\in\mathbb{R}^N}\int_{-1}^1(u_{xx}(x,\ba) + f(x))^2\,dx+\lambda\left((u'(-1,\ba) - c'_L)^2 + (u'(1,\ba) - c'_R)^2 \right)\;.
\end{equation}
In the variational form of~\eqref{eq_neumann_PDE}, the Neumann boundary  naturally appears, and the solution approximation becomes an unconstrained optimization problem without any regularization parameter:
\begin{equation}\label{eq_neumann_drm}
\min_{\ba\in\mathbb{R}^N}\frac{1}{2}\int_{-1}^1u_x^2(x,\ba)\,dx -\int_{-1}^1u(x,\ba)f(x)\,dx - c'_Ru(1,\ba) +c'_Lu(-1,\ba)\;.
\end{equation}
In Figure~\ref{acc_PINN_DRM_reg}, we compare PINN and DRM with different widths, different boundary conditions, and varying the highest frequency in the underlying solution. The  solution is $\sin(2\pi x+\pi/5)+\sin(k_{\max}\pi x+\pi/3)$, and $k_{\max}$ changes. We test the following cases: (i) PINN with Dirichlet boundary regularization, denoted by PINN-R(D); (ii) PINN with Neumann boundary regularization, denoted by PINN-R(N); (iii) DRM with Dirichlet boundary constraint, denoted by DRM-C(D); (iv) DRM with Dirichlet boundary regularization, denoted by DRM-R(D); and (v) DRM with Neumann boundary condition~\eqref{eq_neumann_drm}, denoted by DRM-C(N). Since there is a constant shift when the Neumann condition is used, for this set of experiments, we evaluate the approximations after subtracting the mean. We note that in all of the tested cases, PINN-R(D) and PINN-R(N) perform almost identically using the same regularization parameter $\lambda$. As for DRM, when $N=100$ as shown in (a), we see that different paradigms of DRM perform similarly. When $N$ gets larger, e.g., $N=300$ in (b) and $N=500$ in (c), we observe that DRM-C(D) and DRM-N remain effective and identical, but DRM-R(D) fails to approximate the solution well if $k_{\max}$ is high, while DRM-N consistently achieves better accuracy.  

This experiment demonstrates that for the DRM, imposing Dirichlet boundary conditions via regularization is generally ineffective. This is because with regularization, the minimization of~\eqref{eq_DRM_in_loss} can generally take place in wrong function spaces, and consequently, the result is not guaranteed to approximate the solution of the Poisson equation.  This problem does not occur for the DRM with hard Dirichlet constraints or with Neumann boundary conditions, which are incorporated naturally. These results reinforce the general conclusion that the DRM's accuracy is highly sensitive to the boundary regularization parameter.

\section{Numerical Perspectives of General Cases}\label{sec_other_activate}

Our previous experiments and analysis focused on the well-controlled case where  SNNs are formulated as a linear representation. In particular, with ReLU power as the activation function, the spectral properties of the Gram matrices associated with PINN and DRM can be explicitly characterized. In this section, we generalize our study numerically to other activation functions and in more general settings. More interestingly, we will conduct numerical experiments to show that using a fully nonlinear SNN representation with all parameters trainable can achieve interesting adaptivity. However, the computational results are not necessarily better than using a linear representation with a comparable total degrees of freedom. The costly, large-scale, non-linear optimization involved makes it less favorable to classical and well-developed linear methods such as FEMs.

\subsection{Shallow NNs with non-homogeneous activation}\label{sec_non_homogeneous_activation}

 In this section, we numerically investigate the use of other activation functions in SNNs. Most importantly, we highlight a property not shared by ReLU-type functions: \textit{scaling}. Note that  $\sigma = \ReLU^p$ is $p$-th degree homogeneous; that is, $\sigma(Sx) = S^p\sigma(x)$ for any $S>0$; thus means that a scaling of the variable can be absorbed in the amplitude, or in other words, scaling does not change the space spanned by the family of activation functions. In contrast,  using other more general activation functions such as $\sin$ and $\tanh$, the situation is different. As shown in \cite{zhang2025shallow}, appropriate scaling of the initial parametrization leads to a slower spectral decay of the leading singular values of the corresponding Gram matrix in a range proportional to the scaling.  As a result, it reduces the ill-conditioning and frequency bias and hence improves the representation capability of SNNs significantly, even as a linear representation using a random basis. However, over-scaling with respect to the network width may introduce fast transitions and/or high frequencies that cannot be represented well by the network and lead to unsatisfactory performances. We will show that this is also true when solving PDEs by applying scaling in the initial parametrization. 
 
\begin{figure}
\centering
\begin{tabular}{c@{\vspace{2pt}}c@{\vspace{2pt}}c}
(a)&(b)&(c)\\
\includegraphics[width=0.33\textwidth]{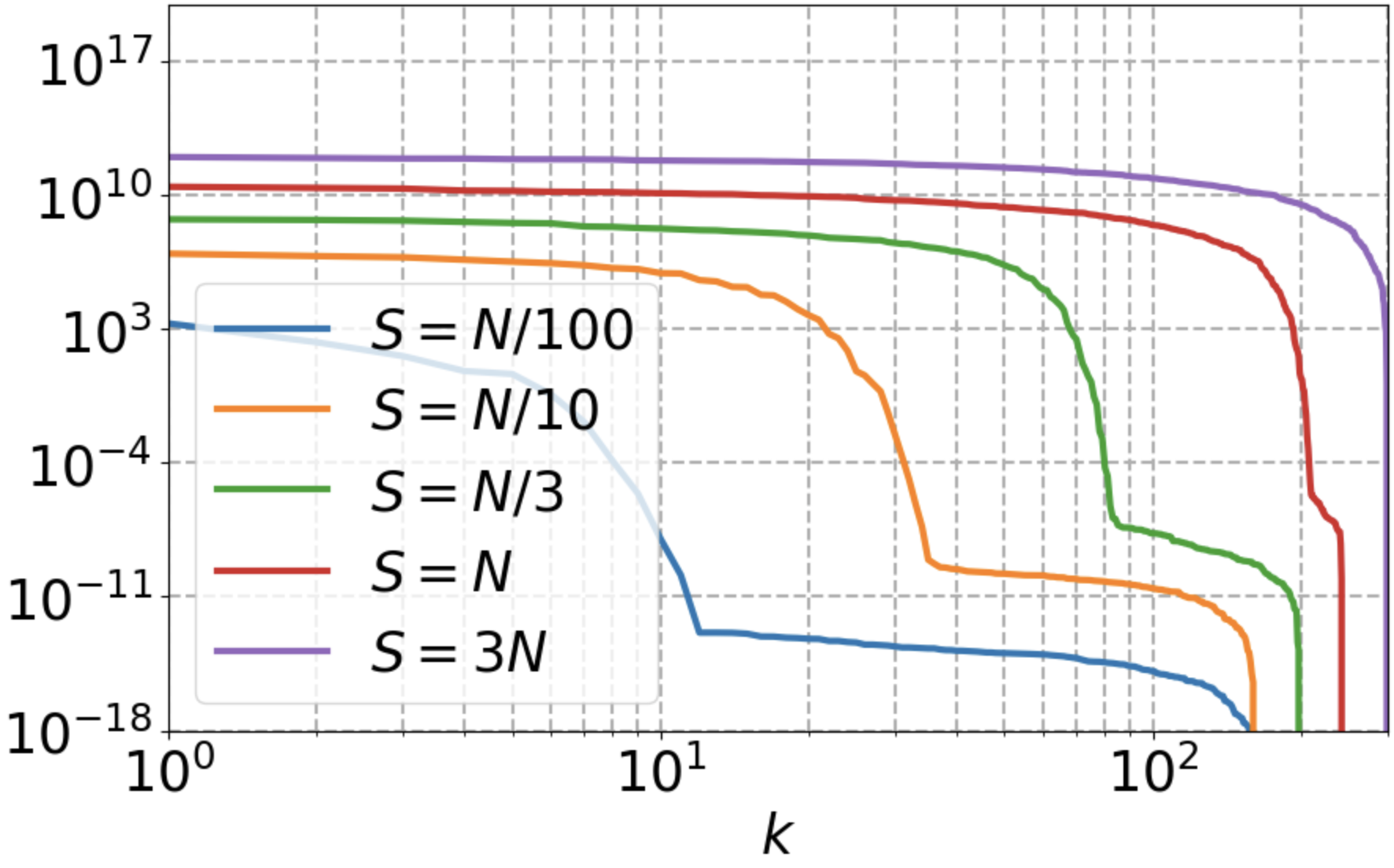}&
\includegraphics[width=0.33\textwidth]{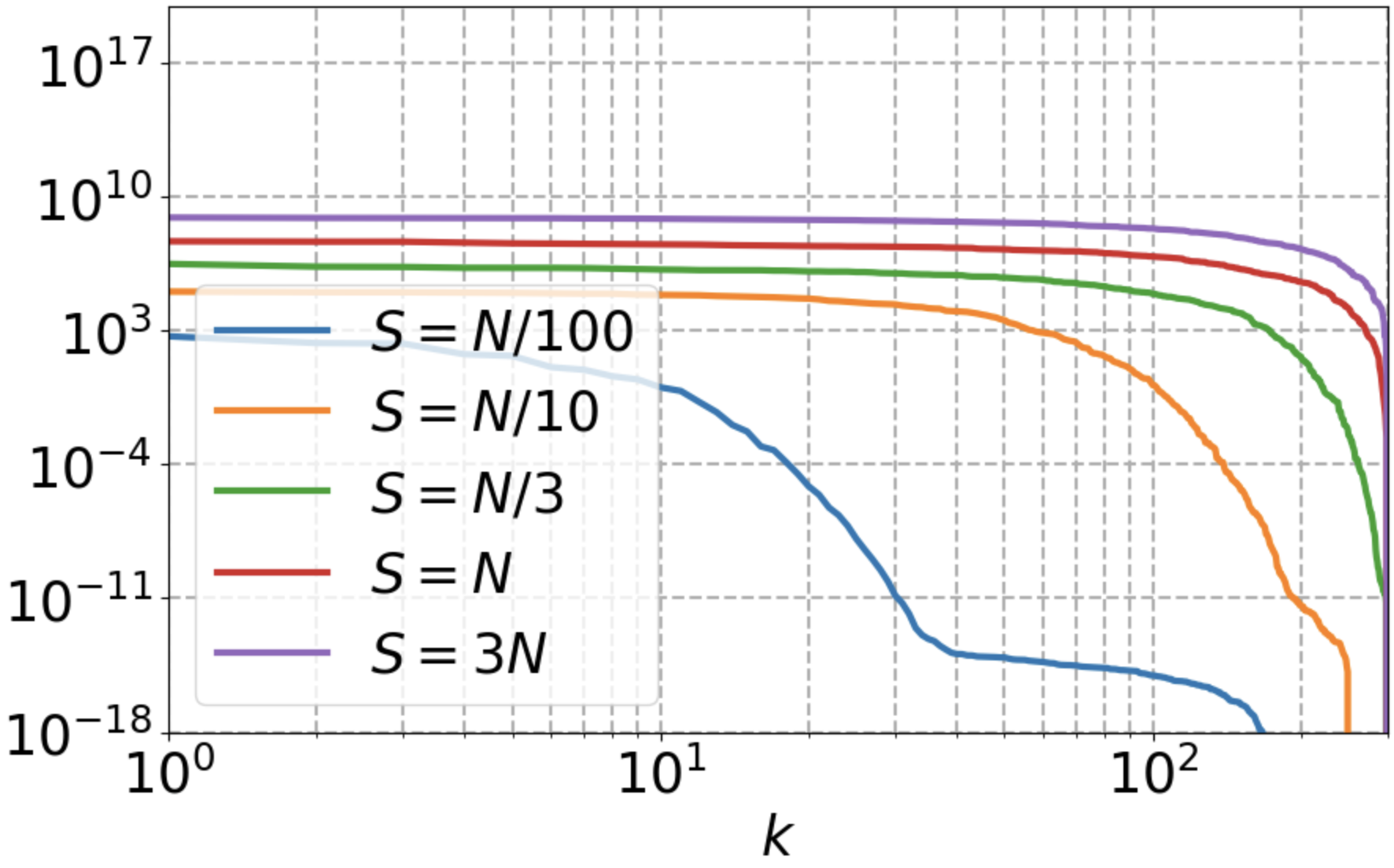}&
\includegraphics[width=0.33\textwidth]{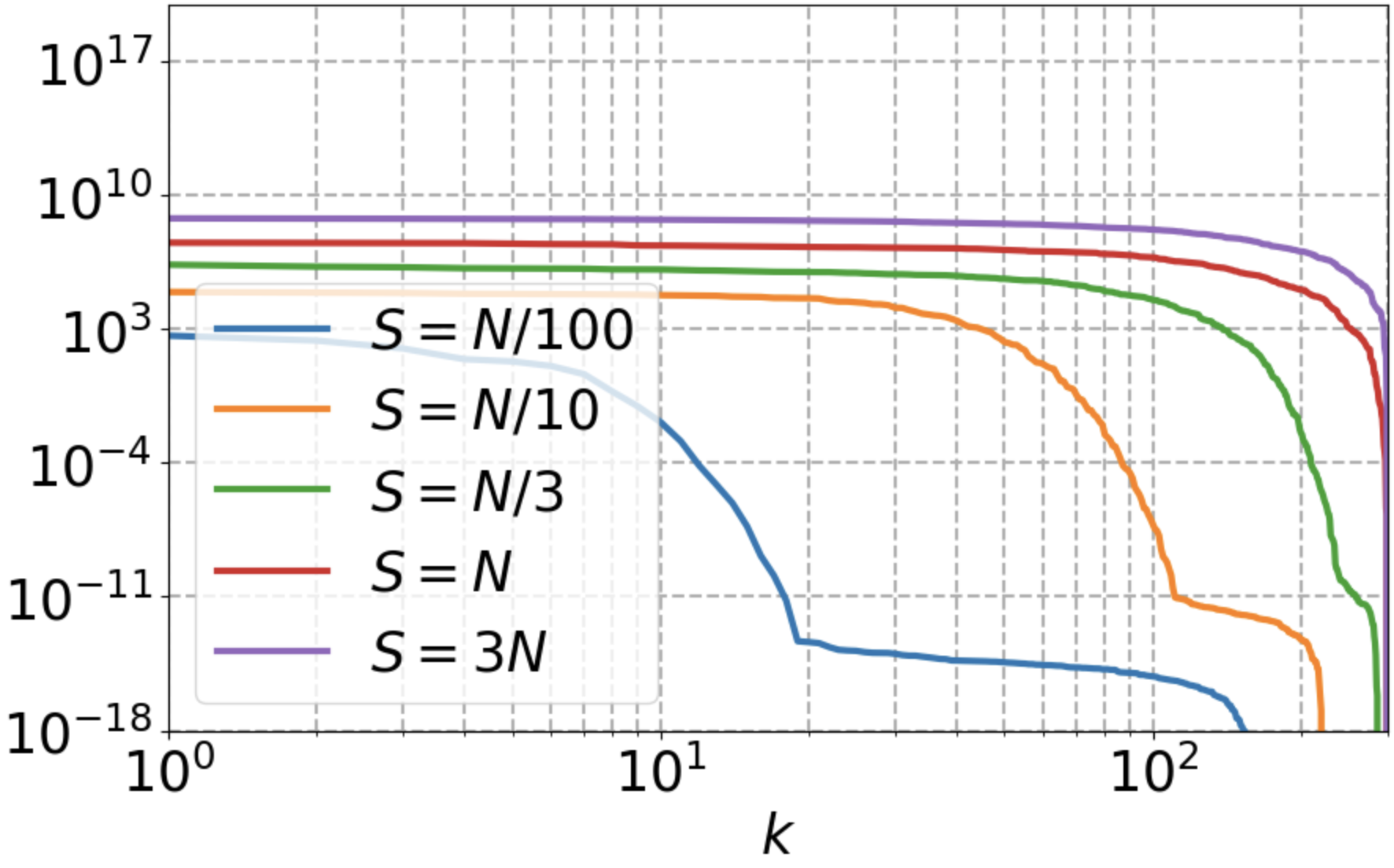}
\end{tabular}
\caption{Spectra of the PINN KKT matrices~\eqref{eq_opt_KKT} associated with bases $\{\sigma(w_i(x-b_i))\}_{i=1}^N$, where $\sigma$ is  (a) $\sin$ activation,  (b) $\tanh$ activation, and (c) GELU activation~\cite{hendrycks2016gaussian}, where $w_i\sim \cU(-S,S)$, $b_i\sim \cU(-1,1)$, $i=1,\dots,N$. The scale $S\in\{N/100, N/10, N/3, N, 3N\}$, the network width $N=300$, and the integrals in the Gram matrices are approximated by a Riemann sum using $10000$ regular grid points. Spectral decays for the DRM Gram matrices are similar.}\label{fig_spectrum_scale}
\end{figure}

Figure~\ref{fig_spectrum_scale} shows the spectra of the PINN KKT matrices~\eqref{eq_opt_KKT} associated with the bases $\{\sigma(w_i(x - b_i)),i=1,\dots, N\}$  where $\sigma$ is the activation function, $w_i\sim\cU(-S,S)$ and $b_i\sim\cU(-1,1)$ are independently sampled, and $S>0$ is a \textit{scaling parameter}. With networks of width $N=300$, in (a), we use the $\sin$ activation function; in (b), we use the $\tanh$ activation function; and in (c), we use GELU~\cite{hendrycks2016gaussian}, a non-homogeneous variant of the ReLU. Due to the scaling, the basis functions are less smooth; that is, they have larger derivatives or higher-frequency components and are therefore less correlated.  We see that Gram matrices with greater values of scaling $S$ exhibit a wider range (proportional to $S$) of significant leading spectrum before it decays fast. It means that initial scaling can reduce ill-conditioning and frequency bias, which is true for other scalable activation functions in general. However, this does not imply that a larger scaling factor  $S$ is always better. As shown in \cite{zhang2025shallow}, the network size needs to be compatible with the scaling, i.e., $S=\mathcal{O}(N)$. Otherwise, if $S$ is too large, the family of $N$ bases may not be able to bridge all immediate scales from $0$ up to $1/S$, or the network size cannot resolve/represent the change or oscillations at scale $1/S$ well. This phenomenon is called \textit{over-scaling}. Next, we show the effect of initial scaling when solving PDEs using SNNs.    

\begin{figure}
\begin{tabular}{c@{\vspace{2pt}}c@{\vspace{2pt}}c}
(a)&(b)&(c)\\
\toprule
\multicolumn{3}{c}{PINN}\\
\includegraphics[width=0.33\textwidth]{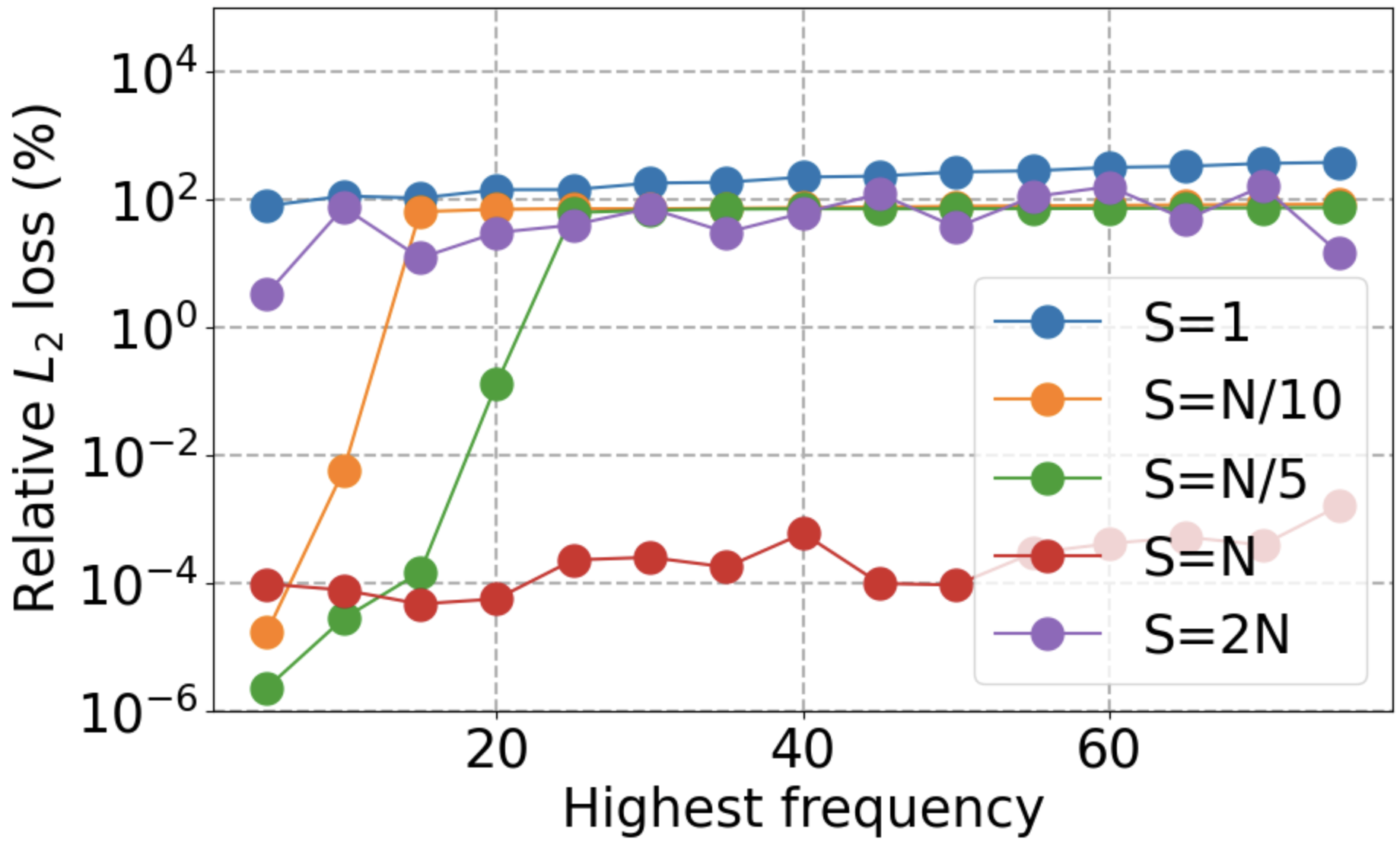}&
\includegraphics[width=0.33\textwidth]{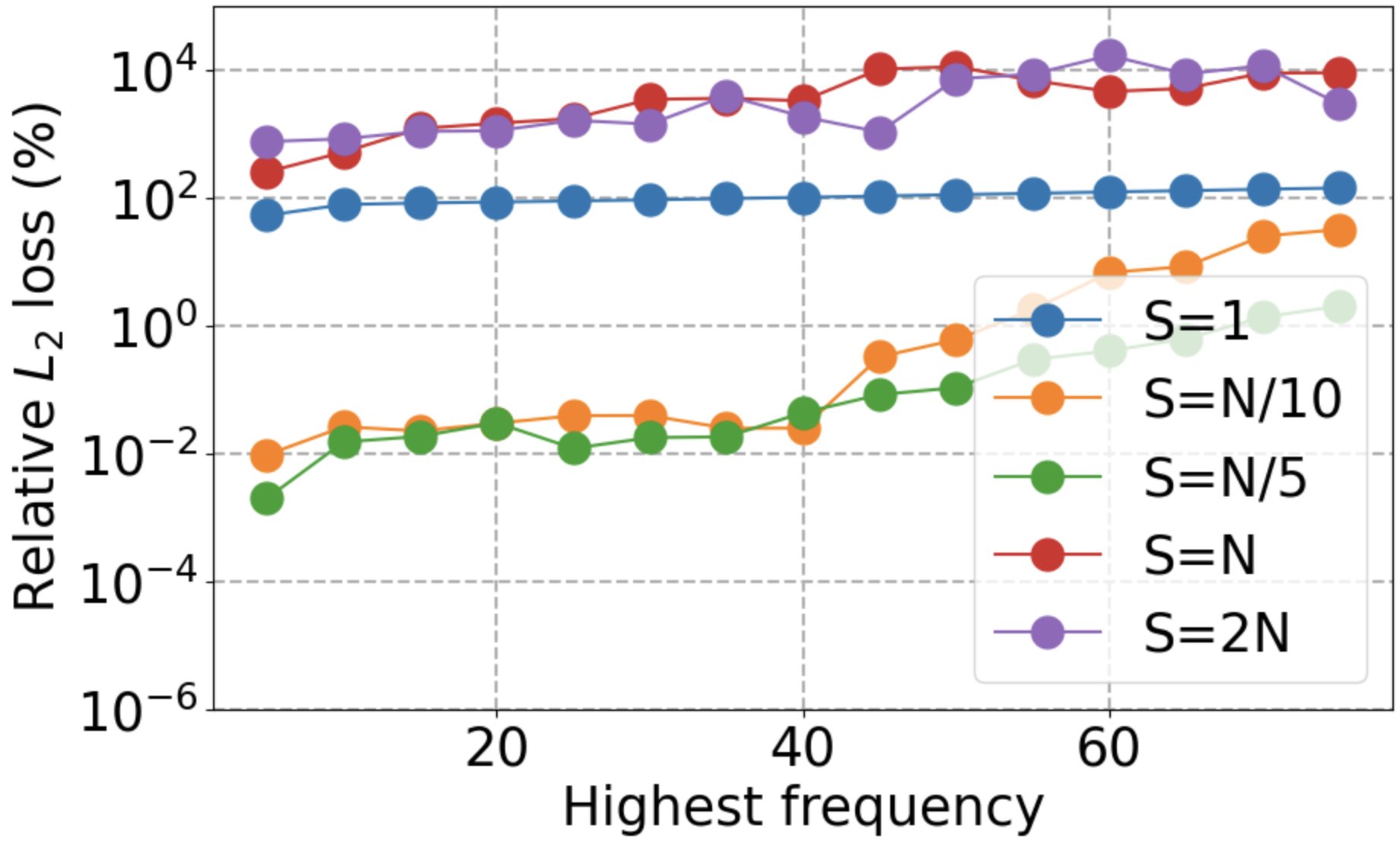}&
\includegraphics[width=0.33\textwidth]{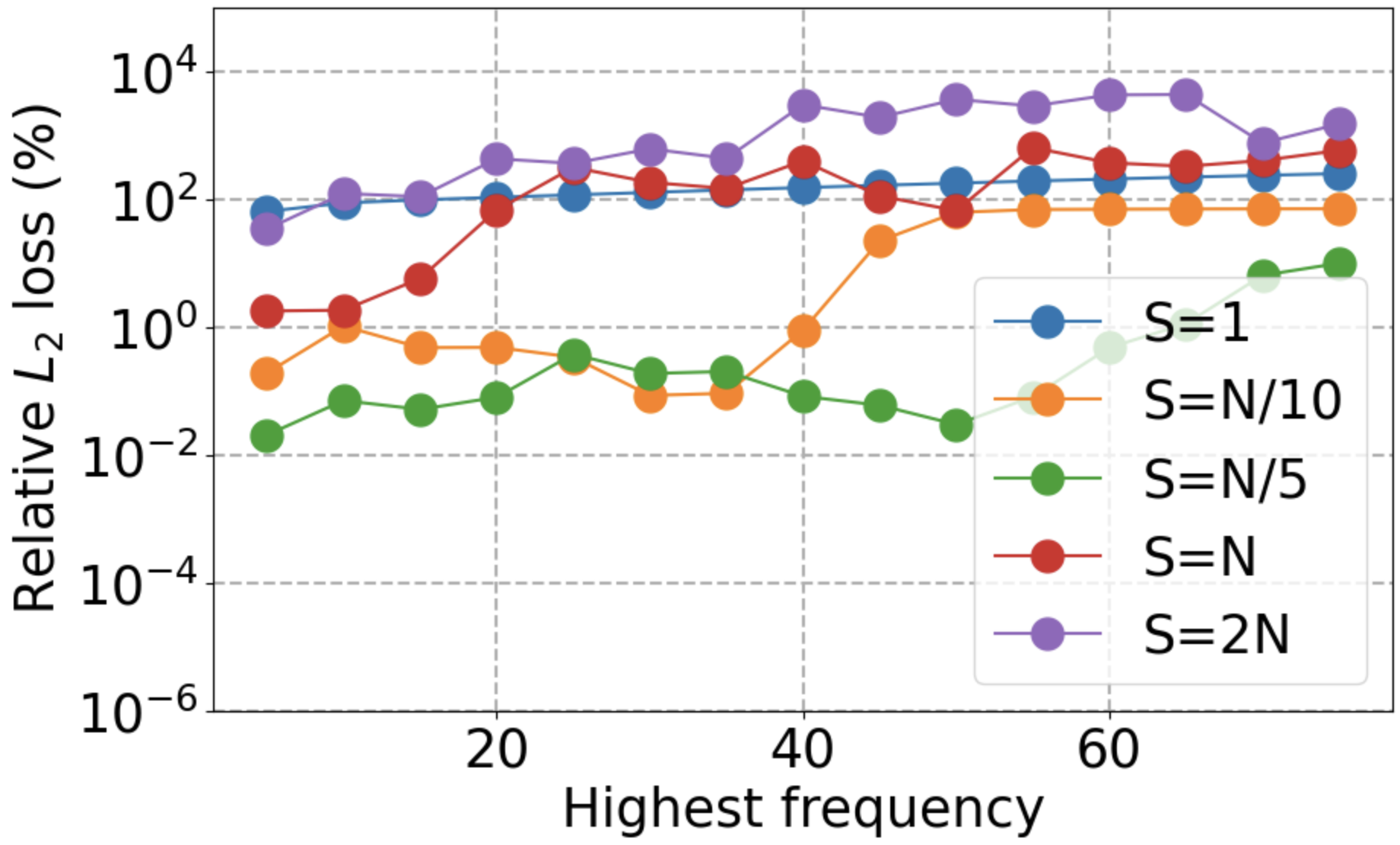}\\
\toprule
\multicolumn{3}{c}{DRM}\\
\includegraphics[width=0.33\textwidth]{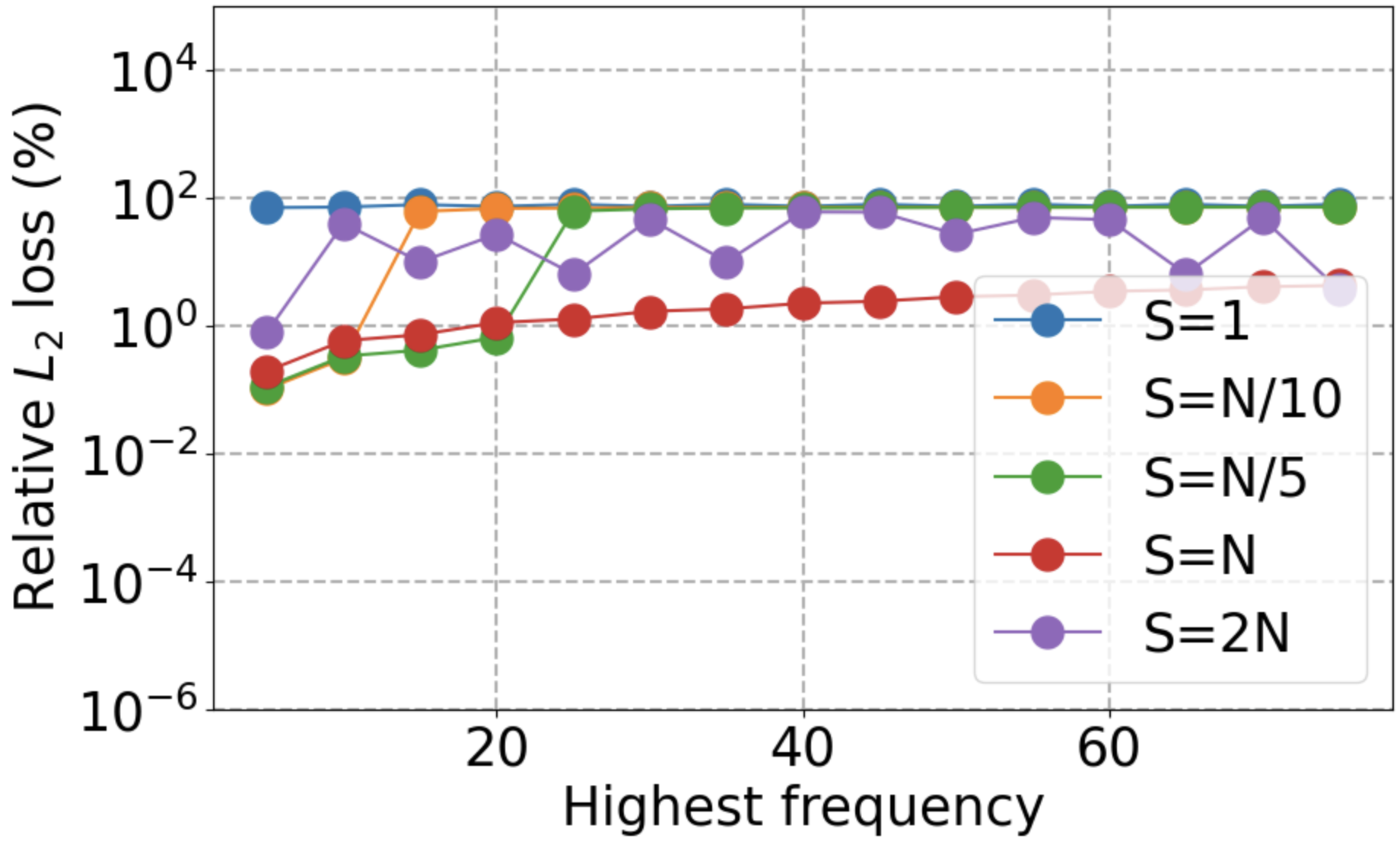}&
\includegraphics[width=0.33\textwidth]{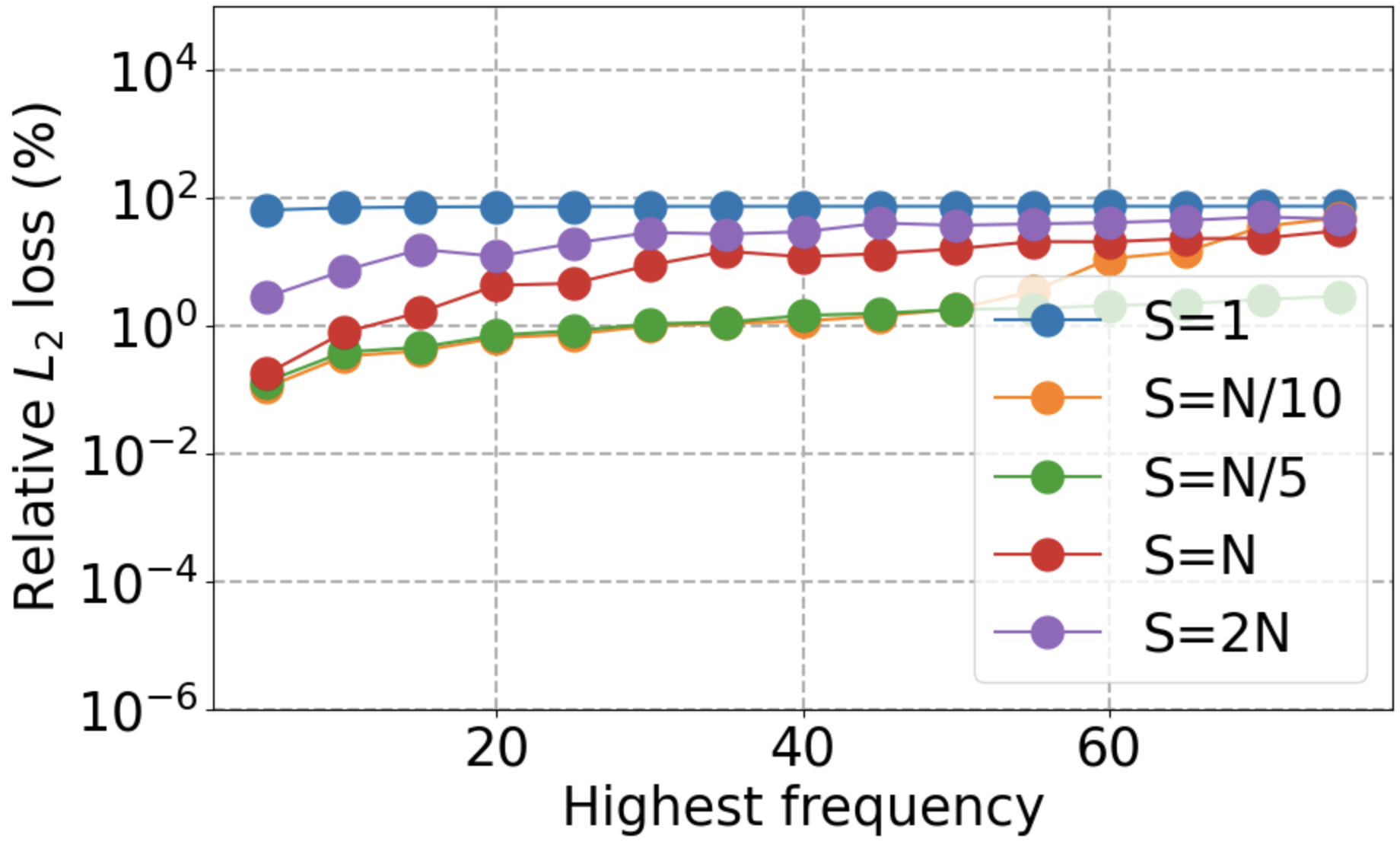}&
\includegraphics[width=0.33\textwidth]{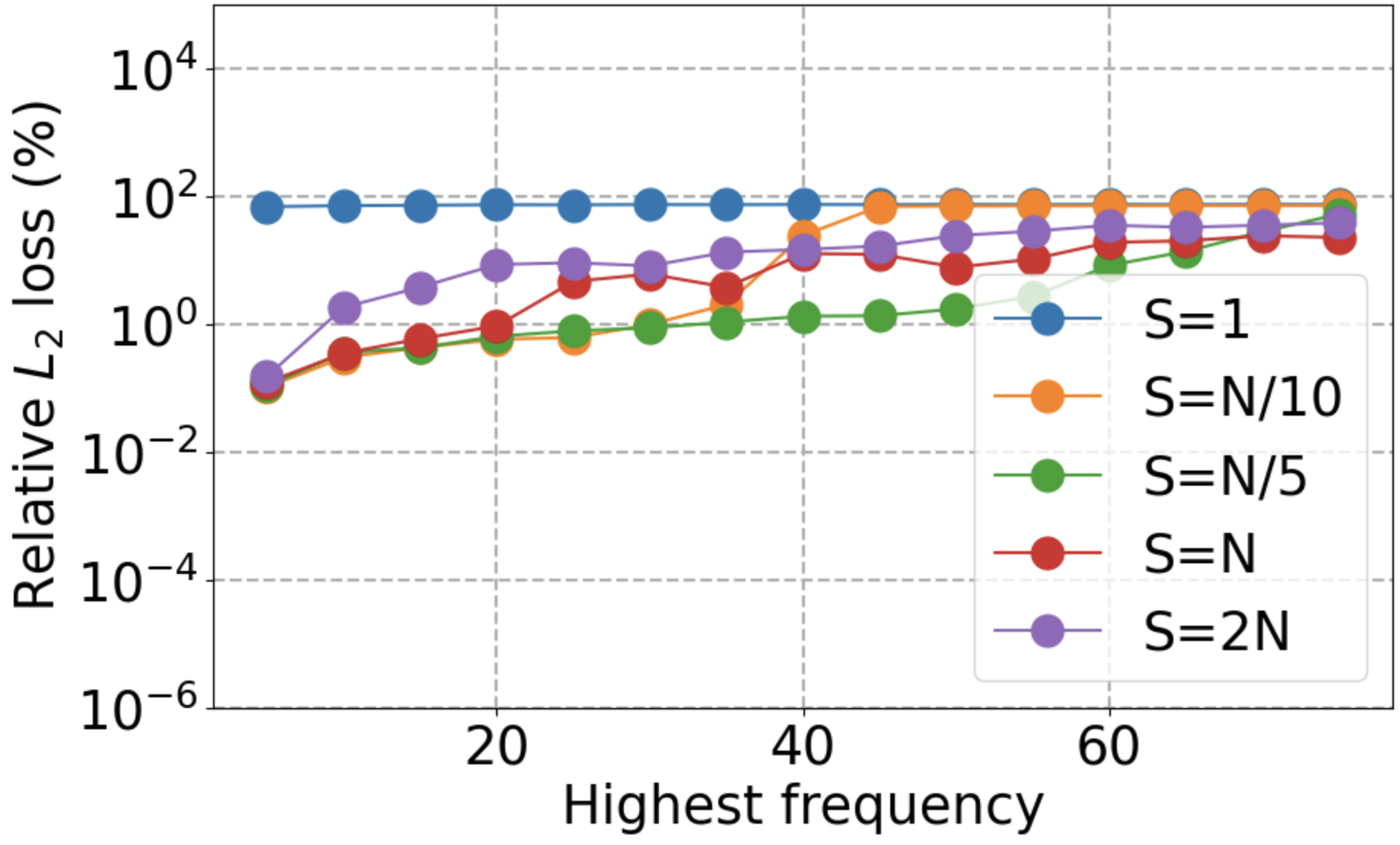}
\end{tabular}
\caption{Approximation errors of PINN and DRM with different scaling and activation:  (a) $\sin$; (b) $\tanh$; and (c) GELU. The network setups are identical to those in Figure~\ref{fig_spectrum_scale}. The boundary conditions are enforced as constraints, and the integrals in the Gram matrices are approximated by a Riemann sum using $10000$ regular grid points. The results are reported as mean values collected from  $20$ independent experiments for each setting.}\label{fig_activation_accuracy}
\end{figure}

In Figure~\ref{fig_activation_accuracy}, we solve the Dirichlet problem for the Poisson equation using linear PINN and DRM model with different activations and scaling. The solution is a superposition of two sine waves with wave numbers $2$ and $k_{\max}$, where $k_{\max}$ varies. We enforce the boundary conditions as constraints and solve a least square problem to find the optimal parameter $\ba$ using the direct method (Section~\ref{sec_constraint_accuracy}). In (a)-(c), we report the relative $L_2$ loss of the solution approximation for $\sin$ activation, $\tanh$ activation, and GELU activation, respectively. The scaling factor $S$ for both weights and biases is a multiple of the network width $N=300$. The results for PINN are shown in the first row, and those for DRM are in the second row.  For all cases, the errors increase as $k_{\max}$ increases, consistent with the performances of SNNs with $\ReLU$ power activations.

For these tests, we see a common phenomenon, which is that the performance improves as the scaling constant $S$, as a multiple of the network width $N$, increases and then degrades when over-scaling happens.  With relatively small scaling, the bases are smooth and lead to a fast spectral decay of the Gram matrix, which results in ill-conditioning and bias against high frequency. Hence, as the solution contains higher-frequency components, the numerical error increases. However, if $S$ becomes too large, it leads to over-scaling of the family of $N$ bases, which will result in poor numerical results.   
Let us analyze the results using sin as the activation function. The linear setting is equivalent to using a set of $N$ Fourier bases with frequency randomly sampled in $[0,S]$ to approximate the solution. 
Once $S$ is above $k_{\max}\pi\approx Nk_{\max}/100$, the space spanned by $\sin$ activation function scaled with $S$ can represent the frequency component of  $k_{\max}$ on an average sense. This explains that the error curves in (a) suddenly increase around $k_{\max}= 30$ for $S=N/3=100$ and  $k_{\max}=50$ for $S=N/2 = 150$. When $S=N=300$, the highest frequency it can cover is around $100$, which explains why the corresponding error curve (the red one) consistently remains low. However,  as $S$ keeps increasing, there exist larger gaps in the randomly sampled frequency. There are certain frequencies, maybe even low frequencies, that cannot be approximated well by the sampled $N$ frequencies. This explains the unstable performances in (a) for $S=2N$. As for the $\tanh$ and GELU activations, we note that $S$ also needs to be appropriately tuned. In particular, with too small $S$, high frequency components cannot be well approximated; and with too large $S$, the distribution of $N$ biases (similar to a grid) in the interval may not be able to resolve/represent the fast changes for those activation functions with large gradients due to over-scaling. 
The specific threshold seems to be related to the maximum of the first derivatives of the activation functions. 

\begin{remark}
For two-layer NNs in the linear setting, it is the same as random feature methods studied in \cite{chen2022bridging, chen2024optimization}. It shows that appropriate scaling can be important. It is equivalent to providing a set of $N$ (network width) diverse and balanced bases that can cover a continuous range of scales as wide as possible. However, using a large set of random global basis will result in a large and dense Gram matrix that is computationally costly to construct or perform matrix operations compared to a well-structured basis, such as a finite element basis.
\end{remark}

\subsection{Scaling and adaptivity}\label{sec_deeper}
In this part, we consider the fully-trainable SNNs, which are nonlinear models~\eqref{eq_two_layer_NN} and the associated optimization problems~\eqref{eq_regularized_optimization} for both PINN and DRM become non-convex and non-linear. The constrained optimization~\eqref{eq_constrained_optimization} is more challenging to solve, thus it is rarely considered in practice. Nevertheless, we mention a recent work~\cite{bao2024wanco} addressing NN training with constraints. 

In Figure~\ref{fig_comparison_pinn_varying}, we compare the solution approximation errors of two models: (1) linear PINN and  (2)  PINN with fixed sampling, i.e., a fixed set of points is used for Monte-Carlo approximation of the loss for each epoch. To illustrate the effects of scaling in more general cases, we consider the following elliptic equation with the Dirichlet boundary condition:
\begin{equation}\label{eq_PDE_problem_varying}
		-\partial_x\left(D(x)\partial_xu\right)  + u=f~\text{in}~(-1,1)
\end{equation}
where $D(x) = \sin(\pi x)+2$ is varying. We set the underlying solution as:
\[
u(x) = 1000 \cdot C(x) \cdot \sum_{i=1}^{3} A_i \cdot \sin\left(\omega_i (x - c_i)^2\right) \cdot \exp\left(-\alpha_i (x - c_i)^2\right)\;,
\]
where the cutoff function $C(x)$ is given by:

\[
C(x) = 
\begin{cases}
\exp\left(-\dfrac{1}{(x + 0.6 + \epsilon)^2 (0.6 - x + \epsilon)^2}\right) & \text{if } -0.6 < x < 0.6 \\
0 & \text{otherwise}
\end{cases}
\]
with $\epsilon = 10^{-12}$, and the bump parameters are:
\[
\begin{array}{llll}
A_1 = 1.0, & \omega_1 = 60, & \alpha_1 = 80, & c_1 = -0.2 \\
A_2 = 0.8, & \omega_2 = 300, & \alpha_2 = 50, & c_2 = 0.0 \\
A_3 = 0.6, & \omega_3 = 40, & \alpha_3 = 60, & c_3 = 0.2
\end{array}
\]
We use SNNs with width $N=100$ and $\sin$ activation function: $\sum_{n=1}^N a_n\sin(w_n (x -b_n))$ where $w_n\sim \cU(-S, S)$ and $b_n\sim \cU(-1,1)$ are used for defining the linear PINN and initializing PINN. Here $S$ is the scaling parameter, which we vary. For the linear PINN, we impose the boundary condition as a constraint, while for the others, we use boundary regularization with weight $\lambda = 250$. For the linear PINN, we evaluate the Gram system~\eqref{eq_opt_KKT} with exact integrals, and for non-linear PINN models, we use the Adam optimizer with a fixed learning rate $8\times 10^{-4}$. Based on $20$ independent experiments with random initializations, we report the mean relative $L_2$ loss in (a) and show approximation results in (b)-(d) with different scaling. 

\begin{figure}
\centering
	\begin{tabular}{cc}
		(a)&(b)\\
\includegraphics[width=0.45\textwidth]{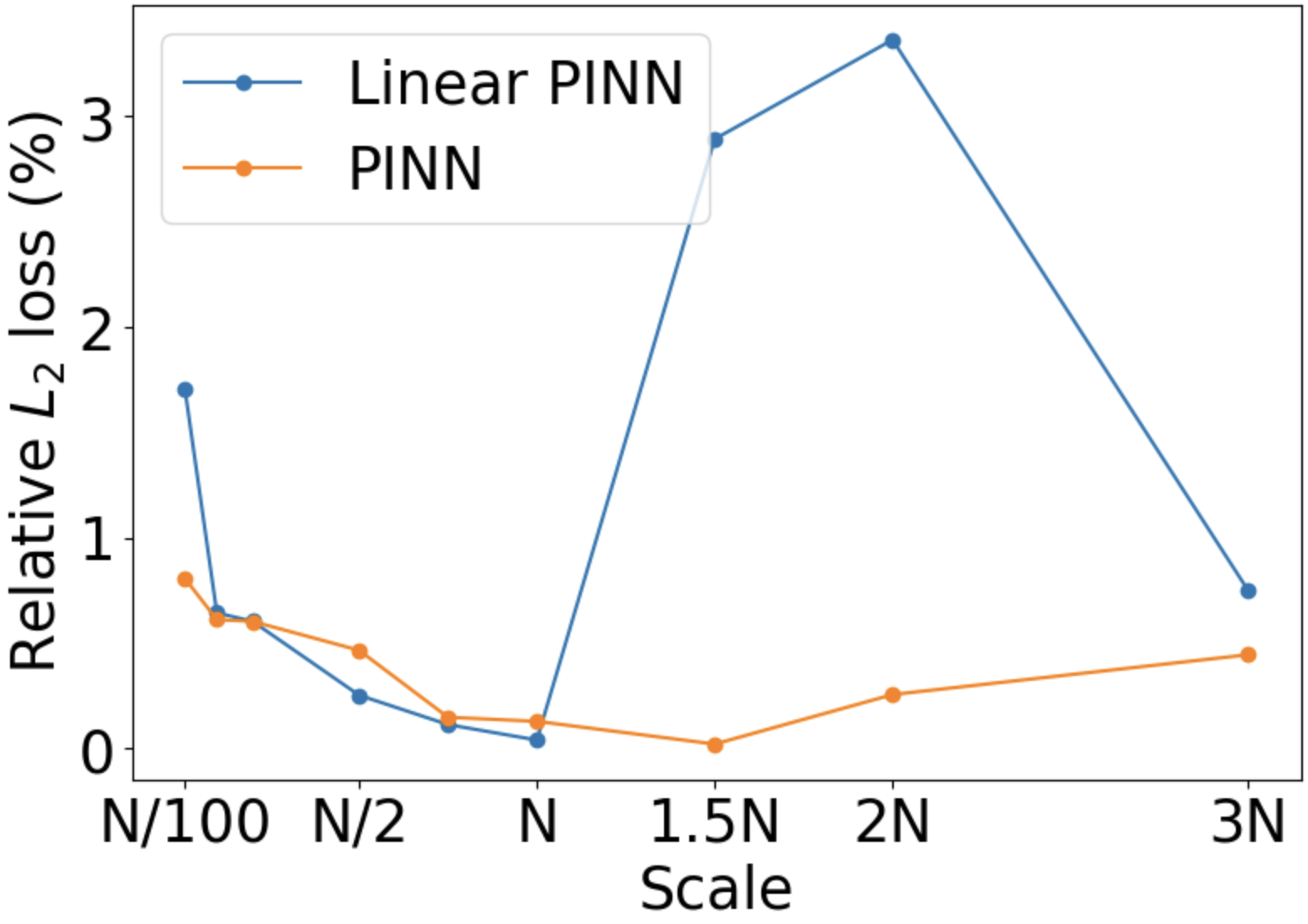}&
\includegraphics[width=0.45\textwidth]{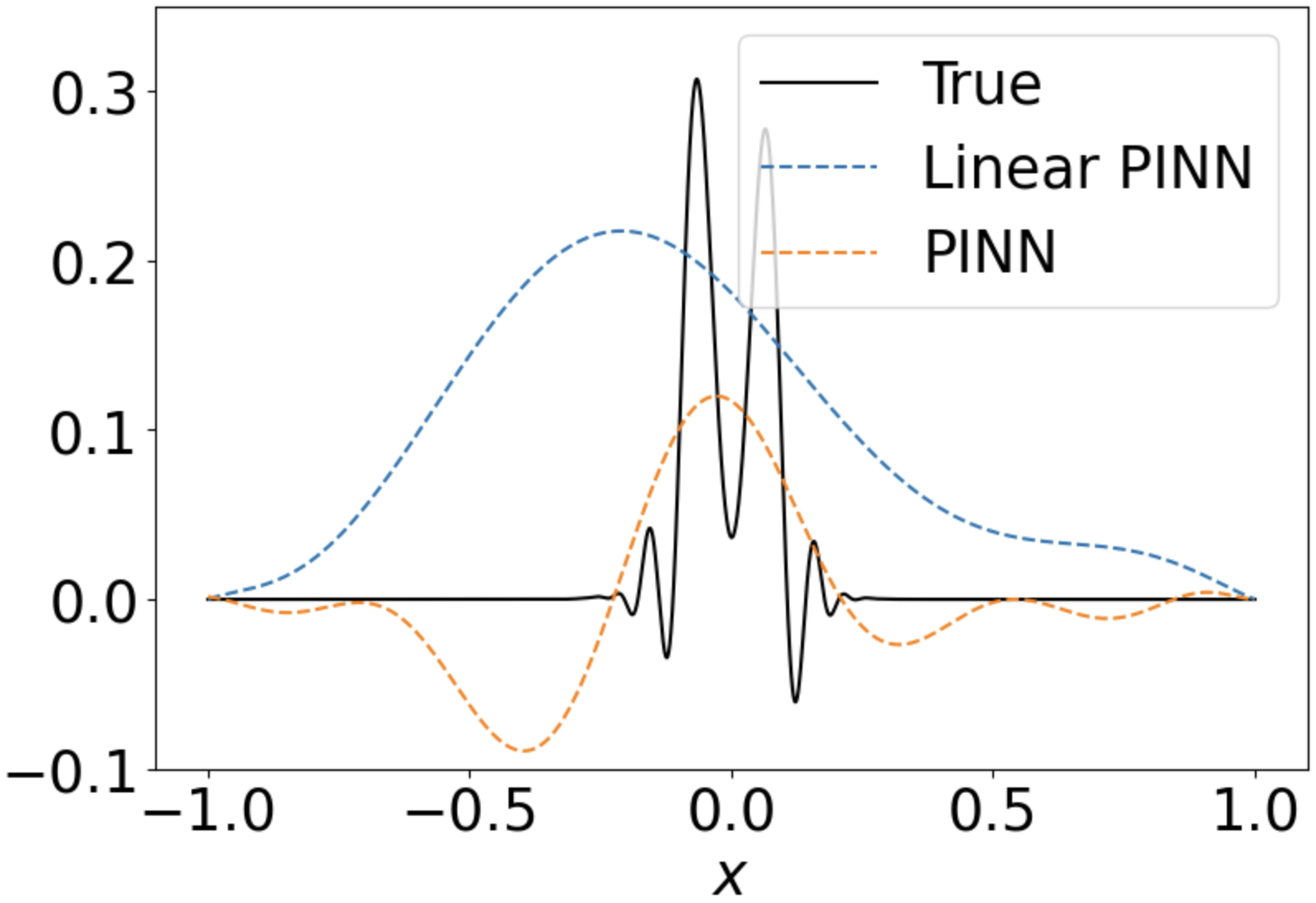}\\
(c)&(d)\\
\includegraphics[width=0.45\textwidth]{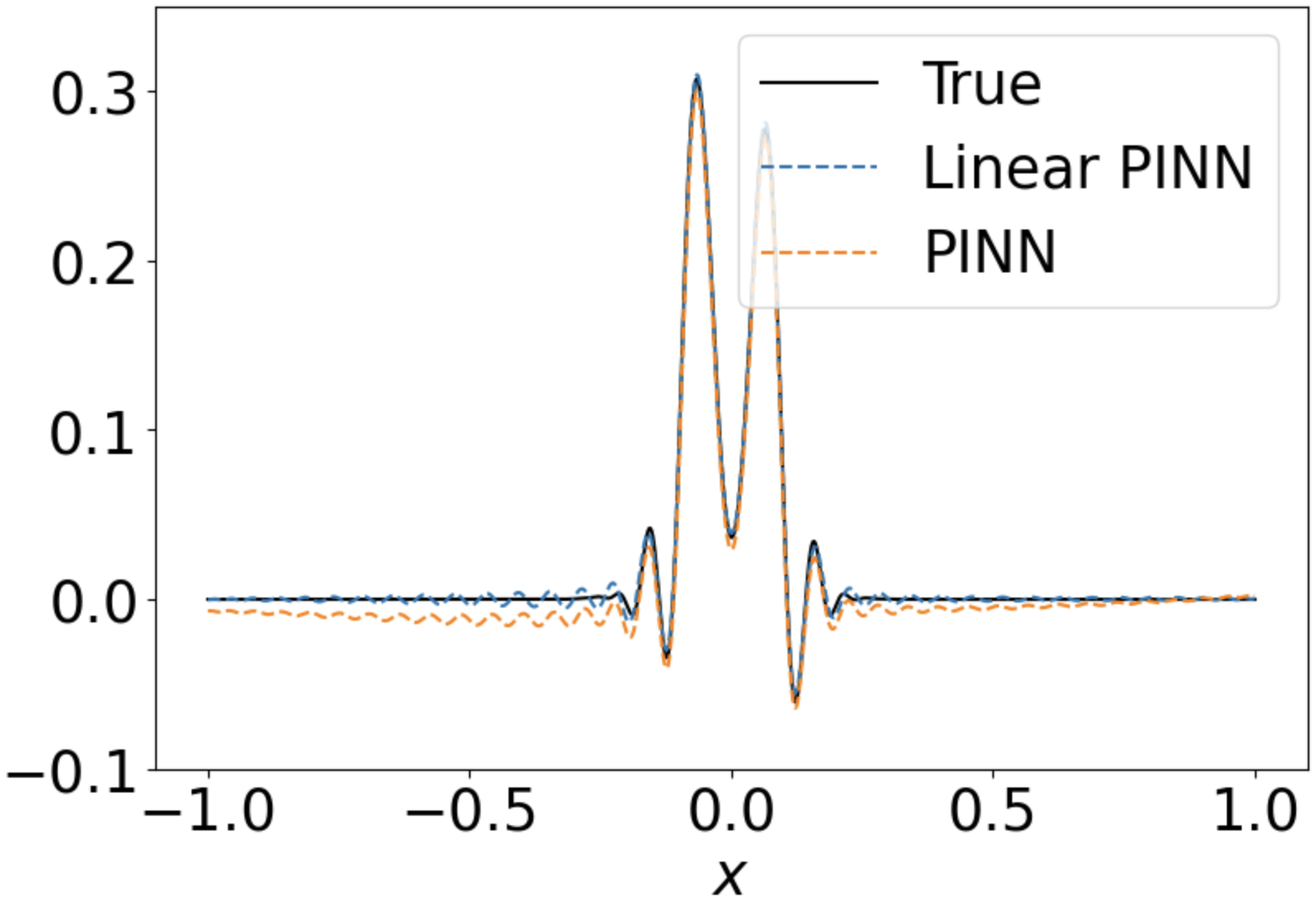}&
\includegraphics[width=0.45\textwidth]{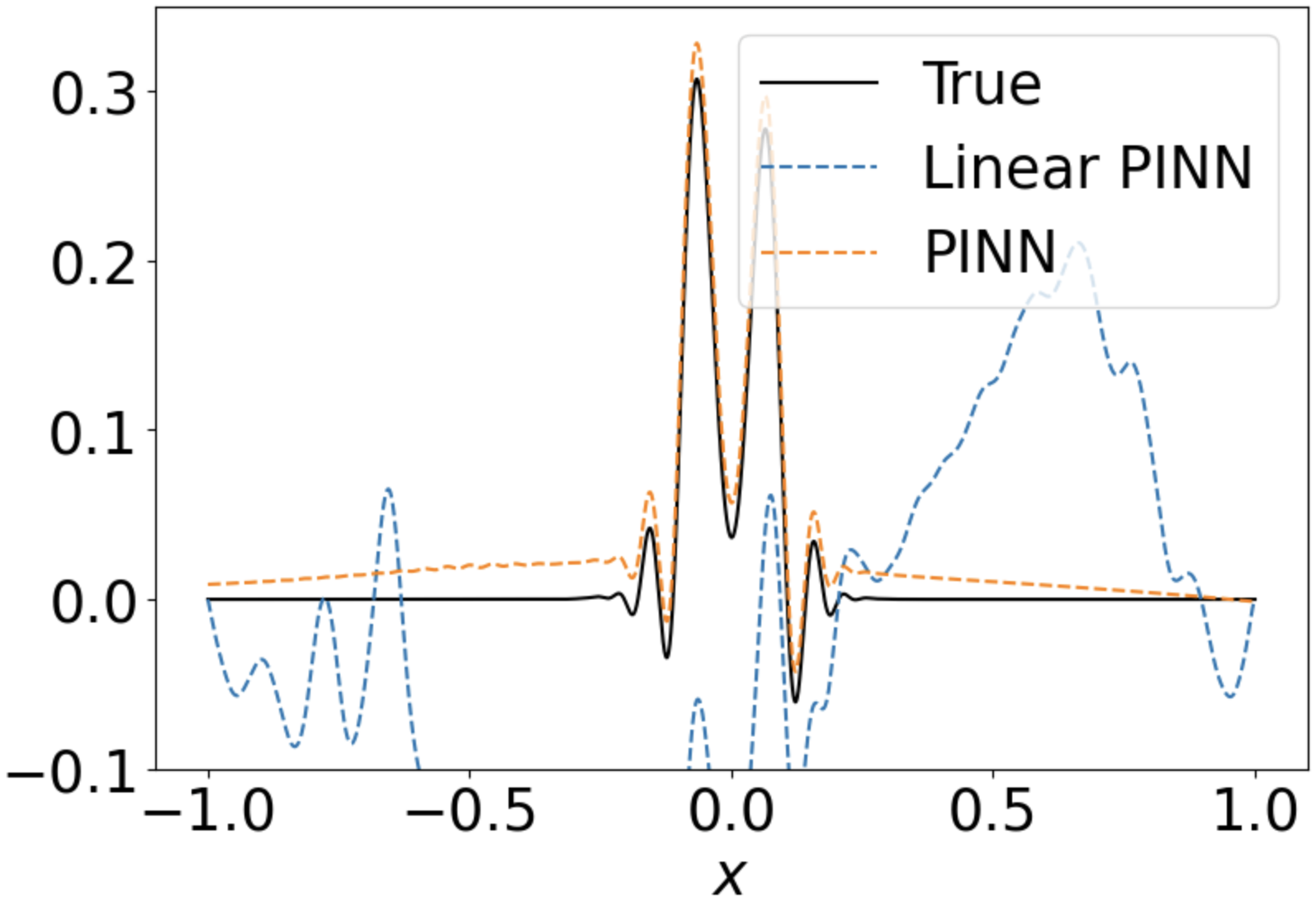}
\end{tabular}
\caption{(Varying coefficients) Comparison under different scaling ($S$) between the linear PINN with exact integral and fully trainable PINN with fixed sampling (each epoch, the same $500$ points are used for the training) (a) Relative $L_2$ loss of the approximations as $S$ increases. (b) Solution approximation when $S=N/100$, (c) $S=N$, and (d) $S=2N$.  All the networks are initialized or fixed  as $\sum_{n=1}^N a_n\sin(w_n (x -b_n))$ where $w_n\sim \cU(-S, S)$ and $b_n\sim \cU(-1,1)$. For PINN, the number of epochs is set to be $20000$, the learning rate is fixed at $8\times 10^{-4}$, and $\lambda=250$ is chosen as the boundary regularization parameter.}\label{fig_comparison_pinn_varying}
\end{figure}

The main observations are: 
\begin{enumerate}
\item When using linear PINNs with non-homogeneous activation functions and a direct solver,  scaling makes a difference. When using a random basis with appropriate scaling, which depends on the network width and activation function as discussed in the previous section, as long as the set of basis spans a space that can approximate the solution well, the performance is as good as expected. When under-scaled in the initial parametrization, due to the lack of diversity of the basis in frequencies, the representation is ill-conditioned and has a strong frequency bias; it can not cover a broad range of frequencies accurately and stably. When over-scaled, the set of random bases cannot cover a continuous range of scales and creates gaps in between. In both scenarios, the performance can be unsatisfactory. For linear representation, due to the lack of adaptivity, the curse of dimensionality cannot be avoided. 
\item 
For fully trainable SSNs, initial scaling also makes a difference. More importantly, when a gradient-descent-based algorithm is used to solve the nonlinear and non-convex optimization problem, the performance is sensitive to ill-conditioning. If the initial parametrization is under-scaled, due to the strong ill-conditioning in the representation, the optimization is not effective in achieving good accuracy. When the scaling is increased, the SNN representation becomes less ill-conditioned and biased, which makes the optimization more effective and can achieve better performance. An interesting phenomenon is that even in the over-scaling regime for linear representation, where the initial random basis can not span a linear space to approximate the solution well, the nonlinear representation with fully trainable parameters shows adaptivity, i.e., the initial set of random basis is updated adaptively to the solution and produces a satisfactory result. This shows ill-conditioning is a key difficulty in achieving effective adaptivity when using gradient-based optimization methods in practice. 
\item  Although the best performance of linear representation, i.e., linear PINN, and nonlinear representation with fully trainable parameters is comparable when using networks of the same size, the nonlinear representation requires much more computation cost vs linear representation due to three times as many parameters and a nonconvex optimization. However, compared to a well-structured FEM, linear PINN using a random basis is much more computationally costly due to the requirement to form the dense Gram matrix and then solve an ill-conditioned dense system, which lacks an effective preconditioner. 
\end{enumerate}

These experiments demonstrate that scaling is crucial for mitigating the spectral bias inherent in neural networks. We emphasize that the linear PINN, as a random feature model, can be preferable to its fully trainable SNN counterpart when an appropriate scaling factor and activation function are used. Its key advantage is computational efficiency: finding a solution only requires solving a linear system, which is far more efficient than a gradient-descent-based optimization needed for non-convex, nonlinear problems.

\section{Discussion and Conclusion }\label{sec_conclude}

This paper presents a systematic study on the behaviors expected when solving elliptic PDEs using shallow neural networks (SNNs), or two-layer networks. For both PINN and DRM, our analysis and experiments show that ill-conditioning and the frequency bias \textit{against} high frequencies inherited from SNNs with ReLU power activations dominate the spectral bias \textit{for} high frequencies induced by the underlying differential operator. 

Our mathematical and numerical study identifies two causes of failure for linear PINNs and DRM with power ReLU activations: (1) ill-conditioned Gram matrices that amplify machine rounding errors, which worsen with smoother activations and wider networks; and (2) the frequency bias inherited from the NN representations, which hinders the approximation of high-frequency components.

We then show that \textit{proper} scaling of activation functions (other than powers of ReLU) mitigates ill-conditioning and frequency bias, which can speed up training and enhance performance. However, using proper scaling is equivalent to providing an initial set of random basis functions that can cover a continuous scale/frequency range compatible with the network width. Since scaling can only stretch the leading spectral range, leaving the rest of the spectrum to decay rapidly, the SNN representation is therefore still highly ill-conditioned and biased against the high frequencies beyond the initial random bases. Hence, achieving adaptivity beyond the initial SNN's representation capability is ineffective with gradient-descent-based optimization.  

This leads us to conclude that it is difficult and costly to gain a crucial adaptive advantage from nonlinear or fully trainable SNNs. Furthermore,  typical global activation functions are random bases that produce dense and ill-conditioned Gram matrices. Without an effective preconditioner, which itself is an interesting open problem, using a random basis will not be nearly as efficient as using well-developed FEMs. FEMs employ structured local bases that lead to sparse systems, for which effective preconditioners are readily available. 

The next important and challenging question is: what do we expect if multi-layer NNs are used to represent the solution? We will report our findings in our future work.

\section*{Acknowledgement}
The research of Roy Y. He is partially supported by NSFC grant 12501594, PROCORE-France/Hong Kong Joint Research Scheme by the RGC of Hong Kong and the Consulate General of France in Hong Kong (F-CityU101/24), StUp - CityU 7200779 from City University of Hong Kong, and the Hong Kong Research Grant
Council ECS grant 21309625. Hongkai Zhao is partially supported by the National Science Foundation through grant DMS-2309551. Yimin Zhong is partially supported by the National Science Foundation through grant DMS-2309530.

\appendix
\section{Proofs}
Here we collect proofs for some statements in this paper.
\subsection{Proof for Lemma~\ref{lemma_vanishing_derivative}}\label{proof_lemma_vanishing_derivative}
\begin{proof}
First, by taking $(p-1)$-times partial derivatives with respect to $x$, we have
\begin{equation*}
\frac{\partial^{p-1}}{\partial x^{p-1}}\cG_p(x,y) = (-1)^{p-1}p!\int_{-1}^1\ReLU( z-x)\ReLU^p( z-y)\,dz\;.
\end{equation*}
Then, taking additional two times partial derivatives with respect to $x$ gives
\begin{equation*}
\frac{\partial^{p+1}}{\partial x^{p+1}}\cG_p(x,y) =  (-1)^{p-1}p! \ReLU^p(x-y)\;.
\end{equation*}
Finally, taking another $(p+1)$-times partial derivatives gives
the desired equation.
\end{proof}
\subsection{Proof for Theorem~\ref{theorem_condition_number1}}\label{proof_theorem_condition_number1}
\begin{proof} The proof is similar to that of Theorem 2 in~\cite{zhang2025shallow}. To simplify the notations, we denote $\bG:=\bG_{\sigma}^{(0)}$ with $\sigma=\ReLU^p$ for some fixed integer $p$. Consider the decomposition
\begin{equation}
\bG = \bA_M+ \boldsymbol{R}_M
\end{equation}
where $\bA_M:=\boldsymbol{\Phi}_M\boldsymbol{\Sigma}_M\boldsymbol{\Phi}_M^\top$, $[\boldsymbol{\Phi}]_{i,j}= \varphi_{j,p}(b_i)$, $\boldsymbol{\Sigma}_{i,j}=\delta_{i,j}\mu_{i,p}$, and $[\boldsymbol{R}_{M}]_{i,j} = \sum_{m>M}\mu_{m,p}\varphi_{m,p}(b_i)\varphi_{m,p}(b_j)$, with $1\leq i,j\leq M$. Denote $\alpha_i$ the $i$-th eigenvalue of $\bA_M$. By Theorem 4.5.9 of~\cite{horn2012matrix}, 
$\alpha_i = \theta_i \mu_{i,p}$ for some $\theta_i\in [\sigma_{\min}^2,\sigma_{\max}^2]$ where $\sigma_{\min}$ and $\sigma_{\max}$ are the minimum and maximum singular values of $\boldsymbol{\Phi}_M$, thus we have
\begin{equation}
|\alpha_i-\frac{N}{2}\mu_{i,p}|= |\mu_{i,p}||\theta_i-\frac{N}{2}|\leq |\mu_{i,p}|\|\boldsymbol{\Phi}_M\boldsymbol{\Phi}_M^\top-\frac{N}{2}\bI_M\|_2\;.
\end{equation}
Combining with Weyl's inequality, we obtain
\begin{equation}
|\lambda_i-\frac{N}{2}\mu_{i,p}|\leq |\lambda_i-\alpha_i| + |\alpha_i - \frac{N}{2}\mu_{i,p}|\leq \|\bR_M\|_2+|\mu_{i,p}|\|\boldsymbol{\Phi}_M\boldsymbol{\Phi}_M^\top-\frac{N}{2}\bI_M\|_2\;.
\end{equation}
Note that $\|\bR_M\|_2=\cO(NM^{-(2p+1)}(2p+1)^{-1})$ and 
$\|\boldsymbol{\Phi}_M\boldsymbol{\Phi}_M^\top-\frac{N}{2}\bI_M\|_2=\cO(M^2)$; 
hence,
\begin{equation}
|\lambda_i-\frac{N}{2}\mu_{i,p}|\leq C\min\left(1, \frac{N}{(2p+1)M^{2p+1}}+\frac{M^2}{i^{2p+2}}\right)\;.
\end{equation}
Note that the left hand side of the above inequality does not depend on $M$, thus we can take $M = (Ni^{2p+2}/2)^{1/(2p+3)}$ and deduce that 
\begin{equation}
|\lambda_i-\frac{N}{2}\mu_{i,p}|=\begin{cases} 
\cO\left(\frac{2p+3}{2p+1} N^{\frac{2}{2p+3}} i^{-\frac{2(p+1)(2p+1)}{2p+3}}\right)&\text{if}~ q_{N,p}\leq i \leq N\\
\cO(1)&\text{if}~ 1\leq i < q_{N,p} 
\end{cases}
\end{equation}
where $q_{N,p} = \left(\frac{2p+3}{2p+1}\right)^{\frac{2p+3}{2(p+1)(2p+1)}} N^{\frac{1}{(p+1)(2p+1)}}$. This implies that as $N\to+\infty$, we have $\lambda_N = \cO(\frac{2p+3}{2p+1}N^{-2p})$ and $\lambda_1=\Theta(N)$; thus we conclude that the condition number $\kappa(\bG) =\Omega(N^{1+2p})$.
\end{proof}

\subsection{Proof for Theorem~\ref{theorem_sharp_eigenvalue}}\label{proof_theorem_sharp_eigenvalue}
\begin{proof}
Let us focus on the case when $k=0$, as the general cases are trivial; thereby, we omit the superscript $k$ in the following proof. By the relation~\eqref{eq_relate_G_sigma_G_FEM} and Theorem 6 in~\cite{hong2022activation}, the following
\begin{equation}
\frac{\lambda_{\min}(\bG_{\FEMsp})}{\lambda_{j}(\bW\bW^\top)}\leq \lambda_{N+1-j}\leq  \frac{\lambda_{\max}(\bG_{\FEMsp})}{\lambda_{j}(\bW\bW^\top)}
\end{equation}
holds for  $j=1,2,\dots,N$. Note that $\bW$ is a banded lower triangular Toeplitz matrix with a generating function $f(z)= (1-z)^{p+1}/(p!(\Delta x)^p)$. By the Avram–Parter Theorem~\cite{bottcher2005spectral}, as $N$ is sufficiently large, the singular values of $\bW$ are distributed as $f(e^{i\theta})$ for $\theta\in [0,2\pi)$, and this implies that 
$\lambda_j(\bW\bW^\top)\sim j^{-2(p+1)}$ as $N\to\infty$, and $\lambda_{\min}(\bG_{\FEMsp}) \sim 1$, $\lambda_{\max}(\bG_{\FEMsp}) \sim 1$ by the Riez-basis property (See~\cite{unser2005cardinal} Theorem 1). Therefore,  statement is proved.
\end{proof}

\subsection{Proof for Lemma~\ref{lemma_eigenvector_KF}}\label{proof_lemma_eigenvector_KF}
\begin{proof}
Let $\bx$ be an eigenvector of $\bG_F$ corresponding to eigenvalue $\lambda>0$. We prove the lemma by solving $\bb$  from 
\begin{equation}
\begin{cases}
\bG_F\bx +\bB^\top\bb = \mu\bx\\
\bB\bx = \mu\bb
\end{cases}
\end{equation}
for some nonzero $\mu\in\mathbb{R}$. From the first equation, we deduce $\bB^\top\bb = (\mu-\lambda)\bx$, and from the second equation, we get $\bb = \mu^{-1}\bB\bx$; hence, $\bB^\top\bB\bx = \mu(\mu-\lambda)\bx$.  We know that either $\bx\in\text{ran}(\bB^\top)$ or $\bx\in\ker(\bB)$. Since $\text{rank}(\bB)=2$,  $\text{ran}(\bB^\top)$ can contain at most two eigenvectors of $\bG_F$. Therefore, at least $N-2$ eigenvectors of $\bB_F$ are in $\text{ker}(\bB)$, which immediately implies that $\widetilde{\bx}^\top =[\bx^\top, \mathbf{0}^\top]$ are eigenvectors of $\bK_F$ with eigenvalue $\lambda$.
\end{proof}

\subsection{Proof for Lemma~\ref{corollary_trunc_SVD_error}}\label{proof_corollary_trunc_SVD_error}
\begin{proof} By Lemma~\ref{lemma_eigenvector_KF}, there are at least $N-2$ indices $i$ such that the $i$-th column of $\bU$ can be written as $[\bu_i^\top, 0, 0]^\top\in\mathbb{R}^{N+2}$, and $\bu_i$ is an eigenvector of $\bG_F$. Denote the set of such indices as $\cI_1$. Therefore,
\begin{align*}
\|\widetilde{\bz}_F^* - \bz_F^*\|^2 &=  \sum_{i:|\lambda_i(\bK_F)|<\varepsilon\lambda_1(\bK_F)} \lambda_i(\bK_F)^{-2} (\bu_i^\top\by_F+a_ic_L+b_ic_R)^2\\
&\geq\sum_{\substack{i\in\cI_1\\i:|\lambda_i(\bK_F)|<\varepsilon\lambda_1(\bK_F) }} \lambda_i(\bK_F)^{-2} (\bu_i^\top\by_F)^2 
\end{align*}
\end{proof}

\section{Explicit integral formulas for $\ReLU$ power Gram systems}\label{sec_explicit_formula}
Here we show explicit formulas for computing  entries of the Gram system~\eqref{eq_Gram_matrix} when $\sigma$ is $\ReLU^p$ for $p=1,\dots,4$.
\begin{itemize}
\item For $\ReLU$:  \begin{equation*}
\int (\omega_i x - \beta_i)(\omega_j x - \beta_j)\,dx = \frac{1}{3} \omega_i \omega_j x^3 - \frac{1}{2} \omega_i \beta_j x^2 - \frac{1}{2} \beta_i \omega_j x^2 + \beta_i \beta_j x
\end{equation*}
\item For $\ReLU^2$: \begin{equation*}
\begin{aligned}
\int \left((\omega_i x - \beta_i)(\omega_j x - \beta_j)\right)^2\,dx &=  \frac{1}{5} \omega_i^2 \omega_j^2 x^5 \\
& + \left(-\frac{1}{2} \omega_i^2 \omega_j \beta_j - \frac{1}{2} \omega_i \beta_i \omega_j^2\right) x^4 \\
& + \left(\frac{1}{3} \omega_i^2 \beta_j^2 + \frac{4}{3} \omega_i \beta_i \omega_j \beta_j + \frac{1}{3} \beta_i^2 \omega_j^2\right) x^3 \\
& + \left(-\omega_i \beta_i \beta_j^2 - \beta_i^2 \omega_j \beta_j\right) x^2 \\
& + \beta_i^2 \beta_j^2 x
\end{aligned}
\end{equation*}
\item For $\ReLU^3$: 
\begin{equation*}
\begin{aligned}
\int \left((\omega_i x - \beta_i)(\omega_j x - \beta_j)\right)^3\,dx &= \frac{1}{7} \omega_i^3 \omega_j^3 x^7 - \frac{1}{2} \omega_i^2 \omega_j^2 (\omega_i \beta_j + \beta_i \omega_j) x^6 \\
&+ \frac{3}{5} \omega_i \omega_j (\omega_i^2 \beta_j^2 + 3 \omega_i \beta_i \omega_j \beta_j + \beta_i^2 \omega_j^2) x^5 \\
&+ \frac{1}{4} (-\omega_i^3 \beta_j^3 - 9 \omega_i^2 \beta_i \omega_j \beta_j^2 - 9 \omega_i \beta_i^2 \omega_j^2 \beta_j - \beta_i^3 \omega_j^3) x^4 \\
&+ \beta_i \beta_j (\omega_i^2 \beta_j^2 + 3 \omega_i \beta_i \omega_j \beta_j + \beta_i^2 \omega_j^2) x^3 \\
&- \frac{3}{2} \beta_i^2 \beta_j^2 (\omega_i \beta_j + \beta_i \omega_j) x^2 + \beta_i^3 \beta_j^3 x.
\end{aligned}
\end{equation*}
\item For $\ReLU^4$: 
\begin{equation*}
\begin{aligned}
&\int \left((\omega_i x - \beta_i)(\omega_j x - \beta_j)\right)^4\,dx = \frac{1}{9} \omega_i^4 \omega_j^4 x^9 - \frac{1}{2} \omega_i^3 \omega_j^3 x^8 (\omega_i \beta_j + \beta_i \omega_j)\\
&+\frac{2}{7} \omega_i^2 \omega_j^2 x^7 (3 \omega_i^2 \beta_j^2 + 8 \omega_i \beta_i \omega_j \beta_j + 3 \beta_i^2 \omega_j^2) \\
& + \frac{2}{3} \beta_i^2 \beta_j^2 x^3 (3 \omega_i^2 \beta_j^2 + 8 \omega_i \beta_i \omega_j \beta_j + 3 \beta_i^2 \omega_j^2) - \frac{2}{3} \omega_i \omega_j x^6 (\omega_i^3 \beta_j^3\\
&+ 6 \omega_i^2 \beta_i \omega_j \beta_j^2 + 6 \omega_i \beta_i^2 \omega_j^2 \beta_j + \beta_i^3 \omega_j^3)\\
&- \beta_i \beta_j x^4 (\omega_i^3 \beta_j^3 + 6 \omega_i^2 \beta_i \omega_j \beta_j^2 + 6 \omega_i \beta_i^2 \omega_j^2 \beta_j + \beta_i^3 \omega_j^3)\\
&+ \frac{1}{5} x^5 (\omega_i^4 \beta_j^4 + 16 \omega_i^3 \beta_i \omega_j \beta_j^3 + 36 \omega_i^2 \beta_i^2 \omega_j^2 \beta_j^2 + 16 \omega_i \beta_i^3 \omega_j^3 \beta_j + \beta_i^4 \omega_j^4)\\
&- 2 \beta_i^3 \beta_j^3 x^2 (\omega_i \beta_j + \beta_i \omega_j) + \beta_i^4 \beta_j^4 x.
\end{aligned}
\end{equation*}
\end{itemize}

For the right hand side~\eqref{eq_Gram_vector}, we shall assume that $f$ is of the form
\begin{equation*}
f(x) = \sum_{k=1}^K m_k\sin( a_k x + \varphi_k)\;,
\end{equation*}
and the corresponding integral formulas are
\begin{itemize}
\item For $\ReLU^0$: 
\begin{equation*}
\int \sin(a_kx+\varphi_k)\,dx =-\frac{\cos(a_k x + \varphi_k)}{a_k}
\end{equation*}
\item For $\ReLU$:
\begin{equation*}
\int \sin(a_kx+\varphi_k)(\omega_i x - \beta_i)\,dx = \frac{a_k (\beta_i - \omega_i x) \cos(a_k x + \varphi_k) + \omega_i \sin(a_k x + \varphi_k)}{a_k^2} 
\end{equation*}
\item For $\ReLU^2$:
\begin{equation*}
\begin{aligned}
&\int \sin(a_kx+\varphi_k)(\omega_i x - \beta_i)^2\,dx=\\
&\frac{2 a_k \omega_i (\omega_i x - \beta_i) \sin(a_k x + \varphi_k) - \cos(a_k x + \varphi_k) \left(\omega_i^2 (a_k^2 x^2 - 2) - 2 a_k^2 \omega_i \beta_i x + a_k^2 \beta_i^2\right)}{a_k^3}
\end{aligned}
\end{equation*}
\item For $\ReLU^3$:
\begin{equation*}
\begin{aligned}
&\int \sin(a_kx+\varphi_k)(\omega_i x - \beta_i)^3\,dx = \\
&\cos(a_k x + \varphi_k) \left(\frac{6 \omega_i^2 (\omega_i x - \beta_i)}{a_k^3} + \frac{(\beta_i - \omega_i x) (\omega_i^2 x^2 - 2 \omega_i \beta_i x + \beta_i^2)}{a_k}\right)\\
&+ \sin(a_k x + \varphi_k) \left(\frac{3 \omega_i^3 x^2 - 6 \omega_i^2 \beta_i x + 3 \omega_i \beta_i^2}{a_k^2} - \frac{6 \omega_i^3}{a_k^4}\right)
\end{aligned}
\end{equation*}
\item For $\ReLU^4$:

\begin{equation*}
\begin{aligned}
&\int \sin(a_kx+\varphi_k)(\omega_i x - \beta_i)^4\,dx =\\
&\cos(a_k x + \varphi_k) \Big(-\frac{24 \omega_i^4}{a_k^5} + \frac{12 \omega_i^4 x^2 - 24 \omega_i^3 \beta_i x + 12 \omega_i^2 \beta_i^2}{a_k^3} + \\
&\frac{-\omega_i^4 x^4 + 4 \omega_i^3 \beta_i x^3 - 6 \omega_i^2 \beta_i^2 x^2 + 4 \omega_i \beta_i^3 x - \beta_i^4}{a_k}\Big)
+ \\
&(\omega_i x - \beta_i) \sin(a_k x + \varphi_k) \left(\frac{4 \omega_i (\omega_i^2 x^2 - 2 \omega_i \beta_i x + \beta_i^2)}{a_k^2} - \frac{24 \omega_i^3}{a_k^4}\right)
\end{aligned}
\end{equation*}

\end{itemize}

In the case of spline FEM, we use the following result to evaluate the associated Gram matrices. \begin{lemma}
For any $\Delta x>0$, denote $b_p(t) = B_p(t/\Delta x)$ where $B_p$ is the cardinal B-spline basis of degree $p$ supported on $[0,p+1]$. Then for any $i,j\in\mathbb{Z}$, we have
\begin{equation*}
\int_{\bbR} b_p'(t+i\Delta x)b_p'(t+j\Delta x)\,dt=-(\Delta x)^{-1} B_{2p+1}^{''}(p+1+i-j)\;.
\end{equation*}
\end{lemma}
\begin{proof}
By change of variables, we have
\begin{align*}
&\int_{\bbR} b_p'(t+i\Delta x)b_p'(t+j\Delta x)\,dt\\
&=\int_{\bbR}b_p'(t)b_p'(t+(j-i)\Delta x)\,dt\\
&=(\Delta x)^{-2}\int_{\bbR}B'_p(t/\Delta x)B'_p(t/\Delta x+(j-i))\,dt\\
&=(\Delta x)^{-1}\int_{\bbR}B'_p(t)B'_p(t+(j-i))\,dt\\
&=-(\Delta x)^{-1} B_{2p+1}^{''}(p+1+i-j)
\end{align*}
and the last equality follows from Lemma 4 of~\cite{garoni2014spectrum}.
\end{proof}

\bibliographystyle{plain}
\bibliography{sample}
\end{document}